%% file: main.tex
\documentclass[a4paper]{article}

\usepackage{a4wide}
\usepackage{amsfonts}
\usepackage{float,graphicx}
\usepackage{amsmath}
\usepackage[amsmath,thmmarks,hyperref]{ntheorem}
\usepackage{xcolor, hyperref}

\newcommand{\proofbox}{{\rule{1.5ex}{1.5ex}}}
\theoremstyle{nonumberplain}
\theoremheaderfont{\normalfont\itshape}
\theorembodyfont{\normalfont}
\theoremseparator{.}
\theoremsymbol{\proofbox}
\newtheorem{proof}{Proof}

\theoremstyle{plain}
\theoremheaderfont{\normalfont\sffamily}
\theorembodyfont{\normalfont\itshape}
\theoremseparator{.}
\theoremsymbol{}
\newtheorem{theorem}{Theorem}

\newcommand{\newsiamthm}[2]{
  \theoremstyle{plain}
  \theoremheaderfont{\normalfont\sffamily}
  \theorembodyfont{\normalfont\itshape}
  \theoremseparator{.}
  \theoremsymbol{}
  \newtheorem{#1}[theorem]{#2}
}

\newsiamthm{lemma}{Lemma}
\newsiamthm{corollary}{Corollary}
\newsiamthm{proposition}{Proposition}
\newsiamthm{definition}{Definition}
\newtheorem{remark}[theorem]{Remark}

\title{Rating transitions forecasting: a filtering approach}

\newcommand{\email}[1]{#1}
\author{Areski Cousin\thanks{Institut de Recherche en Math\'ematique Avanc\'ee, Universit\'e de Strasbourg,
7 rue Ren\'e Descartes, 67084 Strasbourg, cedex, France. (\email{a.cousin@unistra.fr}).} 
\and J. Lelong\thanks{Univ. Grenoble Alpes, CNRS, Grenoble INP, LJK, 38000 Grenoble, France. 
  (\email{jerome.lelong@univ-grenoble-alpes.fr}).}
\and T. Picard\thanks{Univ. Grenoble Alpes, LJK, 38000 Grenoble, France and Nexialog, 75011 Paris, France.
  (\email{tom.picard@grenoble-inp.org}).}}

\newcommand{\diff}{\mbox{d}}

\newcommand{\ga}{\gamma}

\newcommand{\ka}{\kappa}

\newcommand{\bfA}{\mbox{\boldmath $A$}}

\newcommand{\E}{{\mathbb E}}
\newcommand{\Y}{{\mathbb Y}}

\newcommand{\T}{{\mathbb T}}

\newcommand{\bfF}{\mbox{\boldmath $F$}}

\newcommand{\bfG}{\mbox{\boldmath $G$}}

\newcommand{\Prob}{{\mathbb P}}

\newcommand{\J}{{\mathbb J}}

\newcommand{\cF}{{\cal F}}
\newcommand{\cG}{{\cal G}}

\def\cF{\mathcal{F}}
\def\cG{\mathcal{G}}

\definecolor{myblue}{RGB}{58, 110, 255}
\definecolor{blue}{HTML}{1F77B4}
\definecolor{orange}{HTML}{FF7F0E}
\definecolor{green}{HTML}{2CA02C}

\usepackage{dsfont}
\usepackage{hhline} 
\usepackage{scalerel,stackengine}
\usepackage[toc,page]{appendix}
\stackMath
\newcommand\reallywidehat[1]{%
\savestack{\tmpbox}{\stretchto{%
  \scaleto{%
    \scalerel*[\widthof{\ensuremath{#1}}]{\kern-.6pt\bigwedge\kern-.6pt}%
    {\rule[-\textheight/2]{1ex}{\textheight}}
  }{\textheight}%
}{0.5ex}}%
\stackon[1pt]{#1}{\tmpbox}%
}
\begin{document}

\maketitle

\begin{abstract}
Analyzing the effect of business cycle on rating transitions has been a subject of great interest these last fifteen years, particularly due to the increasing pressure coming from regulators for stress testing. In this paper, we consider that 
the dynamics of rating migrations, in a pool of credit references, is governed by a common unobserved latent Markov chain. We explain how the current state of the hidden factor, can be efficiently inferred from observations of rating histories.
We then adapt the classical Baum-Welch algorithm to our setting and show how to estimate the latent factor parameters. Once calibrated, we may reveal and detect economic changes affecting the dynamics of rating migration, in real-time. The filtering formula is then used to predict future transition probabilities according to the economic cycle without using any external covariates. We propose two filtering frameworks: a discrete and a continuous version. We demonstrate and compare the efficiency of both approaches on fictive data and on a corporate credit rating database. The methods
could also be applied to retail credit loans. Finally, under a point process filtering framework, we extend the standard discrete-time filtering formula to a more general setting, where the hidden process does not need to be a Markov chain. 
\end{abstract}

\input{intro}

\input{DiscreteFilter}

\input{ContinuousFilter}
\input{FictiveData}
\input{RealData}
\input{GeneralFilter.tex}
\input{Conclusion}
\section*{Acknowledgments}
The authors are immensely grateful to Ragnar Norberg who significantly contributes to early stages of this work.
 We also want to acknowledge Baye Matar Kandji who greatly improves the Baum-Welch type calibration algorithm.   
This work was carried out as part of a PhD thesis ``CIFRE''. The authors thank Nexialog Consulting for its support.  
\appendix

\input{Appendice}

\bibliographystyle{plain}
\bibliography{references}

\end{document}

%% file: intro.tex
\section{Introduction}

Credit risk research has been on the rise over the last 20 years. In particular, the challenges that arose from the previous financial crisis prompted researchers to develop credit risk valuation models that take into account the evolution of the business cycle. The evolution of the banking supervisor regulations and accounting rules follow this trend: official guidelines of IFRS 9 as \cite{IFRS92016} recommend the use of point-in-time estimation of credit risk, i.e., the use of  macro-economic factors in the credit risk assessment process. 
Moreover, the EBA guidelines \cite{PDLGD2017} on LGD downturn, require to identify economic downturn periods to adjust the initial LGD estimations. In addition,  EBA stress testing methodology described in \cite{Stress2018} strongly relies on past economical scenarios.
\\
\\
A credit rating system evaluates the confidence in the ability of the borrower to comply with the credit's terms. A default probability is associated with each rating, which under Basel regulations, impacts the amount of capital required for a credit (see \cite{Bale2017}).
Such ratings may be generated by internal rating systems (IRB) or issued, by external rating agencies. 
After the assignment of the initial credit rating, reviews are performed either periodically or based on market events. In that way an entity's rating may evolve through time according to its health and to the economic cycle.
Therefore, predicting the evolution of rating migrations is of primary importance for every financial institution. The migrations of a group of credit entities can be described by transition matrices, defining the probabilities to move from one rating state to another in a given period of time. Given recent evolution in banking supervisory and accounting rules, the challenge is to explain changes in transition probabilities due to changes in the business cycle. 
\\
\\
Factor-based migration models provide a nice framework for capturing migration sensitivities to macro-economic changes. Models in this class allow transition probabilities to depend on dynamic factors. Two main families of models are usually considered in the credit risk literature~: the ``ordered Probit'' (or structural approach) introduced by \cite{tobin1958estimation}, popularized by \cite{merton1974pricing} and studied for credit ratings , e.g., in \cite{bangia2002ratings}, \cite{feng2008ordered}, \cite{Gagliardini2005}, \cite{nickell2000stability} and the ``multi-state latent factor intensity model'' (or intensity approach) studied, e.g., in \cite{figlewski2012modeling}, \cite{kavvathas2001estimating}, \cite{koopman2008multi} and \cite{Leijdekker_Spreij2008}.
This paper focuses on the second approach.
\\
\\
In the basic reduced intensity form model, a credit event corresponds to the first jump time of a Poisson process with a constant intensity. The reduced form approach has been widely studied in the credit risk literature, see, e.g., \cite{duffie2007multi}, \cite{jarrow1997markov}. Nevertheless \cite{altman1992implications}, \cite{fledelius2004non} and \cite{hamilton2004rating} provide evidences that migration intensities vary over time. In their research, \cite{hamilton2004rating} and \cite{lando2002analyzing} show that the rating transition probabilities depend on whether the bond entered its current rating by an upgrade or a downgrade. \cite{lando2002analyzing} also notice that the probability to leave a rating category tends to decrease with the time spent at that rating. Above all, \cite{bangia2002ratings}, \cite{nickell2000stability} give strong evidence that credit risk is considerably affected by the macroeconomic conditions and differs across different economic regimes. 
\\
\\
In both structural and intensity models, the factors may be considered observable or unobservable.
The second approach has emerged in response to criticisms made against the first. As \cite{Gagliardini2005} point out, the risk in selecting covariates lies in excluding others which could be more relevant.
\cite{Reda2015} provide an overview of usual modelling and estimation approaches and compare the estimation and the predictive performance of each approach on real data. When the underlying factors are unobservable, they adapt a method given in \cite{Gagliardini2005} to represent the considered factor migration model as a linear Gaussian model, and apply a Kalman filter to predict the state of the underlying latent factor. This approximation lies on the hypothesis that the data set is large enough to apply asymptotic normality. This assumption may be too restrictive 
and may explain the poor quality of predictions obtained by \cite{Reda2015}. 
\\
\\
A natural alternative consists in directly filtering the hidden factor given rating transitions' past history. For a bond portfolio, the dynamics of rating migrations can mathematically be represented as a multivariate counting process, each component representing the cumulative number of transitions from one rating category to another. Estimating the hidden factor dynamics by only using observations of the counting process has already been considered in the credit risk literature. For instance, \cite{fontana2010credit} and \cite{frey2012pricing} follow this approach for pricing derivatives under incomplete information. 
\\
\\
A realistic and standard setting assumes that the unobserved driving factor is given as a finite state Markov chain and that the rating transition process follows a Hidden Markov model (HMM). \cite{Bremaud1981} and \cite{elliott2008hidden} respectively present a detailed analysis of continuous-time and discrete-time filtering under special HMM assumptions. Hidden Markov Chain modeling (HMM) remains a popular approach in credit risk analysis (see e.g., \cite{ching2009modeling}, \cite{elliott2008hidden}, \cite{elliott2014double}, \cite{thomas2002hidden}). The hidden process can have different interpretations according to the assumptions made and the way to filter. In the credit rating literature, \cite{korolkiewicz2008hidden} assume that the observed rating of a firm is a noisy observation of its true credit rating, represented by a hidden Markov chain. They apply an Expectation Maximisation (EM) algorithm for hidden Markov models under a discrete-time setting. They calibrate the filtering formula applied to Markov chains, derived in \cite{elliott2008hidden}, to infer from its rating evolution, the true credit quality of a firm. \cite{damian2018algorithm} extend the parameter estimation via the EM algorithm to continuous-time hidden Markov models. Similarly, they infer the true credit quality from rating observations but also with credit spreads. In these studies, each firm has its own true rating process, therefore its rating dynamic is governed by its own hidden process. Then the dynamic of rating transition of entities are governed by independent and identically distributed hidden Markov chains. Therefore, rating observations are assumed to be independent and an aggregated calibration procedure can be made.  
In the context of this paper, the hidden factor is interpreted as a systematic and common factor, governing transitions of all firms.
Among the studies which share the same interpretation, \cite{giampieri2005analysis}, also use the classical Baum-Welch algorithm (introduced in \cite{baum1970maximization}), for estimation of a two-state hidden factor driving occurrence of defaults. They obtain estimates for the model parameters and are able to reconstruct the most likely past sequence of the hidden factor. Their approach only holds for an unique transition and is not suitable for providing online estimations of the hidden factor state. In the same vein \cite{deroose2008reviewing} and \cite{oh2019estimation} identify two states, one of expansion and the other of contraction. 
In particular, \cite{oh2019estimation}  use an extension of the Baum-Welch algorithm adapted to ``regime switching hidden Markov model'' (RSMC) to forecast sovereign credit rating transitions. In a different scope, they also assume that every rating processes are governed by independent and identically distributed Markov chains.
\\
\\
The contributions of this paper are both theoretical and practical. We apply filtering framework to credit migrations, and show how to infer the current state of the hidden factor from past rating transitions. An EM algorithm is adapted to estimate the parameters involved. Contrary to \cite{damian2018algorithm}, \cite{korolkiewicz2008hidden}, \cite{oh2019estimation}, we assume that the dynamics of rating migrations in a pool of credit references, is governed by a common unobserved latent Markov chain, which aims to represent the economic cycle. Therefore the realization of the unobservable factor is assumed to be common to every firm whereas one hidden factor per bond is considered in \cite{damian2018algorithm}, \cite{korolkiewicz2008hidden}, \cite{oh2019estimation}. We believe that our approach which rather keeps the dependencies within the observations sample, is reliable and realistic. Indeed, rating entities should be affected by the same realization of the economic factor. This different consideration changes the way to calibrate and to filter: our filtering framework uses the whole history of aggregated number of jumps. Once calibrated, we may reveal and detect economic changes affecting the dynamics of rating migration, in real-time. By updating the filtered factor, we are able to forecast rating transitions according to these economic changes. Our approach may be considered as a new Point-in-time (PIT) rating transitions modeling which does not use any macro-economic factors.
\\
Behind every model mentioned, choosing a continuous or discrete approach is crucial and is a matter of debate. This paper aims at participating to this debate by presenting different results: we adapt filtering formulas, derived under special HMM assumptions in \cite{Bremaud1981} and\cite{elliott2008hidden}, to migration ratings context, both in a continuous-time and discrete-time setting. In particular, we show how to adapt the continuous-time filtering framework to handle discrete-time data and simultaneous jumps. 
We assess and compare both approaches on a fictive data set and on a Moody's ratings history [01/2000-05/2021] of a diversified portfolio of 5030 corporate entities. Finally, to further pursue the study, under a point process filtering framework, we derive a general discrete-time filtering formula which extends the standard Markov case. 
\\
\\
The paper is organised as follows. First, Section \ref{sec:rating-discrete} presents the discrete-time filtering framework adapted to credit rating migrations. Correspondingly, Section \ref{sec:rating-continuous} describes the continuous-time filtering version adapted to the same context. Section \ref{sec:fictiveData} illustrates and validates the two filtering approaches on fictive data. Then, in Section \ref{sec:realData}, we compare the two filters on real data sets. Finally, we present a general discrete-time point process filtering equation throughout Section \ref{sec:discreet-filter}, where the Markov assumption is relaxed.

%% file: DiscreteFilter.tex
\section{Discrete-time filtering for rating migrations}
\label{sec:rating-discrete}
We aim to adapt the discrete-time  filtering framework developed for hidden Markov chain in \cite{elliott2008hidden}, to the context of rating migrations. We consider that a common Markov chain governs the dynamics of all transitions. This hidden process may carry the systematic risk shared by rating transitions and might be interpreted as the economic factor. We first present the formula in the context of a single pair of rating categories (a single transition from one given rating, to another). Then we extend the approach to multiple rating transitions. 
\\
\\
Let $\Gamma \in \mathbb{N}$ be a discrete time horizon. We work with the filtered probability space $(\Omega,\bfA,\bfF = (\cF_n)_{n \in \{0,\ldots,\Gamma\}},\Prob)$. Let $\Theta$ be a Markov chain with finite number of states in $\T = \{1,\ldots,m\}$. Let $\bfF^{\Theta}$ be the natural filtration of $\Theta$,  augmented with $\Prob-$null sets. Let's define, for $h \in \T, \ n \in
\{0,\ldots,\Gamma\}$, $ I_n^h=\mathds{1}_{[\Theta_{n}=h]}$, the indicator function of $\Theta$ on state $h$, at time $n$.
\\
We consider the list of rating categories $\bar{\Upsilon}= \{1, \ldots, p \}$. This space represents different credit risk scores or ratings in descending order, $p$ being the default state. For example, Standard and Poor's long-term investment ratings can be translated to AAA = $1$, AA = $2$, A = $3$, BBB =$4$, \ldots, D (Default) = $10$. In practice the number of credit entities monitored over time may vary, either because some names are censored or simply because of missing data. This consideration is deeply discussed in Section \ref{dataDescription}. We attribute the rating $0$ to an entity in this case. Then, it is clear that a transition involving the rating of censure $0$, is assumed to be independent with the states of the hidden factor. Then we call ${\Upsilon}= \{0, \ldots, p \}$, the completed list of ratings. Note that, with this setting, the number of entities observed on ${\Upsilon}$ is constant over time and equal to $Q$. 
Let $Z_n^q\in {\Upsilon}$,  be the random variable, describing the state of bond $q$, $q\in\{1,\dots,Q\}$, at time $n \in \{0,\ldots,\Gamma\}$ and let $Z^q = (Z_n^q)_{n \in \{0,\ldots,\Gamma\}}$ be the migration process that describes its evolution. The counting process, which counts the total number of jumps of the entities, from rating $i$ to $r$, $(i,r)\in {\Upsilon}^2$, is denoted by $N^{ir}$ and is such that, $\forall n \in \{1,\ldots,\Gamma\}$,
$$
\Delta N_{n}^{ir}=\sum_{q\leq Q}\mathds{1}_{[Z_{n-1}^q = i, Z_n^q = r]}.
$$
Let $\bfF^N$ be the natural filtration of $N$, augmented with $\Prob-$null sets. For $n \in \{1, \dots, \Gamma\}$, we introduce for every process $O$, the notation
$$
\hat{O}_{n}=\E[O_{n}|\cF_{n}^N].
$$
In addition, let us denote by the process $Y^i$ representing the number of observed and active entities that belong to rating $i$, which may jump to another one. It may evolve over time, according to censorship, arrivals of new entities on the market with initial rating $i$, rating transitions or bankruptcies.
This process is assumed to be $\bfF^{N}$--predictable. In this framework, $\Theta$ aims to represent the systematic risk factor.
It is unique and governs dynamics of all rating transitions. Furthermore, for the sake of tractability, it is assumed that $\Theta$ impacts entities with the same rating in the same way. Consequently, we consider that entities with the same rating, are perfectly indistinguishable.
Under this exchangeable setting, to infer information on the underlying hidden factor $\Theta$, it is sufficient to observe the aggregated counting processes $(\Delta N^{ir})_{i,r\in \Upsilon}$ and the processes $(Y_i)_{i\in \Upsilon}$. The number of jumps from $i$ to $r$ can not exceed the number of active entities. Then, the support of $\Delta N_n^{ir}$ is $\J^i_n=\{0,\dots, Y_n^i\}$.
Let us define the support of $\Delta N_n$, the product spaces $\J_n^{\otimes}=\prod_{i=1}^p \J^i_n$.
We define the transition probabilities of $\Theta$ as 
$$
\forall (s,h) \in \T^{2}, \ \forall n \ \in \{1,\ldots,\Gamma\}, \ K^{sh}=\Prob(\Theta_{n}=h|\Theta_{n-1}=s), \  \ \Pi^h=\Prob(\Theta_0=h).$$
\newpage
\subsection{Unique rating transition}
In a first framework, we present a specific case of the general setting presented above. we consider a unique transition from a rating called $r_0\in\bar{\Upsilon}$ to another, called $r_1\in\bar{\Upsilon}$. $N$ is reduced to $N=N^{r_{0},r_1}$, the associated univariate counting process which counts the total number of jumps of the entities, from rating $r_0$ to $r_1$, such that, $\forall n \in \{1,\ldots,\Gamma\}$,
$$
\Delta N_{n}=\sum_{q\leq Q}\mathds{1}_{[Z_{n-1}^q = r_0, Z_n^q = r_1]}.
$$
Let $\bfF^N$ be the natural filtration of $N$, augmented with $\Prob-$null sets. For $n\in \{1,\ldots,\Gamma\}$, $Y_n$ represents the number of active and observed entities with rating $r_0$ at time $n-1$, that may jump to rating $r_1$ at time $n$. We present here the recursive equation satisfied by $\hat{I}_{n}^{h}=\E[\mathds{1}_{[\Theta_{n}=h]}|\cF_{n}^N]$. We define, for any $n \in \{1,\ldots,\Gamma\} \ \text{and} \ s \in \T$, the conditional transition probabilities as 
$$L^{s}=\Prob(Z_{n}^q=r_1|Z_{n-1}^q=r_0, \Theta_{n-1}=s).$$
According to the previous notations, the number of jumps from $r_0$ to $r_1$ cannot exceed the number of active entities. Then, the support of $\Delta N_n$ is $\J_n=\{0,\dots, Y_n\}$. Knowing that $\{\Theta_{n-1}=h, Y_{n}=y_{n}\}$, we assume that the conditional distribution of the random variable $\Delta N_{n}$ is binomial with parameters $(y_{n},L^h)$. Similar settings can be found in \cite{caja2015influence} and \cite{giampieri2005analysis}. 
\begin{proposition}
\label{prop:discreet-univariate-rating}
With these assumptions, the filtered process $\hat{I}_{n}^h$ solves the following recursive equation. For $n=1,\dots,\Gamma$,
\begin{equation}
\label{eq:discreet-univariate-rating}
\hat{I}_{n}^{h}=\sum_{j \in \J_n}\frac{\sum_{s=1}^m K^{s h}(L^{s})^{j}(1-L^{s})^{Y_{n}-j}\hat{I}_{n-1}^{s}}{\sum_{s=1}^m(L^{s})^{j}(1-L^{s})^{Y_{n}-j}\hat{I}_{n-1}^{s}}\mathds{1}_{[\Delta N_{n}=j]}.
\end{equation}
\end{proposition}
\begin{proof} 
This formula can be derived from the general discrete-time filtering formula, presented later in this paper, in Section \ref{sec:discreet-filter}. This approach is described in Remark \ref{rem:other-proof}. This formula can also be derived from the filtering formula \cite[Chapter~2-Theorem 4.3]{elliott2008hidden}. However, the filtering formula must be applied to the context described above. We apply the formula to the counting process $\Delta N$, considered as an observable Markov chain with finite number of states, $\J_n$ at time $n\in\{1,..\Gamma\}$, and governed by the common hidden Markov chain $\Theta_n$.  
\end{proof}
\begin{remark}
For the sake of interpretability, our setting assumes that the entities should be affected by the same realisation of the economic factor. Therefore, our filtering framework uses the whole history of aggregated number of jumps, keeping the dependencies within the observations sample.
\end{remark}
In this framework, the hidden factor governs a unique transition. It might be more realistic to assume that it affects all transitions. 
Then, we naturally extend the previous equation to a multivariate setting.
\subsection{Multiple Rating transitions}
In this application, we extend the previous result by considering multiple rating transitions. We present now the recursive equation satisfied by $\hat{I}_{n}^{h}=\E[\mathds{1}_{[\Theta_{n}=h]}|\cF_{n}^N]$, where $N=(N^{ir})_{ir\in\Upsilon}$. According to this setting, we define the conditional transition probabilities of $(Z^{q})_{q}$ as
\[\text{For} \ (i,r) \in {{\Upsilon}}^{2}, \ s \in \T \ \text{and} \ \forall n \ \in \{1,\ldots,\Gamma\}, \ L^{s,ir}=\Prob(Z^{q}_{n}=r|Z^{q}_{n-1}=i, \Theta_{n-1}=s).\]
Note that if $i=0$ or $r=0$ then  $\forall s \ \in \T, \ L^{s,ir}=\Prob(Z^{q}_{n}=r|Z^{q}_{n-1}=i)$. Indeed, the transitions from or to the rating $0$ are assumed to be independent of the hidden factor because censorship is non-informative. 
\\
Since only one realisation of trajectory of $\Theta$ governs observed rating processes, the random variables $\{Z_n^{q},\ q\leq Q\}$, for $n\in\{1,\ldots,\Gamma\}$, are still not independent. Nevertheless, knowing the sate of $\Theta_n$, they are independent. 
Then, the conditional distribution of the multivariate random variable $\Delta N^{ir}_n$, knowing that $\{\Theta_{n-1}=s, Y^i_{n}=y^i_{n}\}$, is multinomial with parameters $(y^i_{n},(L^{s,ir})_r)$.
\begin{proposition}
With such assumptions, the filtered process $\hat{I}_{n}^h$ is solution of the following recursive equation
\begin{equation}
\label{eq:multivariate-discreet-credit} 
\hat{I}_{n}^h=\sum_{\delta \in \J^{\otimes}_n}\frac{\sum_{s} K^{s h}  \prod_{i,r=1}^p (L^{s,ir})^{\delta_{ir}}\hat{I}_{n-1}^{s}}{\sum_{s}   \prod_{i,r=1}^p (L^{s,ir})^{\delta_{ir}}\hat{I}_{n-1}^s}\mathds{1}_{[\Delta N_{n}=\delta]}.
\end{equation}
\end{proposition}
\begin{proof}
We leave the proof to the reader as it goes along the same lines as the proof of Proposition~\ref{prop:discreet-univariate-rating}.
\end{proof}
Once the hidden factor filtered state is obtained, it is then possible to predict the future migration probabilities.\\

\subsection{Transition probability prediction}
We define for $(i,r) \in \bar{\Upsilon}^2$, the process $\nu^{ir}$, which forecasts the transition probability from rating $i$ to rating $r$, for the next time step.  
$$
\forall \ (i,r) \in \bar{\Upsilon}^2, \forall n \in \{1,\ldots,\Gamma\} :\nu_{n-1}^{ir}=\E\left[\mathds{1}_{[Z^q_{n}=r]}|Z^q_{n-1}=i, \cF_{n-1}\right]=\sum_{h \in \T}L^{h,ir}{I}_{n-1}^h.
$$
With the filtered current hidden factor, we can forecast the future transition probabilities 
\begin{equation}
\label{prediction}
\forall \ (i,r) \in \bar{\Upsilon}^2, \forall n \in \{1,\ldots,\Gamma\} :\hat{\nu}_{n-1}^{ir}=\E\left[\mathds{1}_{[Z^q_{n}=r]}|Z^q_{n-1}=i, \cF_{n-1}^N\right]=\sum_{h\in \T}L^{h,ir}\hat{I}_{n-1}^h.
\end{equation}
\subsection{Calibration}
\label{sec:calibration}
In this section, we explain how to estimate model parameters involved in the filtering equation (\ref{eq:multivariate-discreet-credit}). We apply the so-called Baum-Welch algorithm to our discrete-time framework. 
\subsubsection{A Baum-Welch algorithm adapted for a discrete framework}
The proposed method is a maximisation expectation (EM) algorithm for hidden Markov chains (HMM), adapted to the model. We can find studies on the classical model in \cite{bishop2006pattern}, \cite{Ozcan2013}, \cite{Rabinet1989} and \cite{Tenyakov2014}. 
\\
However the classical algorithm is not totally suitable for calibration of the discrete filtering equation (\ref{eq:multivariate-discreet-credit}).
We highlight one inconsistency between the classical algorithm and our model. Rating process trajectories of each entity must be independent whereas in our framework, they are dependent through the common factor $\Theta$.
\\
The first step of the algorithm assigns initial values to the parameters we want to estimate. Then the algorithm replaces the missing data (states of $\Theta$) with Bayesian estimators using the observations and the current parameters estimated values.
\\
The second one consists in improving a conditional likelihood. Better parameters are estimated. Then these new estimates are used to repeat the first step. We iterate this process to converge to a local maximum.
\\
Let $Z=(Z^q)_{q\leq Q}$ be the multivariate rating process and we call for $(n_1,n_2) \in \{0,\ldots,\Gamma\}^2$, $(Z)_{n\in \{n_{1},\ldots,n_2\}}=Z_{n_{1}|n_{2}}$, the rating trajectories between time $n_1$ and $n_2$. As the new rating does not only depend on the economic cycle (state of $\Theta$) but also on the previous rating, we apply the Baum-Welch algorithm by considering that 
$$
\forall n \in \{1,\ldots,\Gamma\}, \ \Prob(Z_{n}|Z_{0},\ldots,Z_{n-1},\Theta_{0},\ldots,\Theta_{n-1})=\Prob(Z_{n}|Z_{n-1},\Theta_{n-1}).
$$
Furthermore, as rating history of all entities are dependent on a same realization of $\Theta$, we must adapt our algorithm differently from \cite{oh2019estimation} who considered that each rating process is governed by its own and independent trajectory of $\Theta$.

\subsubsection{Initialization}
\label{sec:init}
The calibration algorithm presented is based on iterative improvement of a likelihood. This expectation maximization algorithm (EM) (see \cite{dempster1977maximum}), as most of iterative maximisation algorithms, might be trapped in a local maximum. 
Obtained parameters may not be relevant when the global maximum is not found. This success is deeply dependant on the initialization.
Several empirical and analytical methods have been proposed to deal with this matter. In \cite{liu2014proper}, transition probabilities are initiated using empirical frequencies. They succeed to considerably reduce the number of iterations to find their local maximum. By noticing that the transition matrices have strong diagonals, \cite{oh2019estimation} initialized their model by adding small perturbations to identity matrix or to uniform distributions.
In our study we choose a third option which seems to be more reliable: we test a high number of initial values (picked at random) in order to find the global maximum. 
In order to guarantee almost surely convergence to the global maximum, initial values are chosen according to a uniform distribution on the parameters space.

\subsubsection{Bayesian estimators}
This part only presents the main results of the algorithm. One can find more details of the computations in Appendix \ref{proof:calibration-discreet}.
\\
We define the forward probability as denote,
$$
\forall s \in \T, \forall n \in \{1,\ldots,\Gamma\}: 
\alpha_{n}(s)=\Prob(Z_{0|n}=z_{0|n},\Theta_{n-1}=s)
$$
and the backward probability as
\\
$$
\forall s \in \T, \forall n \in \{1,\ldots,\Gamma-1\}: 
\beta_{n}(s)=\Prob(Z_{n+1|\Gamma}=z_{n+1|\Gamma}|Z_{n}=z_{n},\Theta_{n-1}=s).
$$
\\
We use the following recursive formulas in order to compute the two previous probabilities
$$
\label{alpha}
\alpha_{n}(s)=\sum_{l=1}^{m}\alpha_{n-1}(l)K^{ls}\prod_{i,r \in {\Upsilon}}(L^{s,ir})^{\Delta N_n^{ir}},
$$
$$
\label{beta}
\beta_{n}(s)=\sum_{l=1}^{m}\beta_{n+1}(l)K^{sl}\prod_{i,r \in {\Upsilon}}(L^{l,ir})^{\Delta N_{n+1}^{ir}}.
$$
For $n \in \{1,\ldots,\Gamma\}$, we introduce two random variables useful to describe $\Theta$
\\
$$
\label{u}
u_{n}(h)=\mathds{1}_{[\Theta_{n}=h]},
$$
$$
\label{v}
v_{n}(s,h)=\mathds{1}_{[\Theta_{n}=h,\Theta_{n-1}=s].}
$$
The forward and backward probabilities are helpful to compute the following Bayesian estimators 
\\
$$
\check{u}_{n}(h)=\Prob(\Theta_{n}=h|Z_{0|\Gamma}=z_{0|\Gamma})=\frac{\beta_{n+1}(h)\alpha_{n+1}(h)}{L_{\Gamma}},
$$
and 
$$
\check{v}_{n}(s,h)=\Prob(\Theta_{n}=h,\Theta_{n-1}=s|Z_{0|\Gamma}=z_{0|\Gamma})=\frac{\beta_{n+1}(h)K^{sh}\alpha_{n}(s)\prod_{i,r \in {\Upsilon}}(L^{h,ir})^{\Delta N_n^{ir}}}{L_{\Gamma}},
$$
where $L_{\Gamma}$ is the likelihood of the whole sample,
$$
L_{\Gamma}=\Prob(Z_{0|\Gamma}=z_{0|\Gamma})=\sum_{j=1}^m\alpha_{\Gamma}(j).
$$

\subsubsection{Parameters estimation}
The maximization phase consists in finding better parameters than those of the previous iteration.
We call $M^{(\gamma)}=(\Pi^{(\gamma)},L^{(\gamma)},K^{(\gamma)})$, the parameters obtained at the iteration $(\gamma)$.
\\
The new parameters are deemed to improve the likelihood according to: 
$$
\Prob(Z_{0|\Gamma}=z_{0|\Gamma}|M^{(\gamma+1)})\geq \Prob(Z_{0|\Gamma}=z_{0|\Gamma}|M^{(\gamma)}).
$$
To achieve that, we are looking for maximizing $\log\frac{\Prob(Z_{0|\Gamma}=z_{0|\Gamma}|M^{(\gamma+1)})}{\Prob(Z_{0|\Gamma}=z_{0|\Gamma}|M^{(\gamma)})}$, which is equivalent to maximize
\\
\\
$$
Q(M^{(\gamma)},M^{(\gamma+1)}) =\sum_{\theta \in \{1,\ldots,m\}^\Gamma} \Prob(\Theta_{0|\Gamma}=\theta,Z_{0|\Gamma}=z_{0|\Gamma}|M^{(\gamma)})\log\Prob(\Theta_{0|\Gamma}=\theta,Z_{0|\Gamma}=z_{0|\Gamma}|M^{(\gamma+1)}).\\
$$
\\
After optimization, we obtain the following estimators
\\
$$
\Pi^{h}=\check{u}_{0}(h); \quad
L^{s,ir} =\frac{\sum_{n=1}^\Gamma\check{u}_{n-1}(s)\Delta N_n^{ir}}{\sum_{n=1}^\Gamma\check{u}_{n-1} (s)Y_n^i}; \quad 
K^{sh} =\frac{\sum_{n=1}^\Gamma\check{v}_{n}(s,h)}{\sum_{n=1}^\Gamma\check{u}_{n-1}(s)}.
$$

%% file: ContinuousFilter.tex
\section{Continuous-time filtering for rating migrations}
\label{sec:rating-continuous}
In this section, we explain how to apply continuous filtering framework to credit rating migrations.
\subsection{Multiple Rating transitions}
Let $(\Omega,\bfA,\bfF = (\cF_t)_{t \in [0,T]},\Prob)$,
 be a filtered probability space satisfying the ``usual conditions'' of right-continuity and completeness needed to justify all operations to be made. All stochastic processes encountered are assumed to be adapted to the filtration $\bfF$ and integrable on $[0,T]$. In particular, we have $\bfA=\cF_{T}$.
The time horizon $T$ is supposed to be finite.
In order to remain realistic and to fix the terminology, a bond market containing a finite number of individual bonds is considered. All bonds are affected by variable and random market conditions represented by the same latent process $\Theta$. The hidden factor driving process $\Theta$ is assumed to be a Markov chain with finite number of states  in $\T$ and with constant transition intensities $k^{rh}$, $r \neq h$ and such that $k^{rr} = - \sum_{l; l \neq r} k^{rl}$, so that, for small enough $dt$,
\begin{equation}
\label{K-to-k}
\forall \ r \neq h \in \T^2 :\Prob (\Theta_{t+dt} = h\,|\, \Theta_t = r) \approx k^{rh}\, dt.
\end{equation}
The initial distribution of $\Theta$ is defined as
$$
\forall h \in \T : \Pi_{h}=\Prob(\Theta_{0}=h).
$$ 
Let us introduce the state indicator processes $I^h$, $h \in \T= \{1,\ldots,m\}$, defined by
$$
I_t^h = \mathds{1}_{[\Theta_t = h]}, \; h \in \T.
$$
A bond $q$ of the sample is observed between the dates $s^q$ and $u^q$, $0 \leq s^q \leq u^q \leq T$. 
We consider that the bond $q$ may evolve in the same credit rating categories space, $\bar{\Upsilon} = \{1, \ldots, p \}$, than for the discrete-time framework.  
\\
Let $\Y=\{(i,j) \in \bar{\Upsilon}^2,\  i \neq j\}$, be the space of possible migrations.
Let $Z_t^q \in \bar{\Upsilon}$ be the rating state of bond $q$ at time $t$ and $Z^q = (Z_t^q)_{t \in [s^q,u^q]}$ be 
the rating process describing its evolution. The migration counting process associated with $Z^q$, which counts the number of jumps of the entity $q$ from rating $i$ to $j$, is denoted by $N^{q,ij}$ and is such that, $\forall t \in [s^q,u^q]$,
$$
\Delta N_{t}^{q,ij}=\mathbf{1}_{[Z_{t-}^q = i, Z_t^q = j]}.
$$
We introduce $\bfF^N = (\cF_t^N)_{t \in [0,T]}$ the natural filtration of the multivariate counting process $N=(N^{ij})_{i,j\in\bar{\Upsilon}}$.
We assume that two entities cannot jump at the same time, i.e, they do not have any common jumps, i.e.,  $\forall i, \ j, \ r, \ k \in \bar{\Upsilon}, \Delta N_t^{ij} \Delta N_t^{rk} =\delta_{ir}\delta_{jk}\Delta N_t^{ij}$. We also assume that there is no common jump with $\Theta$. 
With the same notation provided in the discrete-time framework, all the processes $O$ filtered  by $ \cF_t^N $ are written $\hat{O}_t = \E [O_t | \cF_t^N]$.
\\
\\
In this study, we assume that $(Z_{t}^q, t\in[0,T])_q$ are described within a factor migration model. 
More specifically, knowing $\cF_T^{\Theta}$, the rating processes $(Z^q)_q$ are assumed to be conditionally independent Markov chains with the same generator matrix.
In reality the change of rating of a bond may also induce the change of state of other bonds but this contagion effect is not considered in this paper. Moreover, the censorship mechanism governing $(s^q, u^q)$ is assumed to be non-informative and can therefore be considered deterministic and belonging to $\cF_0^N$.
Under this exchangeable setting, to infer information on the underlying hidden factor $\Theta$, it is sufficient to observe the aggregated counting processes $N^{ij}$,  $(i,j) \in \ {\Y}$,  defined by
$$
N_t^{ij} = \sum_{q;\, s^q \leq t < u^q} N_t^{q,ij}.
$$
and the exposure processes $Y^i$, $i \in  \bar{\Upsilon}$ defined by
$$
Y_t^i = \sum_{q;\, s^q \leq t < u^q}\, \mathbf{1}_{[Z_{t-}^q =i]} \,.
$$
$Y_t^i$ represents the number of observed entities with rating $i$ at time t, which may jump.
Note that the exposure process $Y^i$ is left continuous. It increases by $1$ when  $N^{ji}$  jumps for any $j \neq i$, or when a new bond enters the pool with rating $i$. It decreases by $1$ when a bound jumps outside rating $i$, i.e., whenever $N^{ij}$ jumps for $j\neq i$ or when a bond expires with rating $i$.
\\
\\
This framework aims to determine the recursive equation satisfied by $\hat{I}_{n}^{h}=\E[\mathds{1}_{[\Theta_{n}=h]}|\cF_{n}^N]$.
We denote by $(\nu_t^{ij})_{(i,j) \in {\Y}}$ the $\bfF$ intensity of $N$.
We assume that the intensities $(\nu^{ij})_{(i,j) \in {\Y}}$ are governed by the finite state hidden Markov chain $\Theta$. 
With these assumptions, the processes $(Z^q)_q$ are governed by their common intensity matrices $(l^h)_{h \in \T}$, such as for small enough $dt$
\begin{equation}
\label{L-to-l}
\Prob [Z_{t+dt}^q = j\,|\, Z_t^q = i,\Theta_t = h ] \approx \ell^{h,ij}\, dt.
\end{equation}
Then the counting processes $N^{ij}$ are governed by the $\bfF$ intensities
$$
\nu_t^{ij} = Y_t^i \, \sum_{r \in \T} \ell^{r,ij}\, I_{t-}^r \,.
$$
As $Y^{i}$ is $\bfF^N$--predictable, the $\bfF^N$--intensity of $N^{ij}$, may be written as
\begin{equation}
\label{sigmaFil}
\hat{\nu}_{t-}^{ij} = Y_t^i \, \sum_{r\in \T} \ell^{r,ij}\, \hat{I}_{t-}^r \,.
\end{equation}
We present the multivariate filtering formula satisfied by the process $(I_{t}^h)_{t\in[0,T]}$ in the following proposition.  
\begin{proposition}
\label{prop:continuous-rating}
\textit{
With the previous assumptions, the unobserved indicator process  filtered with rating jumps satisfies the following recursive equation}
\begin{equation}
d\hat{I}_{t}^{h}=\sum_{r=1}^{m}k^{rh}\hat{I}_{t-}^{r}dt+\sum_{i\neq j}\left(\frac{l^{h,ij}\hat{I}_{t-}^{h}}{\sum_{r}l^{r,ij}\hat{I}_{t-}^{r}}-\hat{I}_{t-}^{h}\right)\left(dN_{t}^{ij}-Y_{t}^{i}\sum_{r=1}^{m}l^{r,ij}\hat{I}_{t-}^{r}dt\right)
\label{eq:continuous-filter-rating}
\end{equation}
\end{proposition}
\begin{proof}
This result stems from the general continuous-time filtering theory developed in \cite[Sec.IV.1]{Bremaud1981}. Compared to the setting of \cite{Bremaud1981}, two adaptations are necessary. First, we deal here with an aggregated multivariate process over the entire portfolio and secondly, we take censorship into account though the processes of risk exposure $Y^i$. For the sake of completeness, we provide in Appendix \ref{ap:continuous-filter} a self-consistent proof yielding a general explicit filtering formula with no simultaneous jumps. Then, we provide in Appendix \ref{sec:rating-continuous-demo}, the adequate adaptations for the rating migration context to obtain the filtering equation (\ref{eq:continuous-filter-rating}).
\end{proof}
Usually, information on rating migrations are only available to public on a daily basis. For large credit portfolios, it is then frequent to observe multiple transitions (of several entities) occurring at the same day. In addition, clustering of rating migrations may also happen following the disclosure of a major economic events. Then, the presented continuous-time filtering approach is not fully compliant with migration data since it precludes simultaneous jumps between counting processes. 
This has lead us to preprocess the data and adapt the calibration algorithm.

\subsection{Adaptation of the continuous-time setting to discrete migration data}
\label{seq: continuous-adaptation}
 This part aims to explain how to adapt continuous filtering to discrete rating migration framework. We propose an adaptation of the calibration algorithm in order to be compliant with the continuous filtering formula (\ref{eq:continuous-filter-rating}). Previous adaptations done for the discrete framework are still required but are not sufficient: Baum-Welch algorithm is an estimation in discrete time which is not compliant with the continuous version of the filter. Furthermore the continuous-time filtering framework assumes the absence of simultaneous jumps.
 \\
For the first deviation, we propose to calibrate discrete-time parameters. Then it is necessary to switch to continuous time dimension for filtering. The transition between probabilities to intensities turns out to be easy when the time interval chosen is small enough according to (\ref{K-to-k}) and (\ref{L-to-l}). 
\\
The second deviation is also essential. There is no common jumps among the rating processes. 
However, ratings are not natural processes. Human decisions and algorithm appreciations are reported at the same moment in a day. Therefore we have to deal with simultaneous daily observations. Our solution consists in considering a different time grid. Each day is cut into small intervals and jumps are randomly spread on these time intervals. We insure to cut enough finely to have a maximum of one jump per interval. This manipulation has two drawbacks. First, the conditional independence between rating entities might be lost if the non-simultaneity of jumps is enforced. Then, distributing simultaneous jumps on a finer time grid may ultimately modify the original information.
\\
The second effect is studied in a testing benchmark at Section \ref{ContinuousFictive}.
\\
Once the data have been modified, we adapt our calibration algorithm to respect continuous structure. To this end, we use a prior law of jumps which respects the constraint of no simultaneity (that the number of observed jumps is only $0$ or $1$). 
The main idea consists on assuming that one entity is randomly chosen to be allowed to jump. Then, the entity may jump according the common migration matrices $\{(L^{h,ij})_{h}$, $ (i,j) \in \bar{\Upsilon}\}$.
\\
\\
More details of this adaptation can be found in Appendix \ref{proof:calibration-continuousBaum}.

%% file: FictiveData.tex
\section{Filtering on simulated data}
\label{sec:fictiveData}
The purpose of this section is to test and validate the continuous-time and discrete-time versions of the filter using simulated data: we build two rating migration databases from the two underlying credit migration models: the discrete-time model (as described in Section \ref{sec:rating-discrete}) and the continuous-time factor migration model (as described in Section \ref{sec:rating-continuous}). Since the structure of the two models are different, different data sample are used for testing the two approaches. The inputs of the filtering models are the evolution of the number of transitions, described by $\{\Delta N_n^{ir}, \ n\in\{1,\ldots,\Gamma\},\ i, \ r \in \Upsilon\}$. 
Our testing framework aims to compare the filtered trajectory of the hidden factor, $(\hat{\Theta}_n)_{n\in\{1,\ldots,\Gamma\}}$ and the real simulated one $({\Theta}_n)_{n\in\{1,\ldots,\Gamma\}}$.
Then it compares estimated point-in-time transition probabilities to real observed transition rates. The real observed rate of a transition is the ratio between the number of observed jumps during a time interval and the number of entities at the beginning of this time interval, susceptible to jump.
Estimations of transition probabilities are respectively given by $\{\hat{\nu}_n^{ir}, \ n\in \{1,\ldots,\Gamma\},\ i, \ r \in \Upsilon\}$, in (\ref{prediction}) for the discrete-time setting and $\{\hat{\nu}_n^{ir}\diff t, \ t\in [0,T],\ (i,r) \in {\Y}\}$, in (\ref{sigmaFil}) for the continuous-time setting. The value of $\diff t$ is chosen equal to $\frac{1}{1000}$, $1000$ being the number of small intervals in which we have cut each day (see Section \ref{seq: continuous-adaptation}). 
To calibrate the models, the EM algorithm of Section \ref{sec:calibration} is run 1000 times by sampling random initial values (see Section \ref{sec:init}). We keep the solution, which provides the highest likelihood estimation.

\subsection{Discrete time filtering approach}
\label{seq:fictive-discrete-filtering}
To build the discrete-time database, we assume that the hidden factor is described by a finite state space Markov chain with $7$ states. We consider a given set of model parameters. Each state of the hidden factor is associated with a specific rating migration matrix. 
We try to choose matrices compliant with filtering: we must have sufficient variability among conditional transition probabilities $(L^{h,ij})_{h}$, $ (i,j) \in \{1,..,p\}$. We work with 3 ratings categories $\{A, B, C\}$ and initialize our sample with 1000 entities per rating. 
Then the hidden Markov chain is simulated on 300 time steps. According to the hidden factor's sample path, we simulate transitions using conditional transition matrices. We use the first 200 time-steps to calibrate the model and the remainder to test is. 
\\
\\
We perform the calibration as described in Section \ref{sec:calibration}. Parameters chosen for sampling the model and estimated parameters are presented in Appendix \ref{ap:parameters}. The average difference is equal to 0,0014 (9,8\% of average relative error on superior and inferior diagonals) for the rating transitions $\{(L^{h,ij})_{h}$, $ (i,j) \in \{1,..,p\}\}$ and equal to 0,0159 (5,5\% of average of relative error) for the hidden factor's transition probabilities $\{K^{rh}, \ r, \ h \in\T\}$. Despite the high dimentionality of the problem, these indicators demonstrate that the estimation algorithm is able to recover the parameters.
\\
Figure \ref{fig} shows the real trajectory and the filtered trajectory of the hidden factor computed on the testing sample (of 100 time-steps). 
\begin{figure}[H]
   \centering
  \includegraphics[scale=0.7]{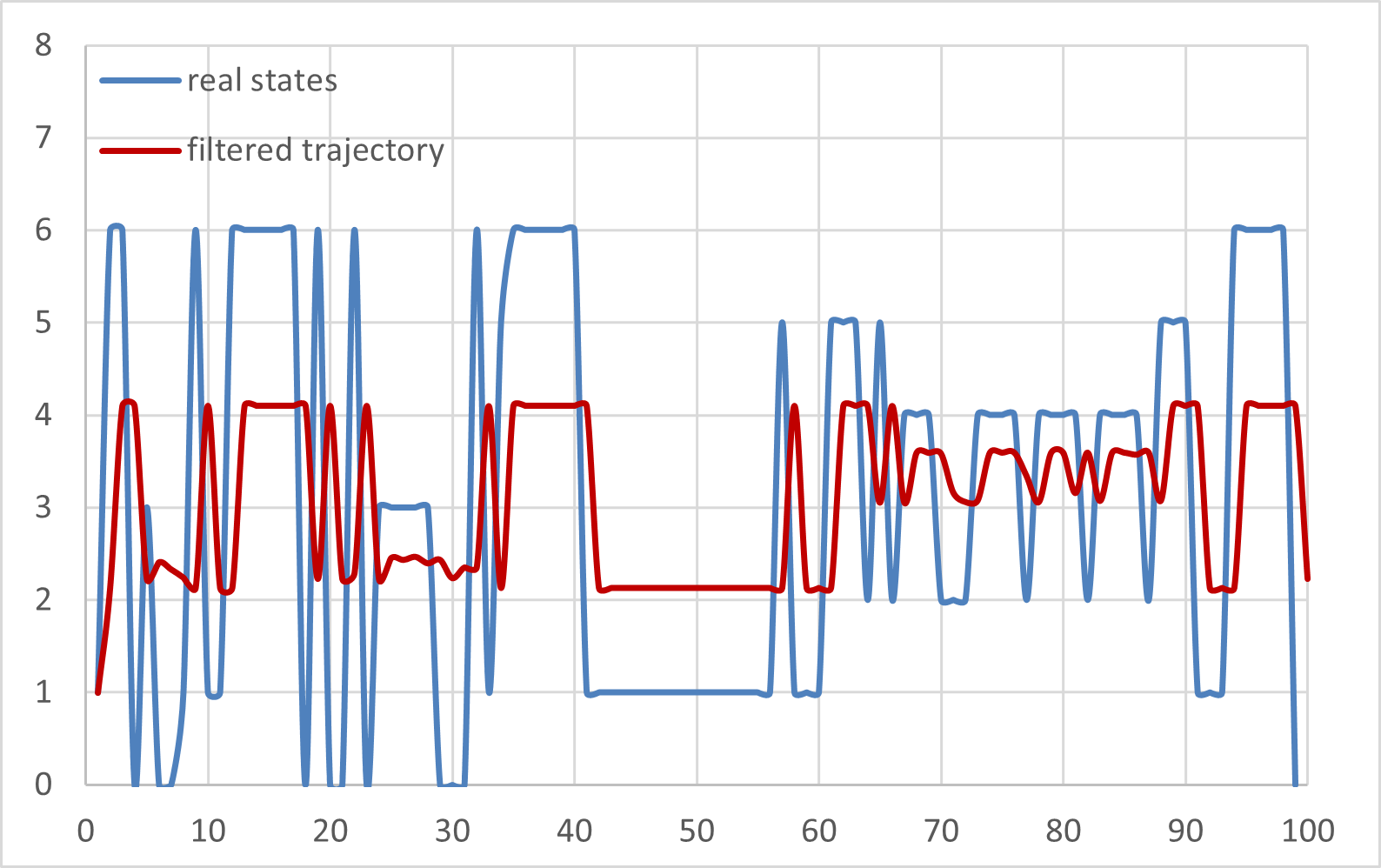}
   \caption{Simulated and filtered trajectories of the hidden factor on the last 100 testing dates}
  \label{fig}
 \end{figure}
We can notice that the filter is able to detect the changes of states. It faithfully follows the real trajectory and respects the different phases and trends. However it does not exactly mimic the true value, since the filtering formula is a weighted average of the state values. We can also observe a small delay in the estimation. 
The explanation is theoretical: it is caused by the effect of delay in the filtering model: 
the impact of the hidden factor at time $t-1$ is observable on ratings at time $t$. Therefore, when the hidden factor at time $t$ is filtered, the freshest observations available at this time, is the rating jumps at time $t$ which have been governed by the hidden factor at time $t-1$. Consequently we infer the current hidden factor state with information generated by its previous value.
\\
We can easily understand that the calibration plays a crucial role to make the filtering efficient. In order to forecast in time, states need to be strongly linked at least to another state. Let's imagine a rare and very unstable state. Since it is hardly visited from other states, it will never influence the direction of the filter and will be difficult to predict. Once the filter realizes that the hidden factor jumps to this state, it is too late, the hidden factor has already returned to another state. Finally, the filter is unable to capture rare events to unstable states. This remains acceptable since our main purpose is to detect transitions to stable regimes. Visiting a state for a brief period of time does not represent useful information for long term forecasting. 
\\
The following Figures \ref{fig:FictiveDiscreetRatios01}, \ref{fig:FictiveDiscreetRatios02}, \ref{fig:FictiveDiscreetRatios10}, \ref{fig:FictiveDiscreetRatios12}, \ref{fig:FictiveDiscreetRatios20}, \ref{fig:FictiveDiscreetRatios21}, represent real ratios of observed transitions with the predicted transitions dynamics, obtained from (\ref{eq:multivariate-discreet-credit}) and (\ref{prediction}), between the three considered rating categories $\{A,B,C\}$.
\\
\begin{figure}[H]
   \begin{minipage}[c]{0.46\linewidth} 
      \includegraphics[scale=0.5]{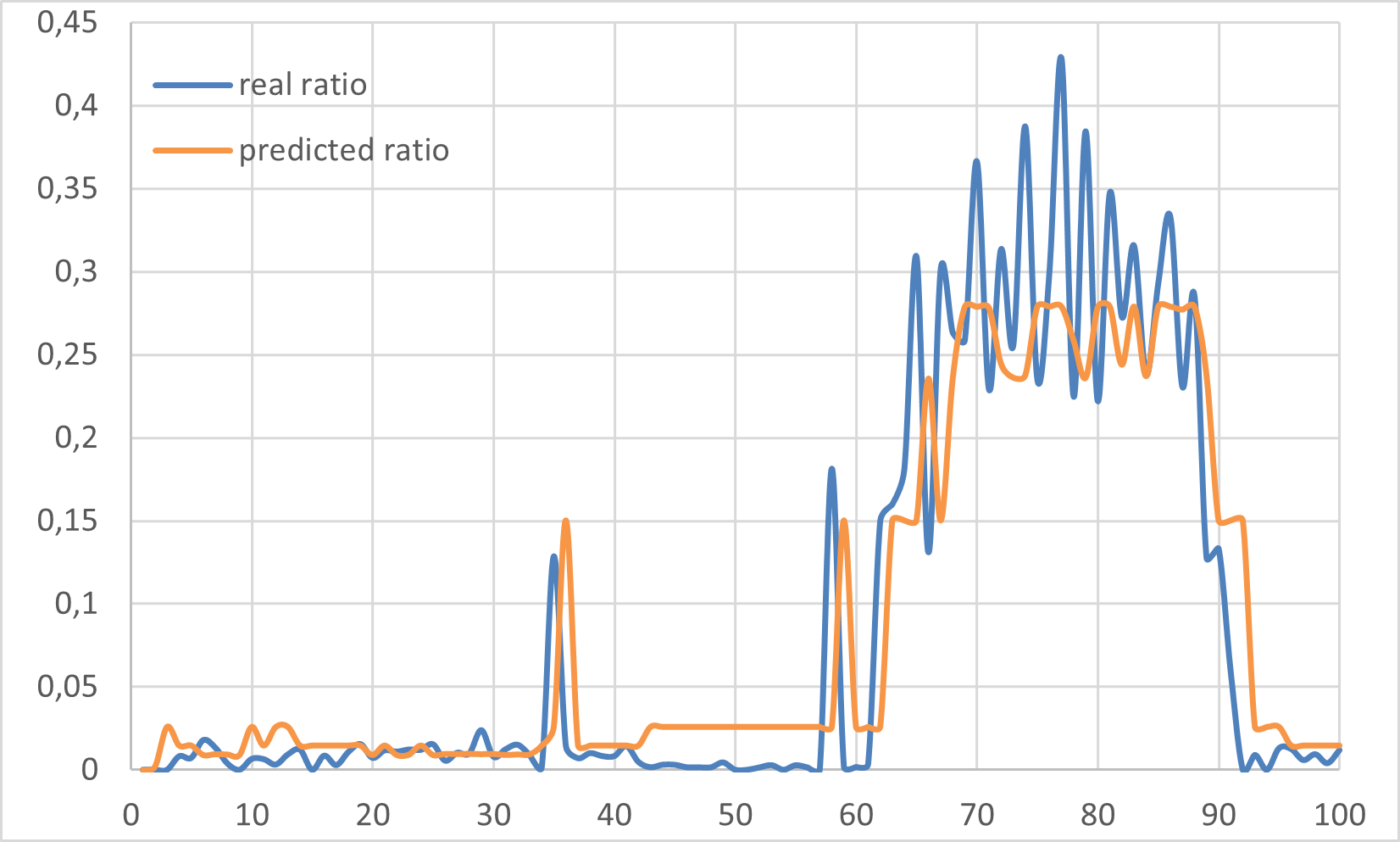} 
      \caption{Real and predicted ratios from A to B}
      \label{fig:FictiveDiscreetRatios01}
   \end{minipage} \hfill 
   \begin{minipage}[c]{0.46\linewidth} 
      \includegraphics[scale=0.5]{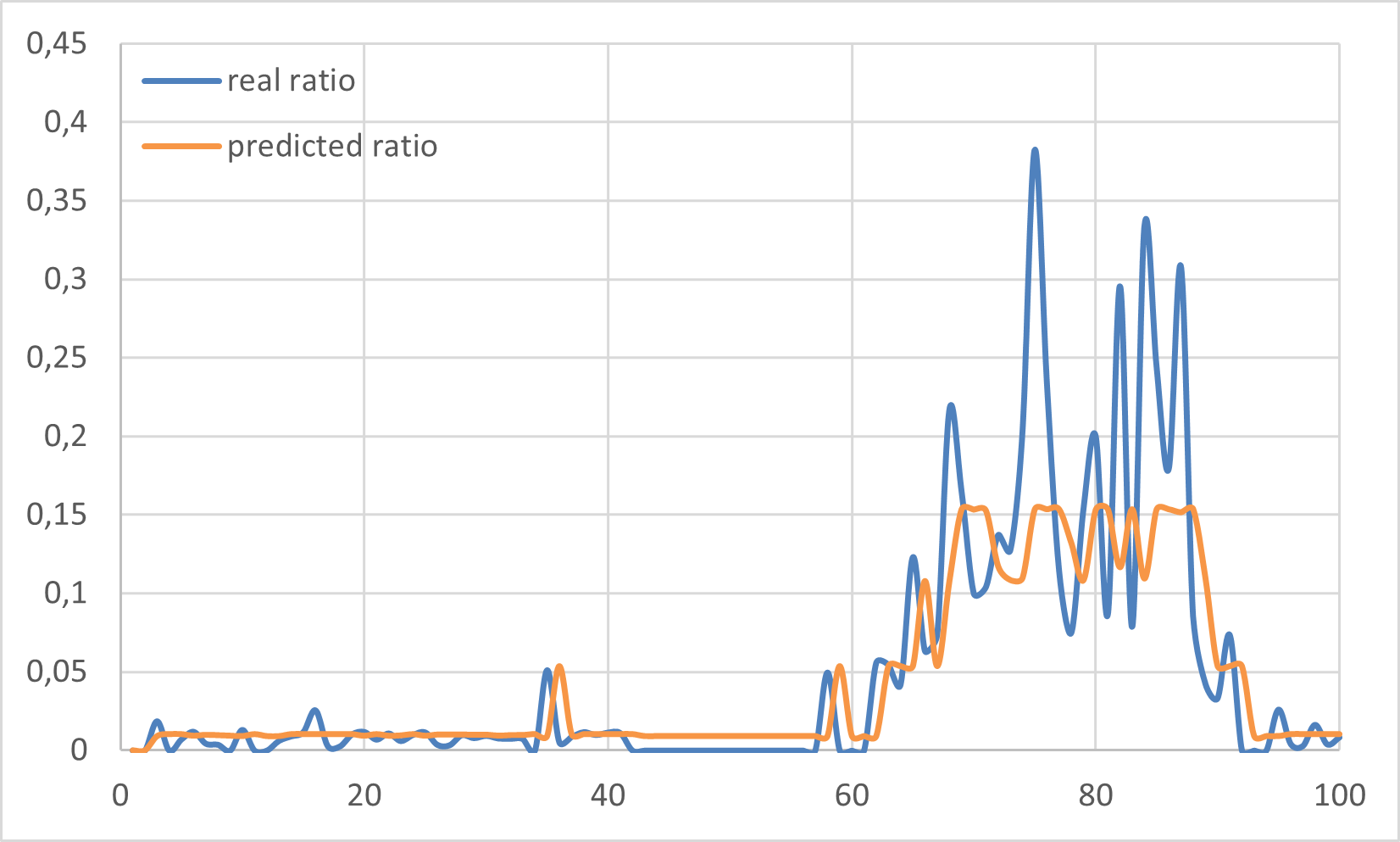} \caption{Real and predicted ratios from A to C}
      \label{fig:FictiveDiscreetRatios02}
   \end{minipage}

   \begin{minipage}[c]{0.46\linewidth} 
      \includegraphics[scale=0.5]{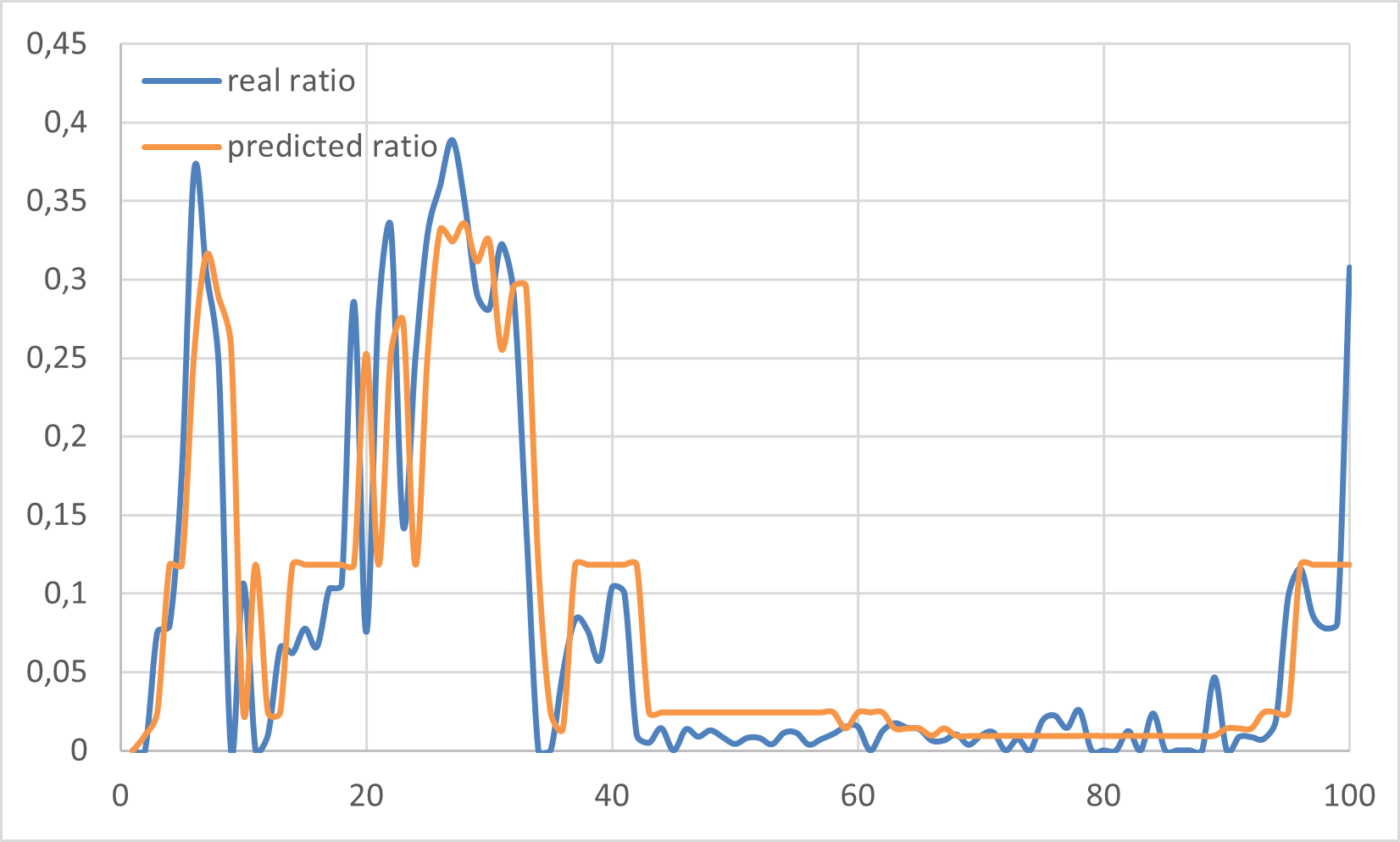} \caption{Real and predicted ratios from B to A}
      \label{fig:FictiveDiscreetRatios10}
   \end{minipage} \hfill 
   \begin{minipage}[c]{0.46\linewidth} 
      \includegraphics[scale=0.5]{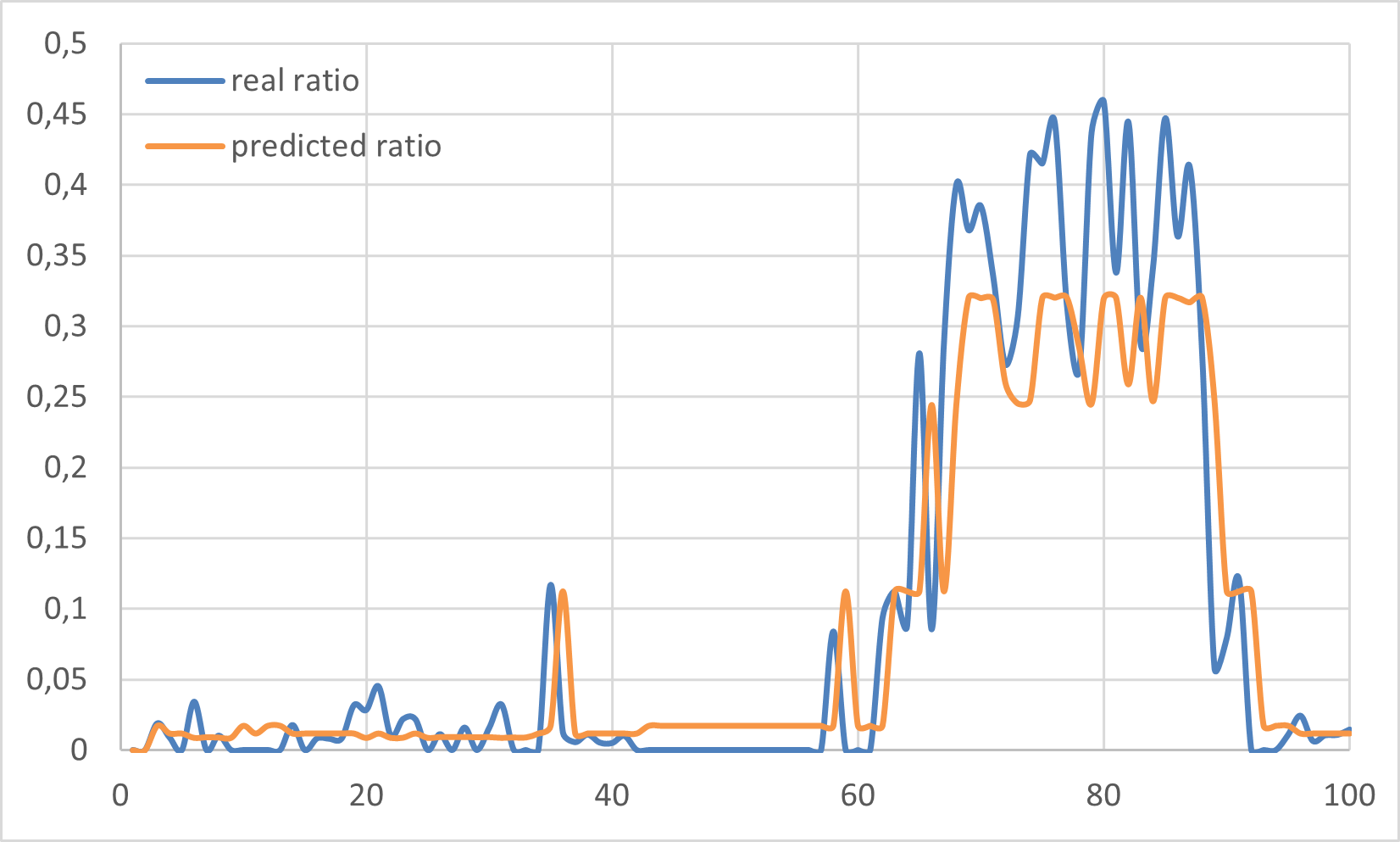} \caption{Real and predicted ratios from B to C}
      \label{fig:FictiveDiscreetRatios12}
   \end{minipage} 
   
   \begin{minipage}[c]{0.46\linewidth} 
      \includegraphics[scale=0.5]{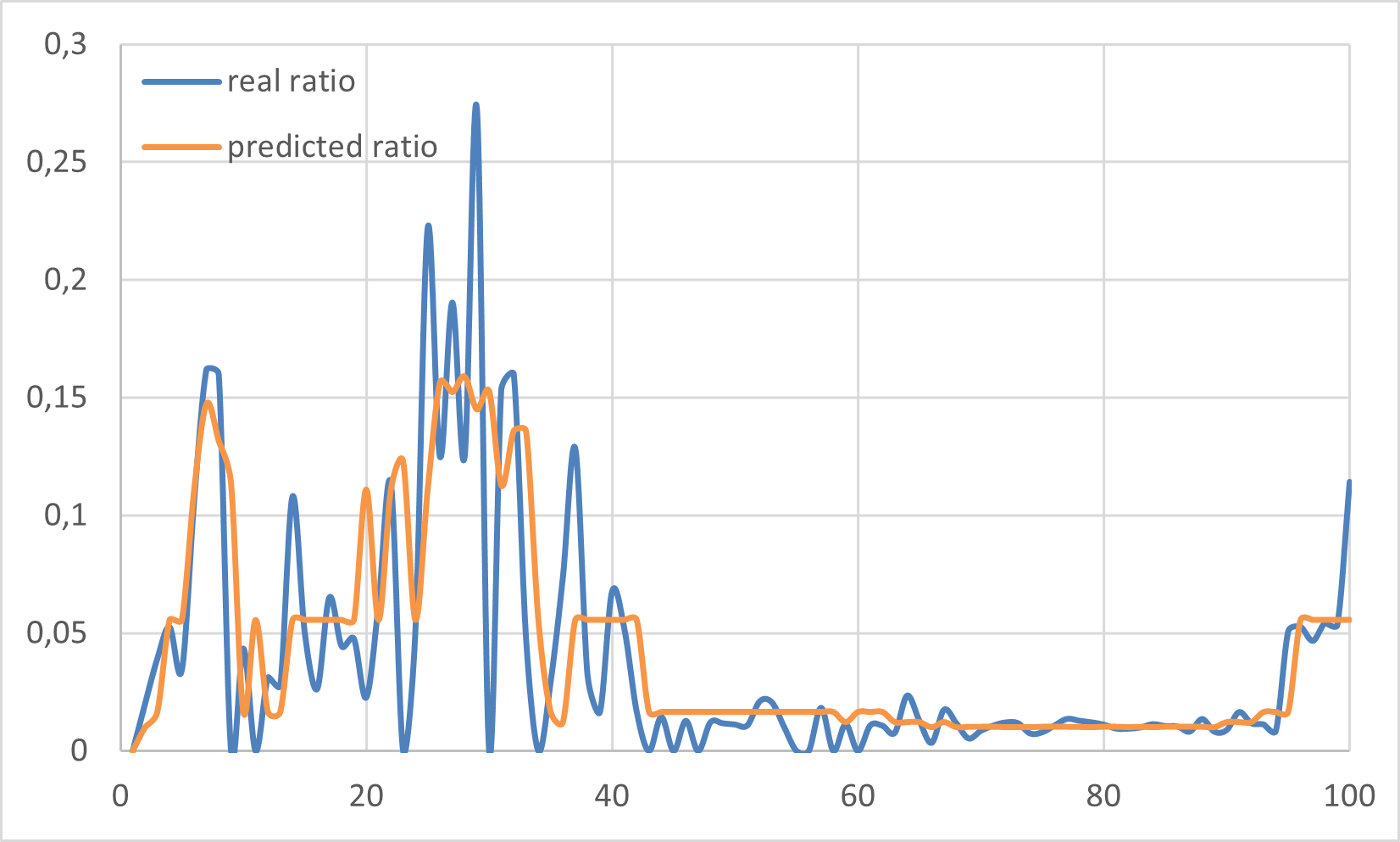} \caption{Real and predicted ratios from C to A}
      \label{fig:FictiveDiscreetRatios20}
   \end{minipage} \hfill 
   \begin{minipage}[c]{0.46\linewidth} 
      \includegraphics[scale=0.5]{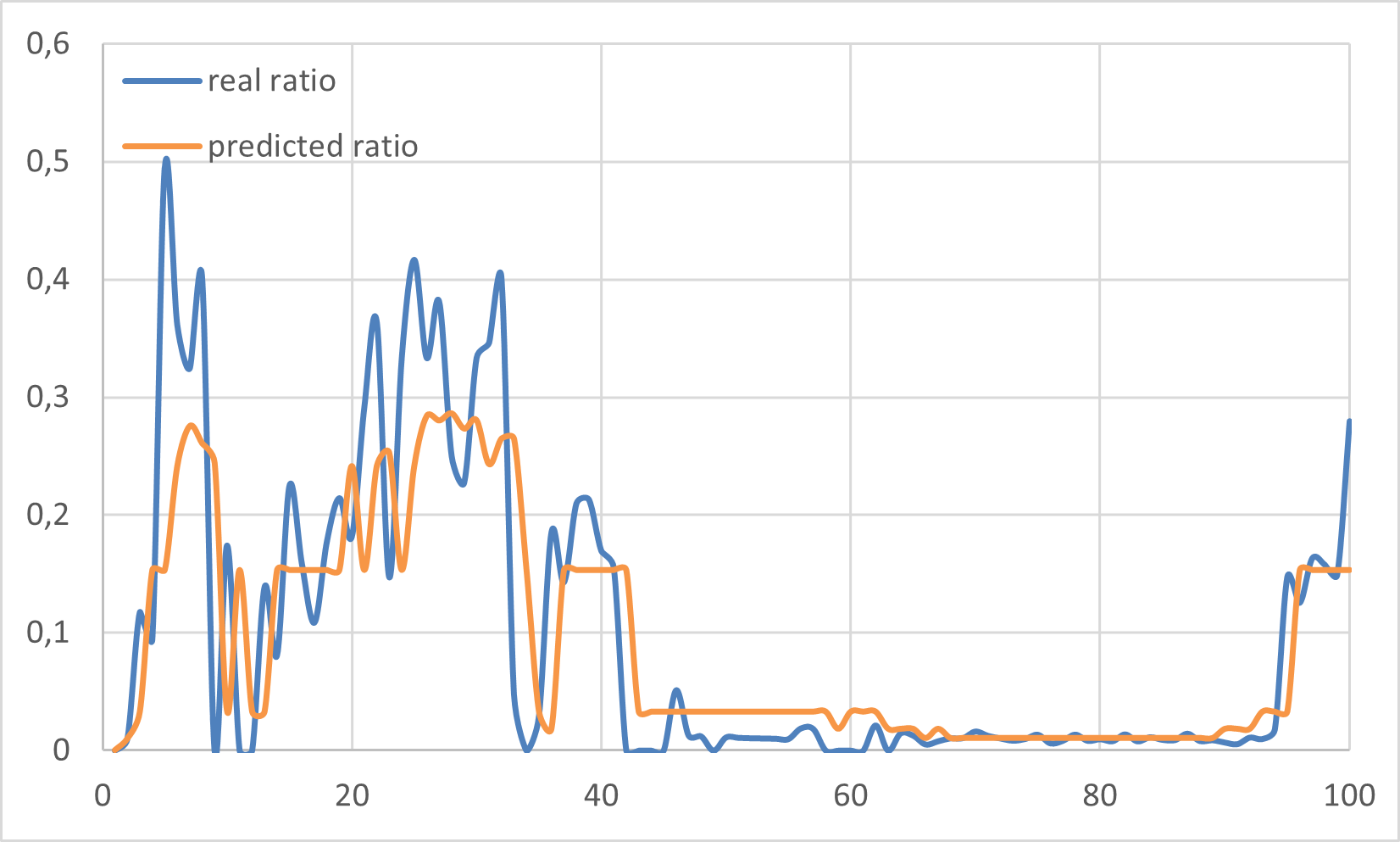} \caption{Real and predicted ratios from C to B}
     \label{fig:FictiveDiscreetRatios21}
   \end{minipage} 
\end{figure}
The results are very encouraging. The filter provides good predictions of future jumps. The predictions vary as a function of the regime cycle. Even when the real ratios sharply increase or decrease, the prediction are immediately corrected. 
\\
Although, one of our previous intuitions is confirmed, the filtering approach can not capture extreme variations since our approach forecasts an average of the rating transition probabilities.  

\subsection{Continuous framework}
\label{ContinuousFictive}
In order to validate the continuous-time filtering approach, we generate a data set using the migration model described in Section \ref{sec:rating-continuous}. The simulated rating processes exhibit no simultaneous jumps. We consider 3 rating categories and 1000 entities per class. 
We build a continuous Markov chain with 5 states. Note that we reduce the number of states compared to the discrete-time framework as the continuous framework is much more computationally demanding. Since the data is fictive and specific to the continuous-time model, this choice has no impact on our validation experiment.
We directly applied the continuous-time filtering approach on the simulated data set, which does not contain simultaneous jumps. Then, in order to challenge the relevance of the use of the continuous model on discrete data, we transform the data set. Jumps are aggregated and randomly spread before filtering as described in Section \ref{seq: continuous-adaptation}. We apply the continuous-time filtering approach and compare the two predicted ratios dynamics. This comparison highlights the effect of the random re-distribution of jumps. 
\\
Figures \ref{fig:without redistribution} and \ref{fig:redistribution} show the dynamics of the proportion of transitions predicted against the real observed ratios from rating A to rating B, respectively without and with redistribution.
\begin{figure}[H]
   \begin{minipage}[c]{0.46\linewidth} 
      \includegraphics[scale=0.5]{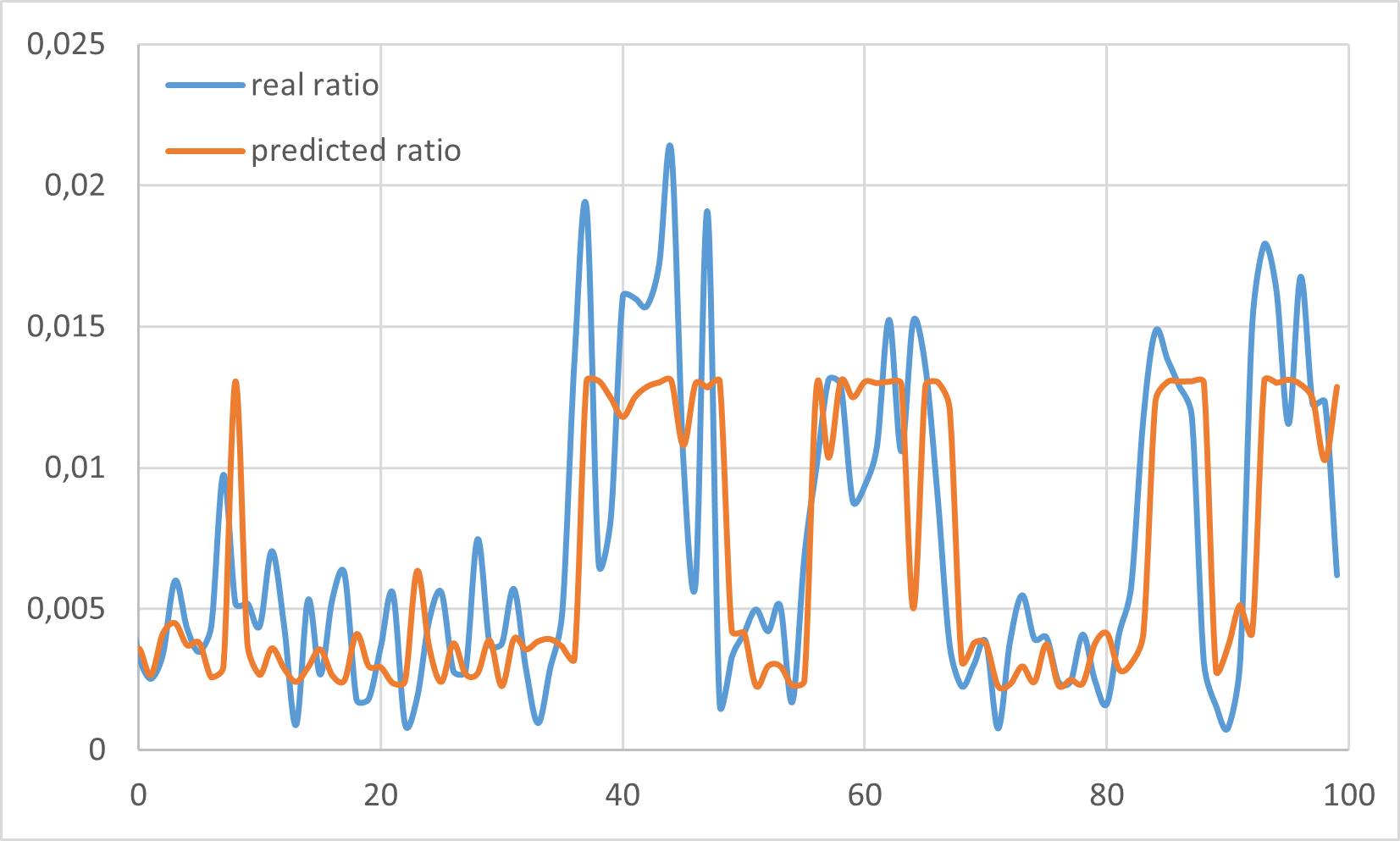} 
      \caption{Real and predicted ratios from A to B without redistribution of common jumps}
      \label{fig:without redistribution}
   \end{minipage} \hfill 
   \begin{minipage}[c]{0.46\linewidth} 
      \includegraphics[scale=0.5]{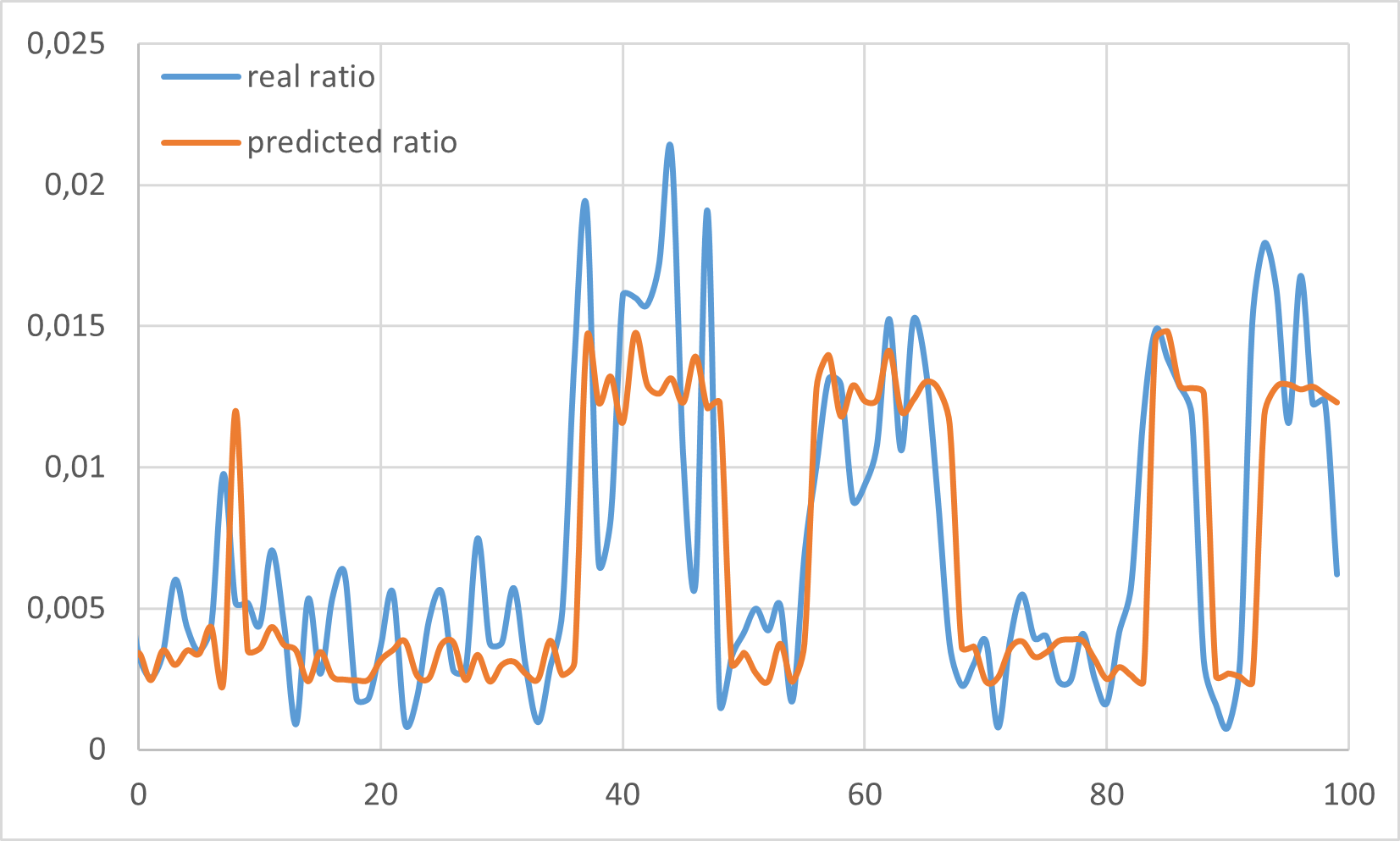} \caption{Real and predicted ratios from A to B with redistribution of common jumps}
      \label{fig:redistribution}
   \end{minipage} 
\end{figure}
The predicted ratios dynamics in Figure \ref{fig:without redistribution}, validate the use of the continuous-time filtering approach: the predicted ratios follow the real trajectory of ratios.
\\
By comparing with Figure \ref{fig:redistribution}, we deduce that spreading information (to avoid the simultaneity of jumps) does not alter the predictions.
Thanks to this comparison exercise, we can apply continuous framework to real data without concern that the results are altered by this action. 
\\
Even if the data samples used are different, we can notice that the changes in both predicted ratios dynamics are less brutal than in the discrete-time filtering framework applied in Section \ref{seq:fictive-discrete-filtering}. The continuous filter is updated with progressive information (due to the absence of simultaneous jumps) and is more flexible than the discrete filter to anticipate regime changes. Assimilating jumps one by one, seems to improve the quality of predictions.
Nevertheless, the effect of delay is still observable.

%% file: RealData.tex
\section{Application on real data}
\label{sec:realData}
This section compares the results of our different models on a real rating database. We consider two discrete-time versions of the filter (one univariate and one multivariate) and a continuous-time multivariate alternative approach.
\subsection{Data Description}
\label{dataDescription}
Credit ratings are forward-looking opinions about the creditworthiness of an obligor with respect to a specific financial obligation.
We build a transitions rating database from Moody's credit rating disclosure. We only use aggregated data (number of transitions). The considered sample contains 7791 days from January 2000 to May 2021. We study the evolution of Long Term ratings of 5030 corporate entities during this period without sector consideration. For specific experiments (analyses, validation, comparison), we consider the whole sample to calibrate the models. For others, such as testing the predictive power of model, we proceed to a cross validation.
We choose a 5 states hidden factor for each experiment.
\\
Moody's rating system relates 21 ratings categories. Keeping this granularity means estimating more than 420 transitions. Therefore, many studies (\cite{elliott2014double}, \cite{koopman2008multi}) reduce the number of rating categories. In the same way, we decide to aggregate the 22 ratings to 6 : A, Baa, Ba, B, C and W. An obligation is rated W when it has no rating. We will also rate W the entity whose rating is not observed. This happens when the data is missing, censored or when it is not appeared yet. There exists many ways to manage not rated status (W). It can be considered as bad information, good information, no information for the credit or not considering them at all. According to \cite{Carty1997}, only few (roughly 13 percent) of the migration to the not rated category are related to changes in credit quality. This argument motivated \cite{nickell2000stability} to use the last method, consisting in removing from the sample all the entities that experiences a not rated status. But this approach is dubious in regard of the loss of information. In this study, we will consider no rated status as censorship. This is achieved by progressively eliminating companies whose rating is not known or withdrawn and adding them when a new rating is provided.
\\
A reference time-step is chosen for each experiment. The daily data are aggregated in order to observe and to predict rating transitions on a larger time window.
\\
\\
\subsection{Discrete-time filtering in sample}

In order to observe and interpret the effect of the discrete-time framework on a real credit rating database, we present in this section, the main results of a univariate and a multivariate filters, calibrated on the whole period.
\subsubsection{Univariate discrete-time filtering}
\label{seq:univariate-results}
In this part, we assume that each transition is governed by its own hidden factor. Under this assumption, each transition evolves according to the evolution of its own latent factor, independently from the others.
This modeling is meaningful to integrate rating specificities in the predictions.
\\
On the data set described above, we focus on a single transition: from rating B to C. We choose this transition because it could be identified as ``transition to default'' and witness of crisis. This will entail the use and calibration of the univariate form of the discrete-time filter (\ref{eq:discreet-univariate-rating}).
\\
A first step consists in calibrating the models with the past history of the involved transition. The reference time step, at stake in every transition, is 30 days. We highlight the efficiency of our approach without cross validation: all past transition history available (from January 2000 to May 2021) is used to calibrate the model.
\\
We obtain in Table \ref{tab:DiscreetTheta}, the calibrated 30 days transition matrix of the hidden factor $\Theta$. Table \ref{tab:DiscreetL} presents the conditional transition probabilities from rating B to C in each state.

\begin{table}[H]
    \centering

    \footnotesize
    \caption{$\Theta$'s transition matrix}
    \begin{tabular}{|*{7}{c|}}
        \hhline{~*{5}{-}}
        \multicolumn{1}{c|}{} & $\Theta=0$ & $\Theta=1$ & $\Theta=2$ & $\Theta=3$ & $\Theta=4$   \\ \hline
        $\Theta=0$ & {0.90598} & {0.074109} & {0.018316} & {0} & {0.001595}    \\ \hline
        $\Theta=1$ & {0.230415} & {0.715919} & {0.040626} & {0} & {0.013040}     \\ \hline
        $\Theta=2$ & {0.000304} & {0.381375} & {0.540412} & {0.077909} & {0}     \\ \hline
        $\Theta=3$ & {0} & {0} & {0.740452} & {0.259548} & {0}     \\ \hline
        $\Theta=4$ & {0.491597} & {0} & {0} & {0.508403} & {0}    \\ \hline

    \end{tabular}
     \label{tab:DiscreetTheta}
\end{table}

\begin{table}[H]
    \centering
    \footnotesize
     \caption{30 days transition probabilities from B to C}
    \begin{tabular}{|c|c|c|c|c|c|}
        \hhline{~*{5}{-}}
        \multicolumn{1}{c|}{} & $\Theta=0$ & $\Theta=1$ & $\Theta=2$ & $\Theta=3$ & $\Theta=4$  \\ \hline
        $B \rightarrow C$ & {0.001814} & {0.0050001} & {0.0158818} & {0.0451715}
 & {0.085771}
   \\ \hline
    \end{tabular}
    \label{tab:DiscreetL}
\end{table}
Table \ref{tab:DiscreetTheta} highlights two stable states, 0 and 1 and an unstable and rare state, state 4. By analysing Table \ref{tab:DiscreetL}, we notice a hierarchy of risk between the states of $\Theta$. State 4 is clearly identified as the riskiest state with a downgrade probability fifty time greater than in state 0, the most favourable state. State 3 is also a state of crisis which is more stable. State 2 can be interpreted as an intermediate state between favourable and unfavourable situation. Consequently we can expect that the economy often remains in a calm and favourable situation and experiences sometimes brief transitions to stressed states when downgrade probability B to C increases a lot.
\\
Figure \ref{fig:univTheta} presents the filtered indicator function trajectories of the own hidden factor of the transition B to C, $\{\hat{I}_n^h, n\in\{0,\ldots,\Gamma\},\ h\in\T \}$, without cross validation. Figure \ref{fig:univRatios} shows the dynamics of 30 days forecasted ratios from rating B to C, $\{\hat{\nu}_{n-1}^{BC}, \ n \in \{1,\ldots,\Gamma\}\}$, given in (\ref{prediction}).
\begin{figure}[H]
   \begin{minipage}[c]{0.46\linewidth}
      \includegraphics[scale=0.5]{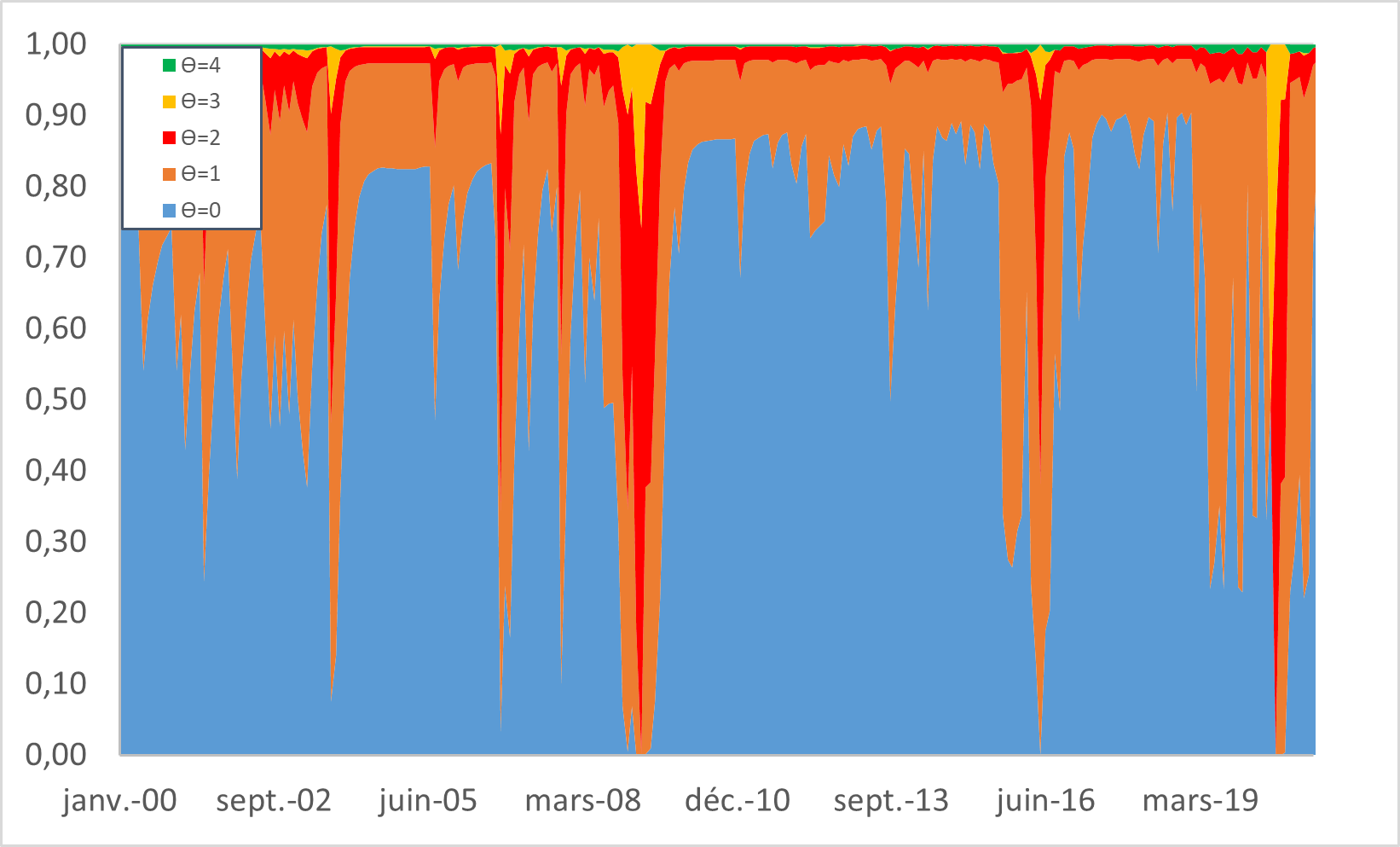}
      \caption{Filtered trajectories of the hidden factor indicator functions}
      \label{fig:univTheta}
   \end{minipage} \hfill
   \begin{minipage}[c]{0.46\linewidth}
      \includegraphics[scale=0.5]{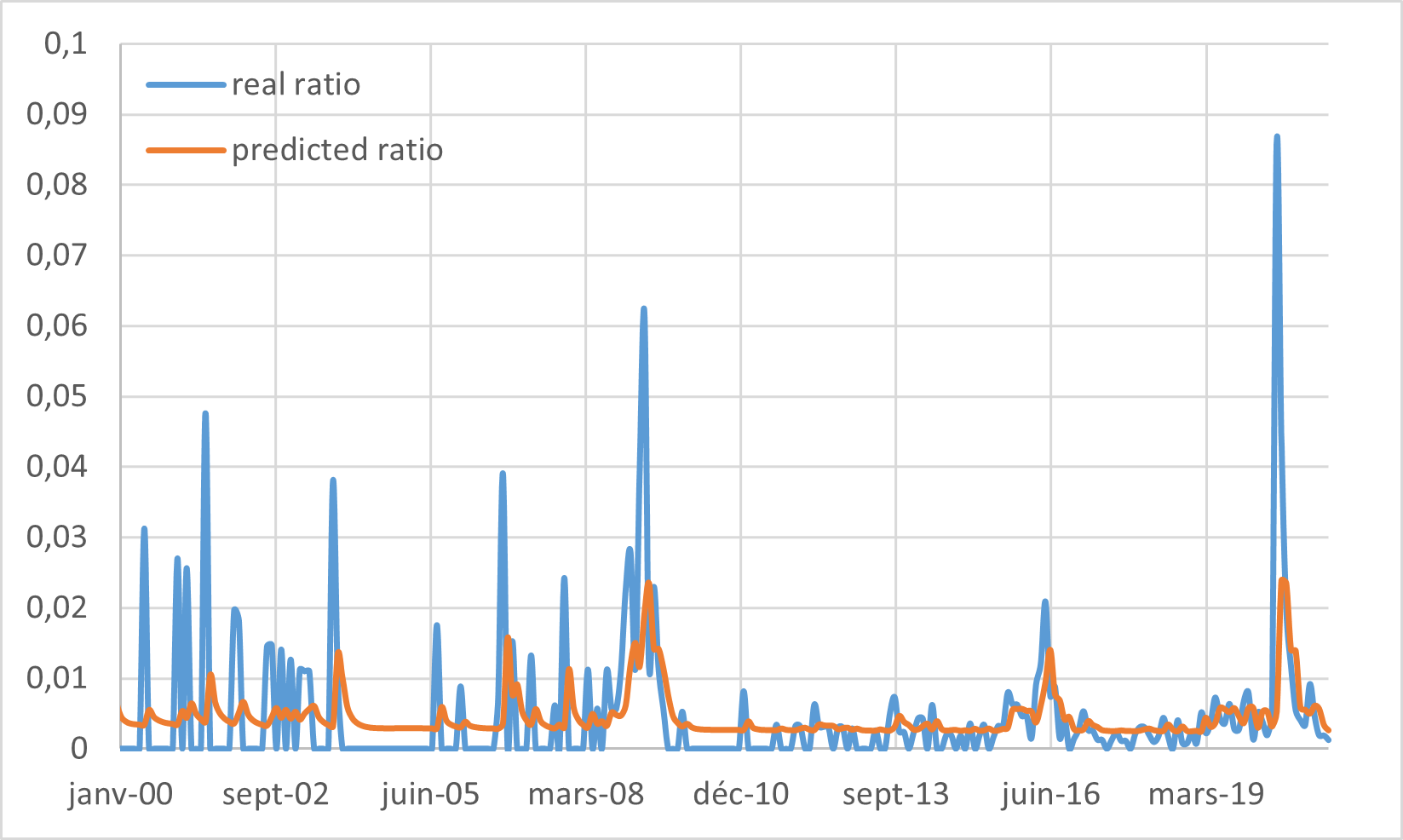} \caption{Real and predicted ratios for transition B to C}
    \label{fig:univRatios}
   \end{minipage}
\end{figure}
Figure \ref{fig:univTheta} shows that the dominant state changes across time and highlights regime switching. Our intuitions are confirmed, the filter is often close to favorable states 0 and 1. The dominant state is sometimes, for a brief moment, state 2, an intermediate state, where the downgrade probability from B to C increases. After periods when state 2 is dominant, the filter sometimes indicates that a state of true crisis, state 3, becomes dominant. Transitions from periods where state 0 or 1 are dominant to periods where state 4 is dominant may be sudden but remain rare. Fortunately this state of extreme ``crisis'' is only dominant for very brief periods. By analyzing Figure \ref{fig:univRatios}, it can be noted that the predicted ratios from B to C reflect the general trend of real ratios with the same "lag" effect observed than on fictive data. The filter is able to detect regimes and transition phases but cannot capture brutal and short transitions. Finally the filter infers that the economic cycle experiences long periods of favorable situations and brief transitions to stress states.
\\
Note that the hidden factor is specific to the involved transition. It may cover systematic risk but also the risk which might be specific to the ratings at stake.
\\
We now consider the multivariate case where the hidden factor is shared by several transitions.

\subsubsection{Multivariate discrete-time filtering}
Using multiple transitions to infer the hidden factor assumes that the later is shared by those transitions. This approach should bring more information to forecast the dynamics of these transitions but presents several difficulties. The calibration algorithm finds centroids in the parameters space which might be far from each other due to the high dimension of the parameters space. Consequently the predicted number of transitions may be very different from the realized one. Furthermore rating transition events may not be sufficiently correlated. Indeed certain transitions are weakly correlated and might bring noise. We must only consider the most correlated transitions to extract the global factor dynamics. Therefore we decide to only focus on adjacent downgrade transitions (the upper diagonal). Indeed empirical results from \cite{Reda2015} show that the upgrades are more subject to idiosyncratic shocks than downgrades. To remove the impact of the remaining transitions on the model, we assign them the same probability for each state of the hidden factor: we use the time-homogeneous intensity estimators to compute these probabilities (see, e.g., \cite{Reda2015}, \cite{duffie2007multi}, \cite{jarrow1997markov}, \cite{koopman2008multi}, \cite{lando2002analyzing}). Consequently we reduce the number of transitions to calibrate to four.
\\
We achieve two experiments. First we consider a time step reference of 30 days. We calibrate on whole period of the data set to observe the behaviour of the multivariate model. Then, along a second experiment, we will proceed to a cross validation to faithfully assess the predictive power of the model. For this experiment which is computationally more expensive, we will choose a larger time window, with a time step of 50 days.
\\
\\
For the first experiment, as in Section \ref{seq:univariate-results}, we again consider 5 states for the hidden factor, a time step of 30 days and we do not proceed to cross validation.
\\
Table \ref{tab:multDiscreetbackTheta} gives the calibrated transition matrix of the hidden factor. Table \ref{tab:downgrade} presents the conditional downgrade probabilities for a time step of 30 days.
\begin{table}[H]
    \centering
    \footnotesize
    \caption{$\Theta$'s transition matrix}
    \begin{tabular}{|*{7}{c|}}
        \hhline{~*{5}{-}}
        \multicolumn{1}{c|}{} & $\Theta=0$ & $\Theta=1$ & $\Theta=2$ & $\Theta=3$ & $\Theta=4$   \\ \hline
        $\Theta=0$ & {0.9499} & {0.0418} & {0.0010} & {0} & {0.0073}    \\ \hline
        $\Theta=1$ & {0.1075} & {0.7661} & {0.1264} & {0} & {0}     \\ \hline
        $\Theta=2$ & {0.0004} & {0.2685} & {0.6340} & {0.0503} & {0.0469}     \\ \hline
        $\Theta=3$ & {0} & {0} & {0.5133} & {0.4867} & {0}     \\ \hline
        $\Theta=4$ & {0} & {0} & {1} & {0} & {0}    \\ \hline

    \end{tabular}
    \label{tab:multDiscreetbackTheta}
\end{table}

\begin{table}[H]
    \centering
    \caption{Adjacent 30 days downgrade probabilities}
    \begin{tabular}{|*{7}{c|}}
        \hhline{~*{5}{-}}
        \multicolumn{1}{c|}{} & $\Theta=0$ & $\Theta=1$ & $\Theta=2$ & $\Theta=3$ & $\Theta=4$   \\ \hline
        $A \rightarrow Baa$ & {0.00297589} & {0.00224944} & {0.00838262} & {0.00801885} & {0.0194804}    \\ \hline
        $Baa \rightarrow Ba$ & {0.00125687} & {0.00146192} & {0.00492593} & {0.00985172} & {0.031583}     \\ \hline
        $Ba \rightarrow B$ & {0.00326413} & {0.00633207} & {0.0150595} & {0.0282736} & {0.0228716}     \\ \hline
        $B \rightarrow C$ & {0.00189228} & {0.00492691} & {0.0128149} & {0.0641203} & {0.0114155}     \\ \hline

    \end{tabular}
     \label{tab:downgrade}
    \footnotesize
\end{table}
By analysing the tables, it is noteworthy that states 0 and 1 are stable states which induce a ``favourable'' situation, where downgrade probabilities are quite low. States 3 and 4 can be interpreted as a stressed economy, where downgrade probabilities are higher. Note that state 4 is totally unstable and transitory. The transition between favourable periods (state 0 and 1) and stable stressed periods (state 3) is exclusively achieved through state 2.
\\
Figure \ref{fig:multDiscreetTheta} shows the filtered trajectories of state probabilities according to (\ref{eq:multivariate-discreet-credit}). Figure \ref{multDiscreetRatio} presents the dynamics of the predicted ratio from rating B to C, within a multivariate framework, without cross validation. We focus on transition B to C to compare with Section \ref{seq:univariate-results}.
\begin{figure}[H]
   \begin{minipage}[c]{0.46\linewidth}
      \includegraphics[scale=0.5]{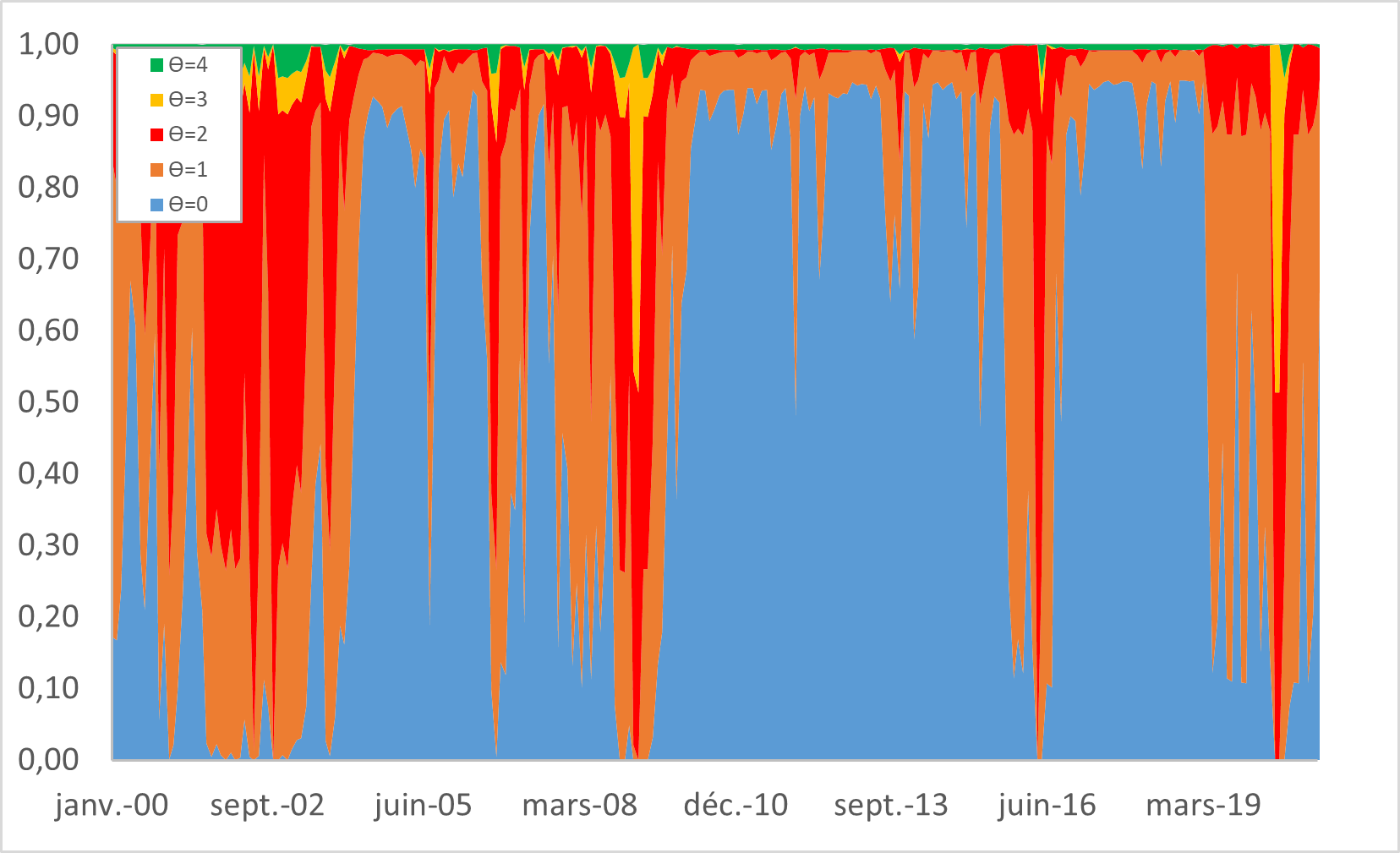}
      \caption{Filtered trajectories of hidden factor indicator functions}
      \label{fig:multDiscreetTheta}
   \end{minipage} \hfill
   \begin{minipage}[c]{0.46\linewidth}
      \includegraphics[scale=0.5]{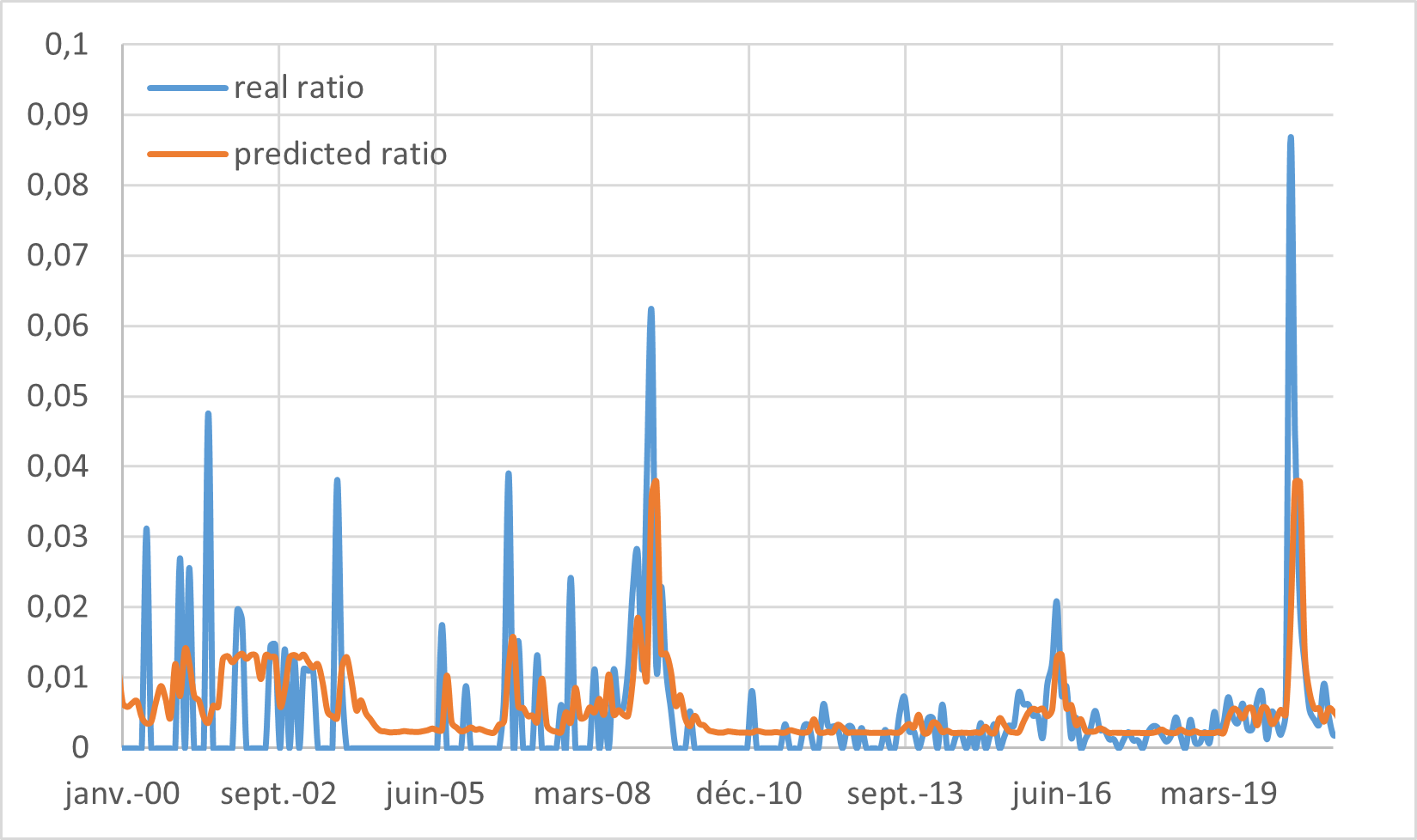} \caption{Real and predicted ratios for 30 days transition from B to C}
      \label{multDiscreetRatio}
   \end{minipage}
 \end{figure}
Figure \ref{fig:multDiscreetTheta} brings us new information on the evolution of the predicted hidden state. Periods of crisis when state 3 and 4 dominant, are pretty rare and brief. By analyzing Figure \ref{multDiscreetRatio}, we can first notice that the multivariate framework is also a good predictor. The forecasted transition ratios follow the trend of observed ratios and fit with different regimes. Comparing with the univariate case (see Figure \ref{fig:univTheta}), the multivariate model seems to be more sensitive to events: the multivariate model better captures the crisis of sep-2000 compared to the univariate model. The forecasted rating transition B to C is not only based on its own past evolution but also stem from the history of others.
\subsection{Comparison of the filters out of sample: annual recalibration}
\label{outOfSample}
We use a cross-validation approach to assess the predictive power of the multivariate models both in the continuous-time and discrete-time frameworks. To this end, we use data from 2000 to 2008 to perform a first calibration and to initialize our parameters. Then, from January 2008 to may 2021, we predict the dynamics of the 50 days transition rates. The model is re-calibrated every year, integrating the new observations of the last year. Note that we changed the reference time step to 50 days for a sake of computational speed.
\\
Note also that since we re-calibrate the model yearly, parameters and states structure vary over time.
\subsubsection{Multivariate continuous-time filtering}
\label{outOfSample-continuous}
In this section, we apply the continuous filtering framework, presented in Section \ref{sec:rating-continuous} and its adaptations, described in Section \ref{seq: continuous-adaptation}, to real data. We choose a reference time step equal to 50 days. The real and predicted 50 days rating transition ratios are presented in Figures \ref{fig:continuousRatiosABaa}, \ref{fig:continuousRatiosBaaBa}, \ref{fig:continuousRatiosBaB} and \ref{fig:continuousRatiosBC}.
\\
\begin{figure}[H]
   \begin{minipage}[c]{0.46\linewidth}
      \includegraphics[scale=0.5]{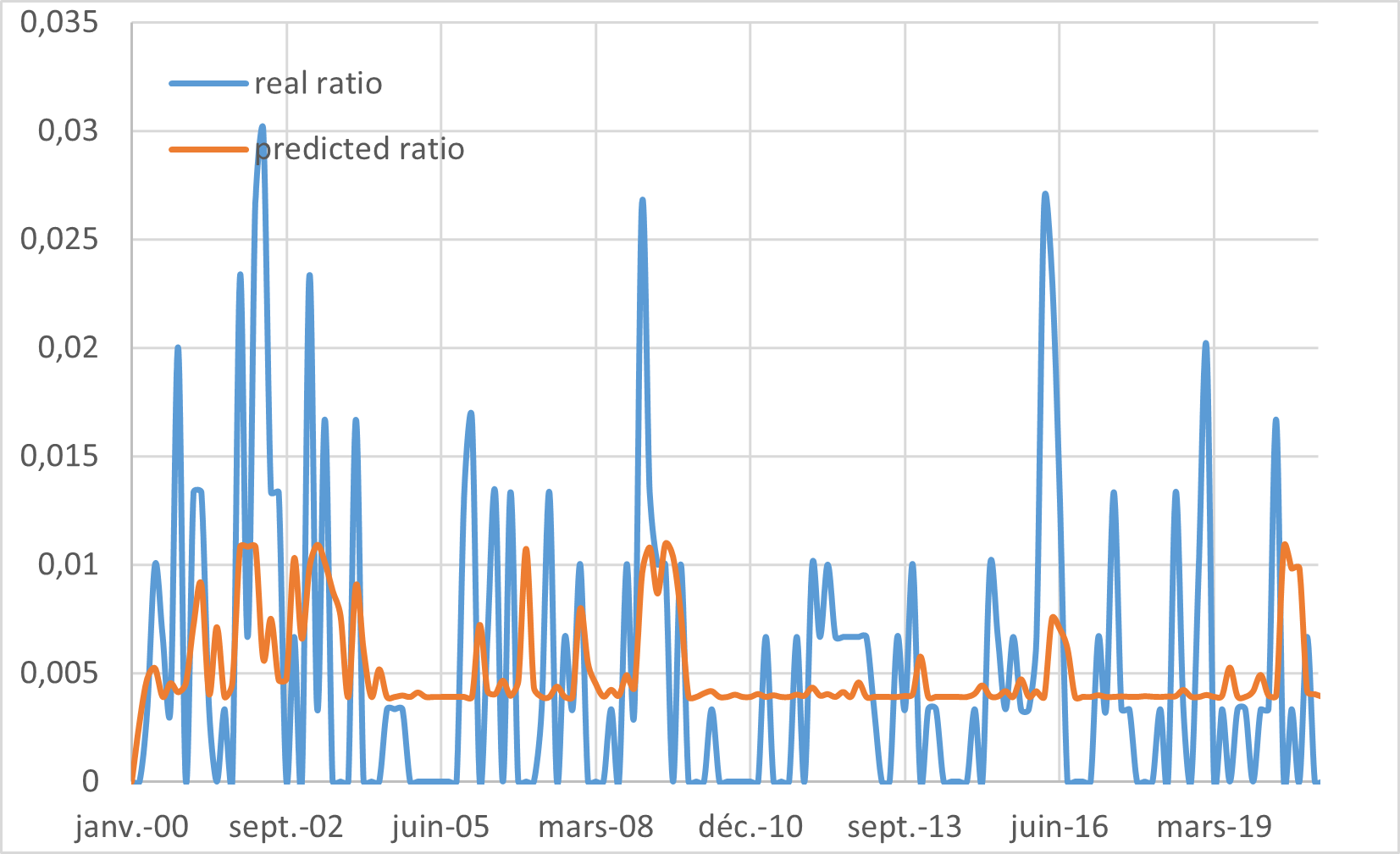}
      \caption{Real and predicted ratios 50 days transition from A to Baa}
      \label{fig:continuousRatiosABaa}
   \end{minipage} \hfill
   \begin{minipage}[c]{0.46\linewidth}
      \includegraphics[scale=0.5]{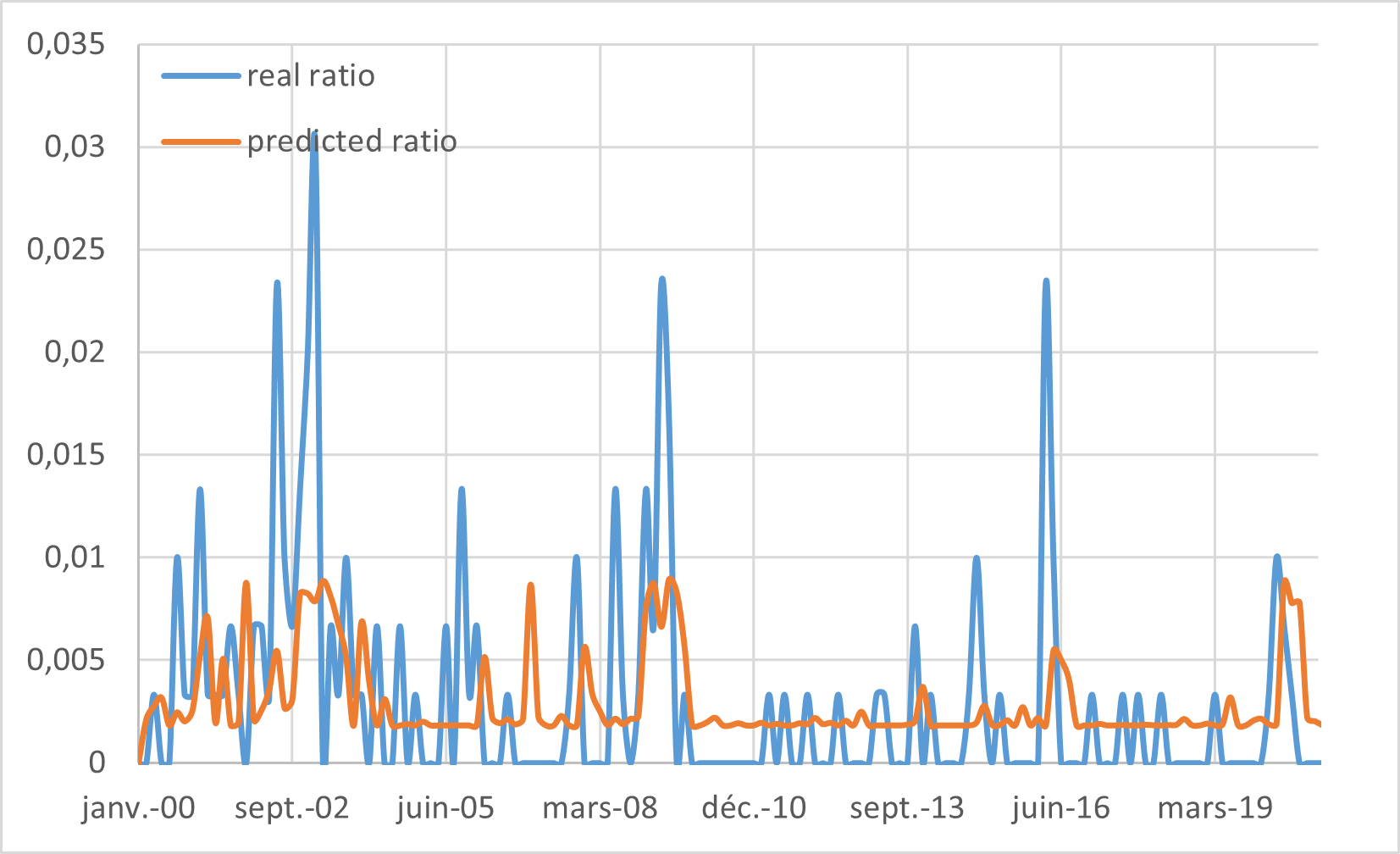} \caption{Real and predicted ratios 50 days transition from Baa to Ba}
      \label{fig:continuousRatiosBaaBa}
   \end{minipage}

   \begin{minipage}[c]{0.46\linewidth}
      \includegraphics[scale=0.5]{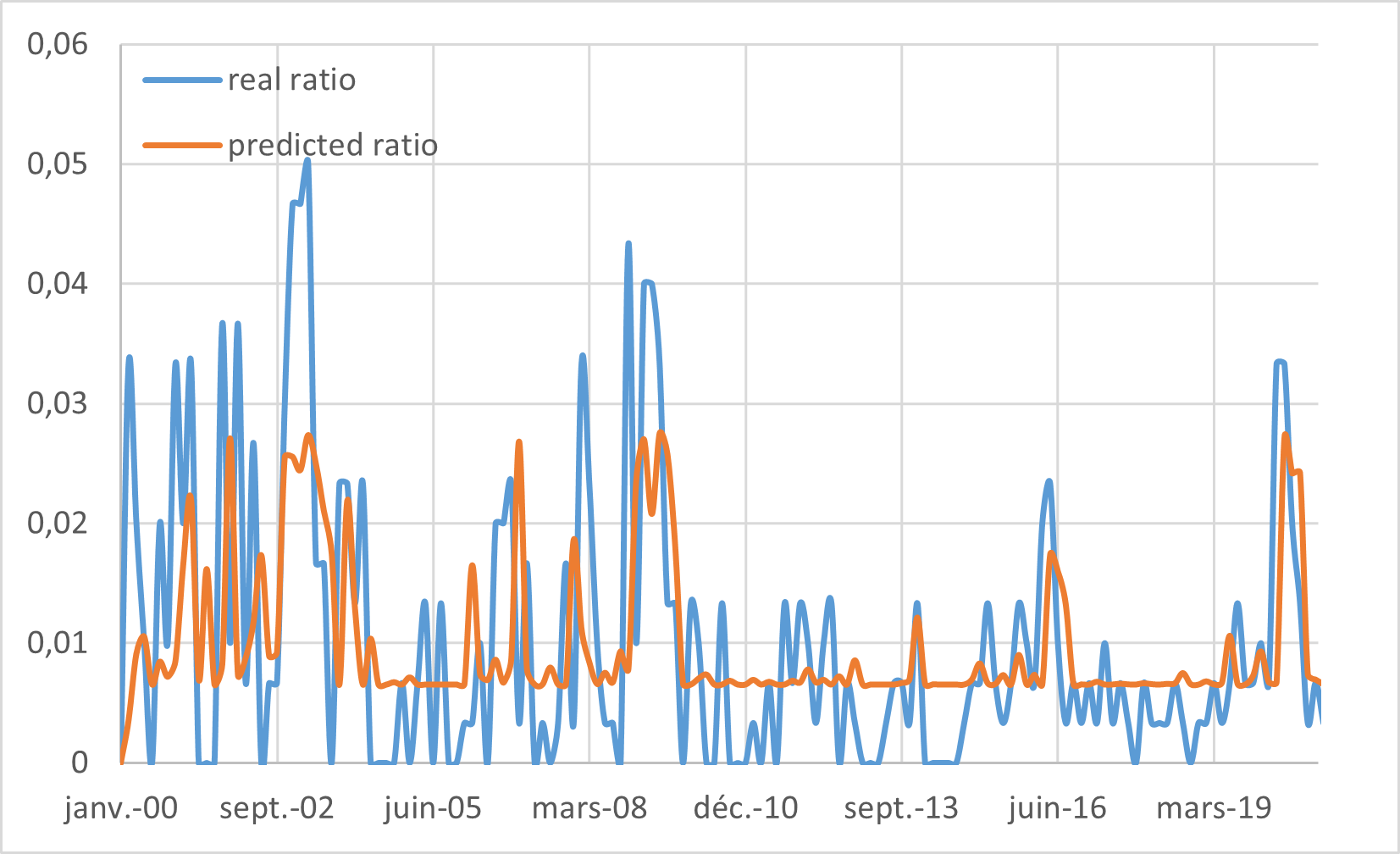} \caption{Real and predicted ratios 50 days transition from Ba to B}
      \label{fig:continuousRatiosBaB}
   \end{minipage} \hfill
   \begin{minipage}[c]{0.46\linewidth}
      \includegraphics[scale=0.5]{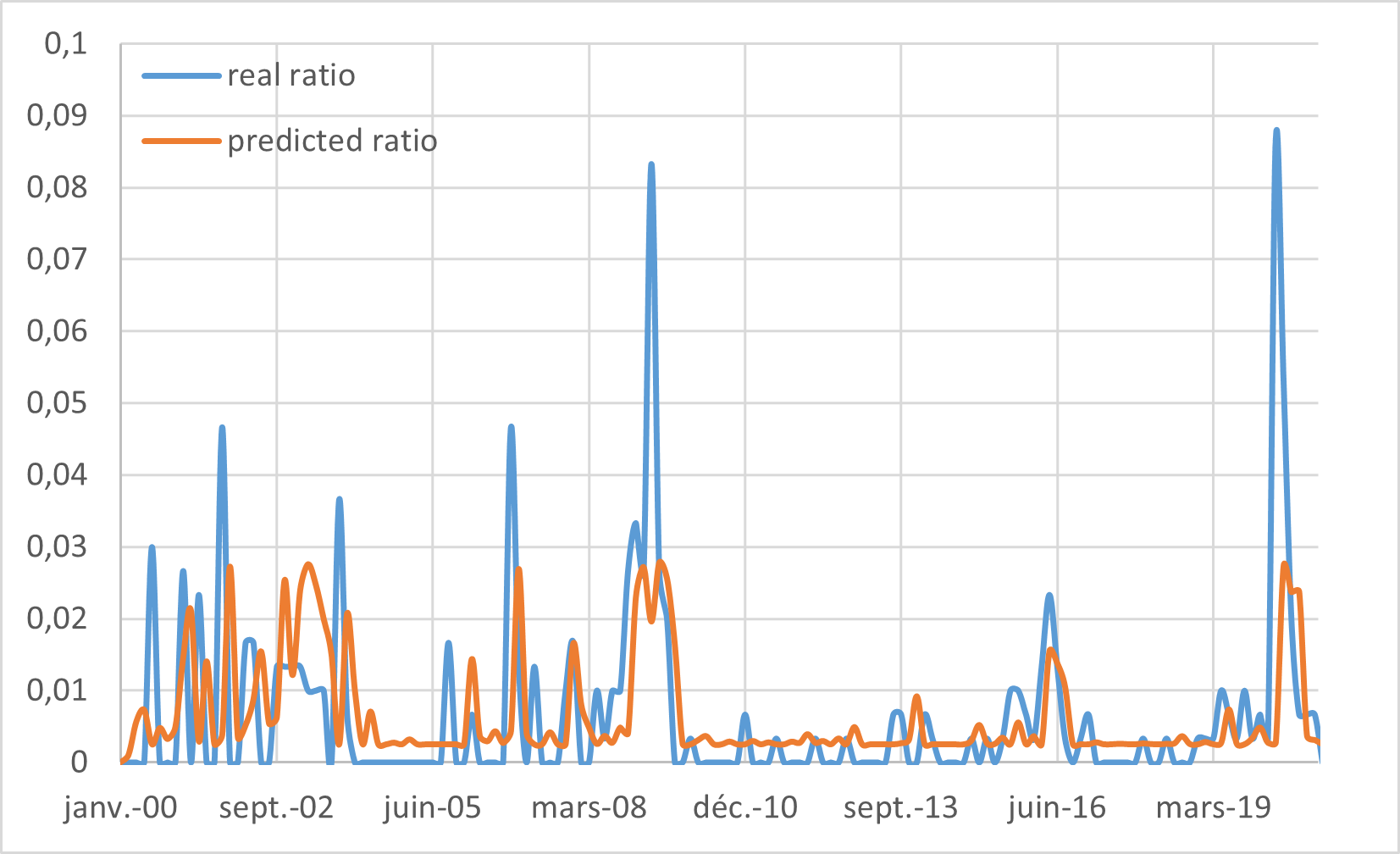} \caption{Real and predicted ratios 50 days transition from B to C}
      \label{fig:continuousRatiosBC}
   \end{minipage}
\end{figure}
\subsubsection{Multivariate discrete-time filtering}
Here, we apply the discrete-time filtering framework, presented in Section \ref{sec:rating-discrete} to real data. The model is applied on the same sample used for the continuous-time filtering approach, with annual recalibration as in Section \ref{outOfSample-continuous}. We keep a reference time step equal to 50 days.
Figures \ref{fig: MultDiscreetBackRatiosABaa}, \ref{fig: MultDiscreetBackRatiosBaaBa}, \ref{fig: MultDiscreetBackRatiosBaB}, \ref{fig: MultDiscreetBackRatiosBC} compare the dynamics of predicted transition ratios to observed one.
\begin{figure}[H]
   \begin{minipage}[c]{0.46\linewidth}
      \includegraphics[scale=0.5]{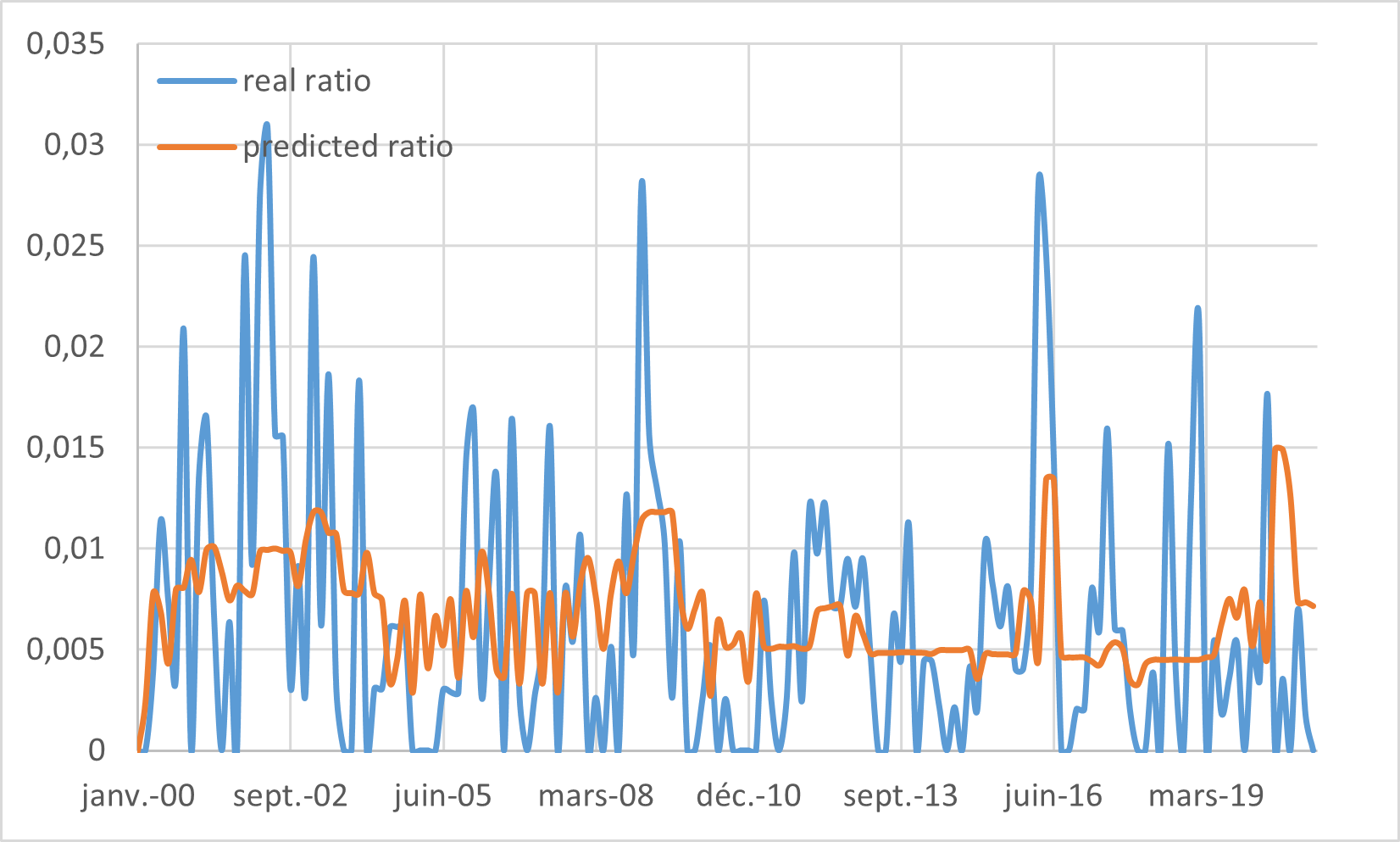}
      \caption{Real and predicted 50 days-transition ratios from A to Baa}
      \label{fig: MultDiscreetBackRatiosABaa}
   \end{minipage} \hfill
   \begin{minipage}[c]{0.46\linewidth}
      \includegraphics[scale=0.5]{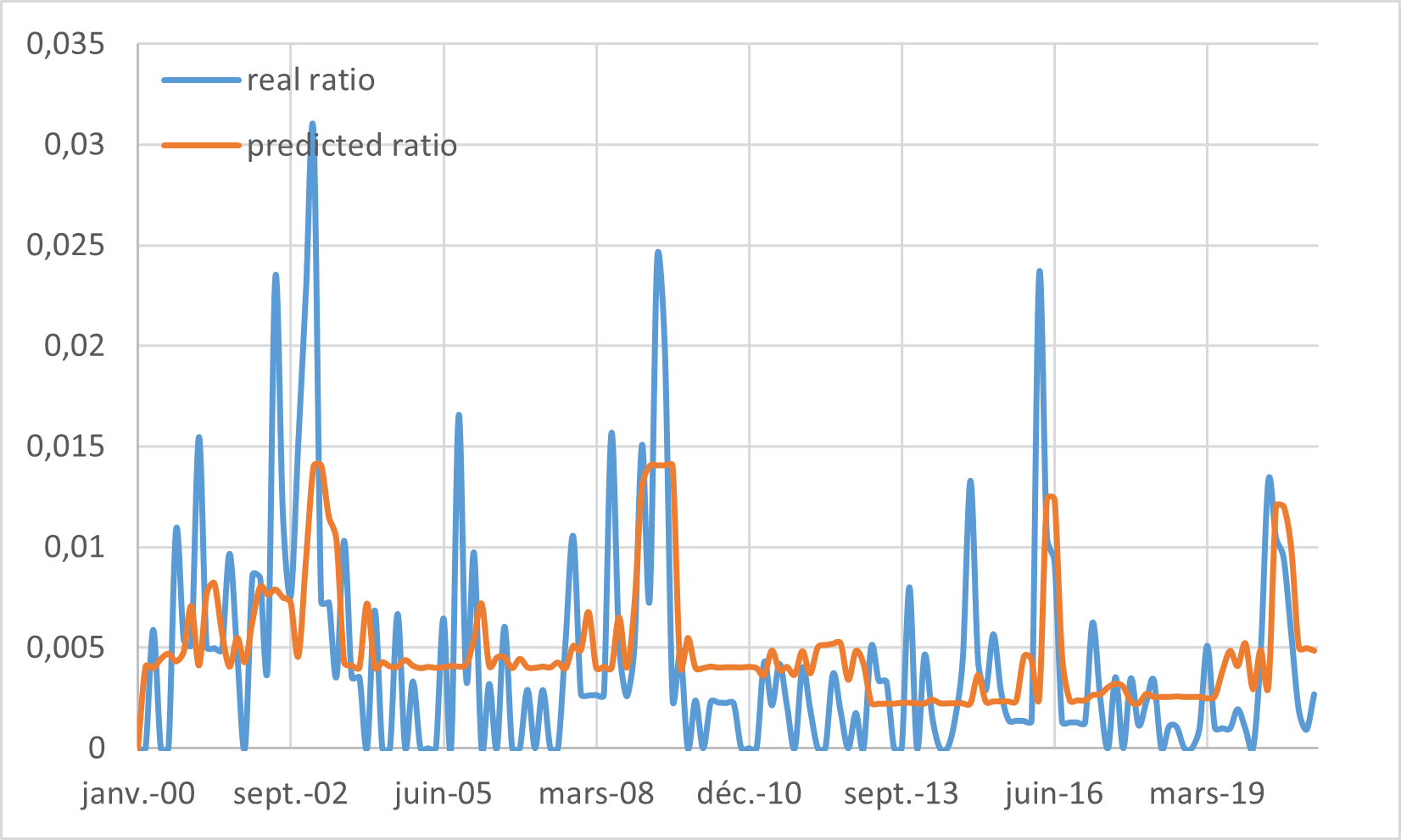} \caption{Real and predicted 50 days-transition ratios from Baa to Ba}
      \label{fig: MultDiscreetBackRatiosBaaBa}
   \end{minipage}

   \begin{minipage}[c]{0.46\linewidth}
      \includegraphics[scale=0.5]{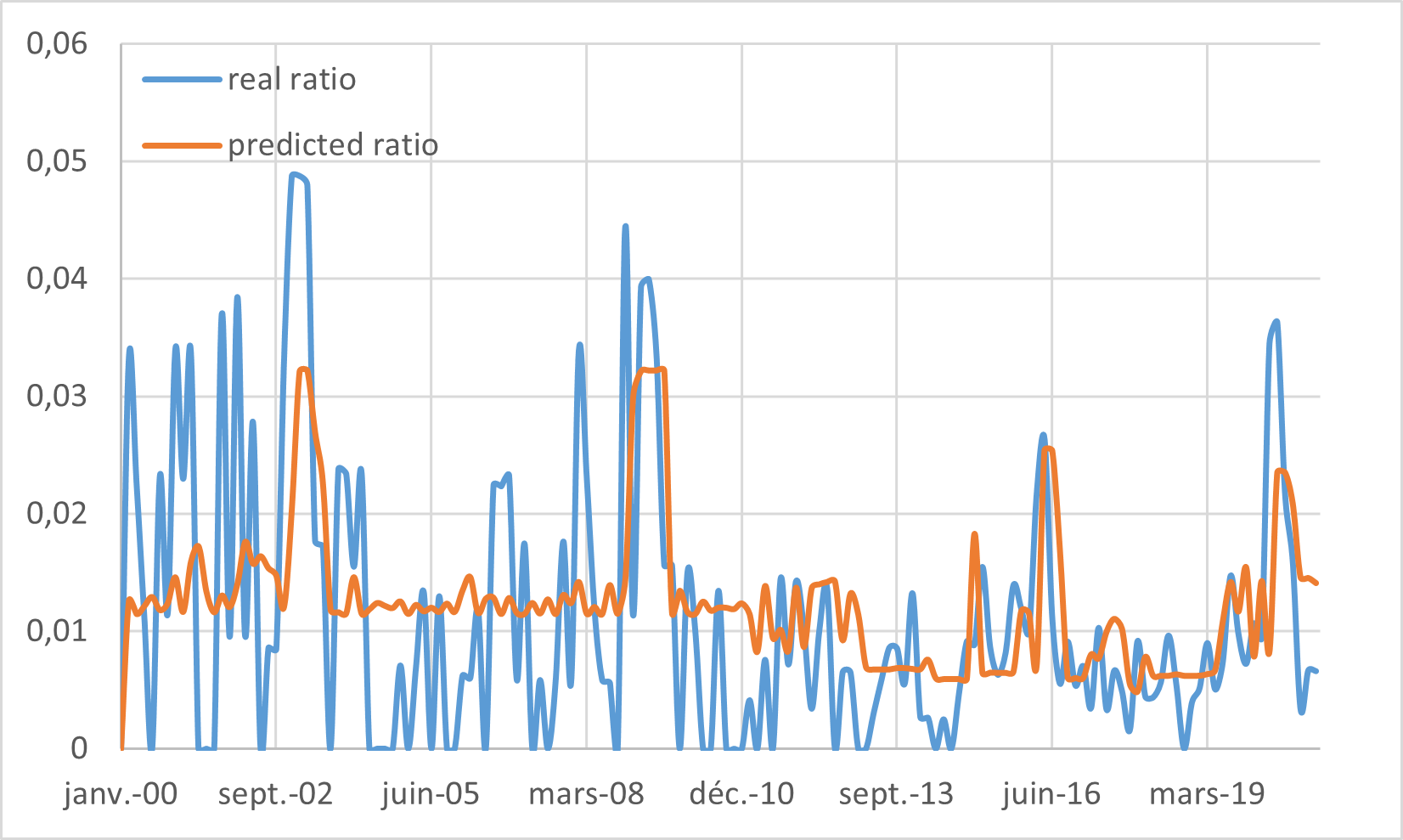} \caption{Real and predicted 50 days-transition ratios from Ba to B}
      \label{fig: MultDiscreetBackRatiosBaB}
   \end{minipage} \hfill
   \begin{minipage}[c]{0.46\linewidth}
      \includegraphics[scale=0.5]{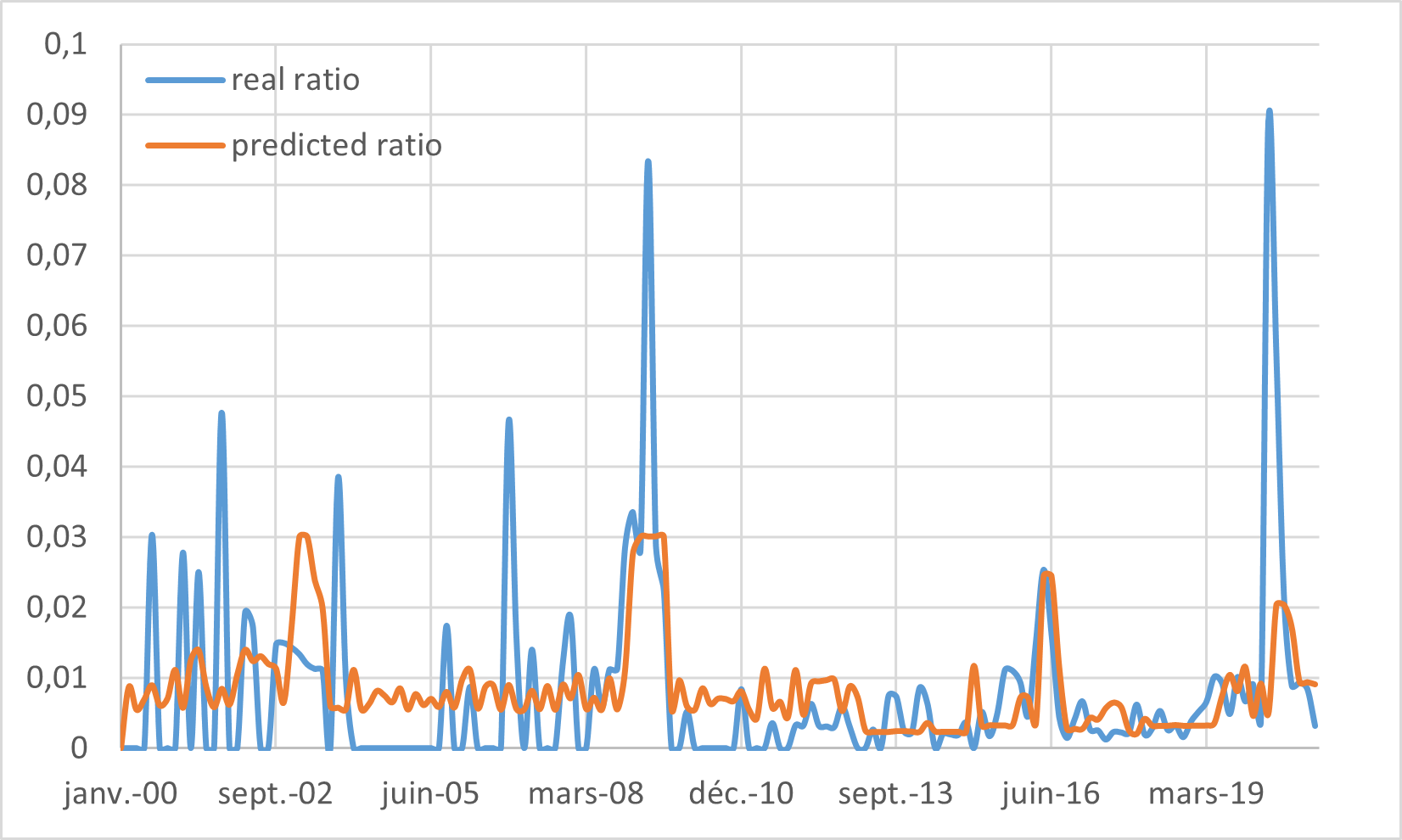} \caption{Real and predicted 50 days-transition ratios from B to C}
      \label{fig: MultDiscreetBackRatiosBC}
   \end{minipage}
\end{figure}

\subsubsection{Comparisons and analyses}

The results looks almost similar in both approaches. The dynamics of predicted ratios follow the trend of realized ratios. The forecasts also evolve when noteworthy crisis occurs.
We notice that transitions are more correlated during specific periods like crisis. Four crisis periods can be identified: a first small one around 2002, a moderated one in 2016 and two significant in 2008 and 2020. These latter are clearly identified as the subprime crisis and the health crisis caused by the Covid 19. The two others, moderated, would be respectively the consequences of the dot-com bubble in 2000 and the China stock market crash in June, 2015. During these periods the downgrades probabilities increase.\\

Both models are able to detect the evolution of the economic cycle from observations of rating migrations. The forecasts are adapted to the inferred economic state. During crisis periods, the models are able to predict adapted and higher downgrade probabilities.
\\
\\
We can underline three advantages of the continuous-time version compared to the discrete-time one.
\begin{itemize}
    \item The effect of delay (or lag effect) is less significant than in the discrete-time framework. By spreading simultaneous jumps in small time intervals, we make last information used for filtering fresher than it actually is. This fictive operation, however, improves the predictions.

    \item We note that the discrete-time model struggles to capture brief and brutal variations. As we observed in Section \ref{ContinuousFictive}, continuous-time filtering approach has the advantage of assimilating jumps one by one and of being more flexible and suitable to anticipate sudden transitions. Since information is spread and distributed in fictive intervals, the filter progressively assimilates information and is therefore quicker to adapt its predictions.
\end{itemize}
Nevertheless we can see that this framework is not fully adapted to rating migrations. The discrete version is easier and faster to compute: manipulations described in Section \ref{seq: continuous-adaptation} increases considerably the number of time intervals to consider, the complexity and remains laborious. Above all, the discrete model is more consistent with the data and finally, provides predictions of a better quality.
The continuous approach deals with continuous-time Markov chains. Therefore it could be improved by using an EM algorithm which estimates intensities directly. The effectiveness of this method would rather be highlighted by filtering a continuous phenomenon, where observations cannot occur simultaneously and exact occurrence dates are known. This intuition is confirmed by the following experiment. We compute the $R^2$ coefficient in the sample, to compare the forecasting power of the considered predictive models. We keep a reference time step of 50 days. We respectively compare the $R^2$ of the constant generator intensity model, the univariate models, and the multivariate discrete models and the continuous model in Table \ref{fig:my_label}.

\begin{table}[ht]
    \centering
    \footnotesize
    \caption{$R^2$ in the sample}
    \begin{tabular}{|c|c|c|c|c|c|}
        \hhline{~*{4}{-}}
        \multicolumn{1}{c|}{} & $A \rightarrow Baa$ & $Baa\rightarrow Ba$ & $Ba \rightarrow B$ & $B \rightarrow C$   \\ \hline
        Constant & {{0.463012}
} & {{0.250668}
} & {{0.483855}
} & {{0.250514}
}    \\ \hline
        Univ.Discrete & {{0.431184}
} & {{0.346886}
} & {{0.608684}
} & {{0.22263515}
}       \\ \hline
        Mult.Discrete & {{0.494395}
} & {{0.479324}
} & {{0.644975}
} & {{0.367094}
}       \\ \hline
        Mult.Continuous & {{0.2022}
} & {{0.331062}
} & {{0.49396}
} & {{0.279736}
}      \\ \hline

    \end{tabular}
     \label{fig:my_label}
\end{table}

We can directly notice that the multivariate discrete-time filter is the most accurate whatever the transition. The $R^2$ of the continuous filter is closed to discrete filter for transitions Baa to Ba and B to C but is lower for the transitions A to Ba, even lower than the $R^2$ from the constant generator model.
This phenomenons can be explained by the poorer calibration achieved for the continuous-time filtering (for a sake of rapidity) and the inconsistency with the format of the data.
The adapted continuous-time version can be applied to rating transitions framework and provides satisfactory predictions but can not reach the performance of the discrete-time version.
Note that, in the univariate case, each transition has its specific model.

%% file: GeneralFilter.tex
\section{General discrete-time version of the filter}
\label{sec:discreet-filter}
We present in this section a general discrete version of the filter, where the hidden process does not need to be a Markov chain.

\subsection{Univariate Form}

Let $\Gamma \in \mathbb{N}$ be the discrete time horizon. We work with the filtered probability space $(\Omega,\bfA,\bfF = (\cF_n)_{n \in \{0,\ldots,\Gamma\}},\Prob)$.
Let $N$ be a discrete-time $\bfF$-adapted counting process starting from $0$. 
Let $\bfF^N$ be the natural filtration of $N$, augmented with $\Prob-$null sets. We write
$$
 \forall \ n \in \{1,\ldots,\Gamma\}, \ \Delta N_{n}=N_{n}-N_{n-1}.
$$
We denote by $\J$, the support of the jumps of $N$ ($\J=\mathbb{N}$ for a Poisson process). The support could vary over time but as this would not impact our results, it is assumed constant for the sake of clearness. To account for the times when $N$ does not jump, we assume that $0 \in \J$ and we define $\bar{\J}=\J\backslash\{0\}$ as the support of the true jumps. Let $\Theta$ be a square integrable $\bfF$-adapted process such that conditionally on the $\sigma$-field $\cF_{n-1}$, $\Theta_n$ and $\Delta N_n$ are independent. Let $\bfF^{\Theta}$ be the natural filtration of $\Theta$,  augmented with $\Prob-$null sets.
$\Theta$ has a natural decomposition of the form 
$$
\Theta_{n}=A_{n}+M_{n} \,,
\label{eq:univariate-discreet-filter}
$$
where $A$ is a $\bfF$--predictable square integrable process and $M$ is a square integrable $\bfF$--martingale.
\\
We define the filtration $\bfG=(\cG_n)_{n \in \{0,\ldots,\Gamma\}}$ by $\cG_n=\cF^{\theta}_n \vee \cF^N_{n-1}$. For $n \in \{1, \dots, \Gamma\}$, we introduce
$$
\hat{O}_{n}=\E[O_{n}|\cF_{n}^N], \ \tilde{O}_{n-1}=\E[O_{n}|\cF_{n-1}].
$$
 For $n \in \{1, \dots, \Gamma\}$ and $j \in \J$, we define
$$
\epsilon_{n}^{j}=\mathds{1}_{[\Delta N_{n}=j]}, \
\lambda_{n-1}^{j}=\E[\epsilon_{n}^{j}|\cF_{n-1}], \
\hat{\lambda}_{n-1}^{j}=\E[\lambda_{n-1}^{j}|\cF_{n-1}^{N}]=\E\left[\E[\epsilon_{n}^{j}|\cF_{n-1}]|\cF_{n-1}^{N}\right].
$$
Note that, for all $n$, we have $\sum_{j\in \J} \epsilon_{n}^j= \sum_{j\in \J} \lambda_{n}^j=\sum_{j\in \J} \hat{\lambda}_{n}^j=1$.
Under this setting, we present a discrete time adaptation of the univariate filtering equation presented in \cite{Bremaud1981}.
\begin{proposition}
\label{prop}
The filtered process $\hat{\Theta}$ satisfies the following equation
\begin{equation}
\label{eq:discreet-filter}
\hat{\Theta}_{n}=\sum_{j\in \J}\frac{\widehat{(\tilde{\Theta}\lambda^j)}_{n-1}}{\hat{\lambda}_{n-1}^{j}}\mathds{1}_{[\Delta N_{n}=j]} , \quad \forall n=1,\dots,\Gamma.
\end{equation}
\end{proposition}
The proof of Proposition \ref{prop} heavily relies on the following lemma.  
\begin{lemma}
\label{lemma}
Let $K$ be a square integrable and $\bfG$-adapted process such that $\hat{K}$ is a $\bfF^N$-martingale. Then, $K$ is solution of the following recursive equation
\begin{equation}
\hat{K}_n={K}_{0} + \sum_{k=1}^{n}\sum_{j\in \J}\frac{\widehat{(\tilde{K}\lambda^j)}_{k-1}}{\hat{\lambda}_{k-1}^{j}}(\epsilon_{k}^{j}-\hat{\lambda}_{k-1}^{j}), \quad \forall n=1,\dots,\Gamma.
\end{equation}
\end{lemma}
\begin{remark}
If $K$ is also a $\bfF$-martingale then the previous equation also holds true with $\tilde{K}= K$. 
\end{remark}

\begin{proof}[of Lemma \ref{lemma}]
Let $P$ be a square integrable $\bfF^N$--martingale. Then, there exists a measurable function $g$ such that $P_{n}=g(n,N_{0},..N_{n})$. It can equivalently be written $P_{n}=h(n,\Delta N_{1},..,\Delta N_{n})$.
\\
We can write
$$
P_{n}=\sum_{j\in \J} h(n,\Delta N_{1},..,\Delta N_{n-1},j)\epsilon_{n}^{j}.
$$
Since $P$ is a $\bfF^N$--martingale,
\begin{align*}
P_{n}-P_{n-1}=P_n -\E[P_{n}|\cF_{n-1}^N]=\sum_{j\in \J} h(n,\Delta N_{1},..,\Delta N_{n-1},j)(\epsilon_{n}^{j}-\hat{\lambda}_{n-1}^{j}).
\end{align*}
Then, P has the following martingale representation 
$$
P_{n}=P_{0}+\sum_{k=1}^n\sum_{j\in \J} H_{k-1}^{j}(\epsilon_{k}^{j}-\hat{\lambda}_{k-1}^{j}),
$$
with $H_{k-1}^{j}$ being $\cF_{k-1}^N$-measurable.
Conversely, it is easy to check that a process written like that is a  $\bfF^N$--martingale.
\\
\\
Since $\epsilon_{k}^0=1-\sum_{j\in \bar{\J}}\epsilon_{k}^j$ and $\hat{\lambda}_{k}^0=1-\sum_{j\in \bar{\J}}\hat{\lambda}_{k}^j$, we can rewrite $P$ with the following martingale representation
\begin{equation}
\label{representationBis}
P_{n}=P_{0}+\sum_{k=1}^n\sum_{j\in \bar{\J}} W_{k-1}^{j}(\epsilon_{k}^{j}-\hat{\lambda}_{k-1}^{j}),
\end{equation}
with $W_{k-1}^{j}=H_{k-1}^j -H_{k-1}^0 $.
\\
Let $K$ be a square integrable $\bfG$-adapted process such that $\hat{K}$ is a $\bfF^N$--martingale. 
According to (\ref{representationBis}), we can write 
$
\hat{K}_{n}=\zeta+\sum_{k=1}^n\sum_{j\in \bar{\J}} W_{k-1}^{j}(\epsilon_{k}^{j}-\hat{\lambda}_{k-1}^{j})$, where the sequence $(W_k)_k$ is $\cF^N$-adapted.
For any square integrable $\bfF^N$--adapted process $X$, we have $\E[(K_{n}-\hat{K}_{n})X_{n}]=0$.
Choosing $X$ to be a $\bfF^N$--martingale with the decomposition $g + \sum_{k=1}^n \sum_{j\in \bar{\J}} G_{k-1}^{j}(\epsilon_{k}^j-\hat{\lambda}_{k-1}^j)$ with $(G_k)_k$ $\bfF^N$-adapted,  we obtain
$$
\E\left[\left(K_{n}-\zeta-\sum_{k=1}^n\sum_{j\in \bar{\J}} W_{k-1}^{j}(\epsilon_{k}^j-\hat{\lambda}_{k-1}^{j})\right)\left(g+\sum_{k=1}^n\sum_{j\in \bar{\J}} G_{k-1}^{j}(\epsilon_{k}^j-\hat{\lambda}_{k-1}^j)\right)\right]=0.
$$
Choosing $G_{k-1}^j=0$  for all $j\in \bar{\J}$ leads to $\zeta=\E[\hat{K}_{n}]=\E[K_{n}]=K_{0}$.
\\
With no loss of generality, we can consider that $K_0=0$. By choosing $g=0$ we obtain: 
\begin{align}
\label{eq:decomposition}
\E\left[K_{n}\sum_{k=1}^n\sum_{j\in \bar{\J}} G_{k-1}^{j}(\epsilon_{k}^{j}-\hat{\lambda}_{k-1}^{j})\right]-\E\left[\sum_{j_{1},j_{2}\in \bar{\J}}\sum_{k_{1},k_{2}\geq 1}^n W_{k_{1}-1}^{j_{1}}(\epsilon_{k_{1}}^{j_{1}}-\hat{\lambda}_{k_{1}-1}^{j_{1}})G_{k_{2}-1}^{j_{2}}(\epsilon_{k_{2}}^{j_{2}}-\hat{\lambda}_{k_{
2}-1}^{j_{2}})\right]=0.
\end{align}
\\
For $k_{1}< k_{2}$,
\begin{align*}
\begin{aligned}
\E\left[\E\left[W_{k_{1}-1}^{j_{1}}(\epsilon_{k_{1}}^{j_{1}}-\hat{\lambda}_{k_{1}-1}^{j_{1}})G_{k_{2}-1}^{j_{2}}(\epsilon_{k_{2}}^{j_{2}}-\hat{\lambda}_{k_{
2}-1}^{j_{2}})|\cF_{k_{2}-1}^N\right]\right]
=0
\end{aligned}
\end{align*}
Noticing that $\E\left[K_n|\cF_{k}^N\right]=\E\left[\E\left[K_n|\cF_{n}^N\right]|\cF_{k}^N\right]=\hat{K}_k$, we compute the term 
\begin{align*}
\begin{aligned}
S_{n,k}
= & \E\left[K_{n}G_{k-1}^{j}(\epsilon_{k}^{j}-\hat{\lambda}_{k-1}^{j})\right]
=  \E\left[\E\left[K_{n}|\cF_{k}^N\right]G_{k-1}^{j}(\epsilon_{k}^{j}-\hat{\lambda}_{k-1}^{j})\right]\\
= & \E\left[\hat{K}_{k}G_{k-1}^{j}(\epsilon_{k}^{j}-\hat{\lambda}_{k-1}^{j})\right]
= \E\left[{K}_{k}G_{k-1}^{j}(\epsilon_{k}^{j}-\hat{\lambda}_{k-1}^{j})\right].\\
\end{aligned}
\end{align*}
Now, let us compute
\begin{align*}
\begin{aligned}
\E[K_{k}(\epsilon_{k}^{j}-\hat{\lambda}_{k-1}^{j})|\cF_{k-1}^N] 
= & \E[K_{k}\epsilon_{k}^{j}|\cF_{k-1}^N]- \E[K_{k}\hat{\lambda}_{k-1}^j|\cF_{k-1}^N]\\
= & \E\left[\E[K_{k}\epsilon_{k}^{j}|\cF_{k-1}]|\cF_{k-1}^N\right]-\hat{K}_{k-1}\hat{\lambda}_{k-1}^{j}.
\end{aligned}
\end{align*}
Remember that $K$ is $\bfG$-measurable and that $\cG_k=\cF^{\theta}_k \vee \cF^N_{k-1}$. The conditional independence of $\Theta_k$ and $\Delta N_k$ knowing $\cF_{k-1}$ yields the conditional independence of $K_k$ and $\epsilon_k$. Hence, we obtain
\begin{align*}
 \begin{aligned}
 \E\left[\E[K_{k}\epsilon_{k}^{j}|\cF_{k-1}]|\cF_{k-1}^N\right]-\hat{K}_{k-1}\hat{\lambda}_{k-1}^{j}
= &\E\left[\E[K_{k}|\cF_{k-1}]\E[\epsilon_{k}^{j}|\cF_{k-1}]|\cF_{k-1}^N\right]-\hat{K}_{k-1}\hat{\lambda}_{k-1}^{j}\\
= &\widehat{(\tilde{K}\lambda^j)}_{k-1}-\hat{K}_{k-1}\hat{\lambda}_{k-1}^{j}.\\
\end{aligned}
\end{align*}
From \ref{eq:decomposition}, we obtain
\begin{align*}
\E\left[\sum_{k=1}^n \sum_{j\in \bar{\J}} G_{k-1}^{j}\left(\widehat{(\tilde{K}\lambda^j)}_{k-1}-\hat{K}_{k-1}\hat{\lambda}_{k-1}^{j}-\E\left[\sum_{i\in \bar{\J}} W_{k-1}^{i}(\epsilon_{k}^{i}-\hat{\lambda}_{k-1}^{i})(\epsilon_{k}^{j}-\hat{\lambda}_{k-1}^{j})|\cF_{k-1}^{N}\right]\right)\right]=0
\end{align*}
As this holds true for any $\bfF^N$-adapted process $(G_k)_k$, we can choose
\begin{align*}
G_{k-1}^{j}=\widehat{(\tilde{K}\lambda^j)}_{k-1}-\hat{
K}_{k-1}\hat{\lambda}_{k-1}^{j}-\E\left[\sum_{i\in \bar{\J}} W_{k-1}^{i}(\epsilon_{k}^{i}-\hat{\lambda}_{k-1}^{i})(\epsilon_{k}^{j}-\hat{\lambda}_{k-1}^{j})|\cF_{k-1}^N\right].
\end{align*}
Then, we deduce the following system of linear equations, 
\begin{equation}
\label{syst}
\forall \ j \in \bar{\J}, \ k\in\{1,\dots,n\}, \ \sum_{i\in \bar{\J}} W_{k-1}^{i}\E\left[(\epsilon_{k}^{i}-\hat{\lambda}_{k-1}^{i})(\epsilon_{k}^{j}-\hat{\lambda}_{k-1}^{j})|\cF_{k-1}^N\right]=\widehat{(\tilde{K}\lambda^j)}_{k-1}-\hat{
K}_{k-1}\hat{\lambda}_{k-1}^{j} \quad a.s.
\end{equation}
Let $c_j^{(k-1)}= \widehat{(\tilde{K}\lambda^j)}_{k-1}-\hat{
K}_{k-1}\hat{\lambda}_{k-1}^{j}$. Note that
\begin{equation*}
\E\left[(\epsilon_{k}^i-\hat{\lambda}_{k-1}^{i})(\epsilon_{k}^j-\hat{\lambda}_{k-1}^j)|\cF_{k-1}^N\right]= 
\begin{cases}
-\hat{\lambda}_{k-1}^i\hat{\lambda}_{k-1}^j & 
\mbox{for } i \neq j\\
\hat{\lambda}_{k-1}^j (1 - \hat{\lambda}_{k-1}^j) & \mbox{for } i = j
\end{cases}
\end{equation*}
The equations in (\ref{syst}) can be written for all $j \in \J$
\begin{align*}
\begin{aligned}
\hat{\lambda}_{k-1}^j (1 - \hat{\lambda}_{k-1}^j)W_{k-1}^j-\sum_{\substack{i \in \bar{\J}\\i \neq j}} \hat{\lambda}_{k-1}^i\hat{\lambda}_{k-1}^j W_{k-1}^i=c_j^{(k-1)}. 
\end{aligned}
\end{align*}
Then, we obtain for all $j \in \J$
\begin{equation}
\label{syst2}
\hat{\lambda}_{k-1}^j W_{k-1}^j-\hat{\lambda}_{k-1}^j\sum_{i \in \bar{\J}} \hat{\lambda}_{k-1}^i W_{k-1}^i=c_j^{(k-1)}.
\end{equation}
By summing for $j \in \bar{\J}$, we deduce from (\ref{syst2}),
\begin{align*}
\begin{aligned}
\sum_{j \in \bar{\J}} \hat{\lambda}_{k-1}^j W_{k-1}^j-\sum_{j\in \bar{\J}}\hat{\lambda}_{k-1}^j\sum_{i \in \bar{\J}} \hat{\lambda}_{k-1}^i W_{k-1}^i=\sum_{j\in \bar{\J}}c_j^{(k-1)}.
\end{aligned}
\end{align*}
Then 
\begin{align*}
\begin{aligned}
\hat{\lambda}_{k-1}^0\sum_{j \in \bar{\J}} \hat{\lambda}_{k-1}^j W_{k-1}^j=\sum_{j\in \bar{\J}}c_j^{(k-1)}.
\end{aligned}
\end{align*}
Inserting the expression of $\sum_{i \in \bar{\J}} \hat{\lambda}_{k-1}^i W_{k-1}^i$ in (\ref{syst2}) gives 
\begin{align*}
\begin{aligned}
W_{k-1}^j=\frac{c_j^{(k-1)}}{\hat{\lambda}_{k-1}^j} + \frac{\sum_{i\in \bar{\J}}c_{i}^{(k-1)}}{\hat{\lambda}_{k-1}^0}.
\end{aligned}
\end{align*}
Then, we obtain 

\begin{align*}
\begin{aligned}
W^{j}_{k-1}=\frac{1}{\hat{\lambda}_{k-1}^{j}}(\widehat{(\tilde{K}\lambda^j)}_{k-1}-\hat{K}_{k-1}\hat{\lambda}_{k-1}^{j}) + \frac{1}{\hat{\lambda}_{k-1}^{0}}\sum_{i\in \bar{\J}}(\widehat{(\tilde{K}\lambda^i)}_{k-1}-\hat{K}_{k-1}\hat{\lambda}_{k-1}^{i}).
\end{aligned}
\end{align*}

Finally, by replacing $\sum_{j\in \bar{\J}}\epsilon_{k}^{j}$, $\sum_{j\in \bar{\J}}\hat{\lambda}_{k-1}^{j}$ by $1-\epsilon_{k}^{0}$ and $1-\hat{\lambda}_{k-1}^{0}$, we derive the general filtering formula 
\begin{align*}
\begin{aligned}
\label{eq:univarite-discreete-martingale}
\hat{K}_{n}
= & K_{0} + \sum_{k=1}^{n}\sum_{j\in \bar{\J}}\left(\frac{\widehat{(\tilde{K}\lambda^j)}_{k-1}}{\hat{\lambda}_{k-1}^{j}}-\frac{\widehat{(\tilde{K}\lambda^0)}_{k-1}}{\hat{\lambda}_{k-1}^{0}}\right)(\epsilon_{k}^{j}-\hat{\lambda}_{k-1}^{j})\\
= & {K}_{0} + \sum_{k=1}^{n}\sum_{j\in \J}\frac{\widehat{(\tilde{K}\lambda^j)}_{k-1}}{\hat{\lambda}_{k-1}^{j}}(\epsilon_{k}^{j}-\hat{\lambda}_{k-1}^{j}),
\end{aligned}
\end{align*}
This finishes the proof. 
\end{proof}
\begin{proof}[of Proposition \ref{prop}]
We define
\begin{equation}
\label{eq:ak}
a_{n}=A_{n}-A_{n-1}=\E[\Theta_{n}|\cF_{n-1}]-\Theta_{n-1},
\end{equation}
\begin{equation*}
B_{n}=\sum_{k=1}^{n}\E[a_{k}|\cF_{k-1}^N]
\end{equation*}
\begin{equation*}
L_{n}=\sum_{k=1}^n a_{k} - \E[a_{k}|\cF_{k-1}^N].
\end{equation*}
Note that
\begin{equation*}
\Theta_{n}-B_{n}=M_{n}+L_{n}.
\end{equation*}
We know that $\hat{M}$ is a $\bfF^N$--martingale and $\hat{L}$ is clearly a $\bfF^N$--martingale too. 
Then, we can apply Lemma~\ref{lemma} to $\Theta - B$. Note that $B$ is $\bfF^{N}$ predictable, then $\tilde{B}_{k-1}=B_{k}$ and so $\widehat{\tilde{B}\lambda^{j}}_{k-1}=B_{k} \hat{\lambda}^{j}_{k-1}$. Then,
\begin{align*}
\begin{aligned} 
\hat{\Theta}_{n}-\hat{B}_{n}
= & \hat{\Theta}_{0} - \hat{B}_{0} + \sum_{k=1}^{n}\sum_{j\in \J}(\frac{\reallywidehat{((\tilde{\Theta}-\tilde{B})\lambda^{j})}_{k-1}}{\hat{\lambda}_{k-1}^{j}})(\epsilon_{k}^{j}-\hat{\lambda}_{k-1}^{j})\\
= & \hat{\Theta}_{0} - \hat{B}_{0} + \sum_{k=1}^{n}\sum_{j\in \J}\frac{\widehat{(\tilde{\Theta}\lambda^j)}_{k-1}}{\hat{\lambda}_{k-1}^{j}}(\epsilon_{k}^{j}-\hat{\lambda}_{k-1}^{j}).
\end{aligned}
\end{align*}
We compute 
$\hat{\Theta}_{n}-\hat{\Theta}_{n-1}
=\hat{B}_{n}-\hat{B}_{n-1} + f(\lambda_{n-1},\epsilon_{n},\tilde{\Theta}_{n-1})
=\E[a_{n}|\cF_{n-1}^N] + f(\lambda_{n-1},\epsilon_{n},\tilde{\Theta}_{n-1})$.
From (\ref{eq:ak}), 
$\E[a_{n}|\cF^N_{n-1}]=\E[\Theta_{n}|\cF_{n-1}^N]-\hat{\Theta}_{n-1}$ 
and using that
\begin{align*}
\begin{aligned}
f(\lambda_{n-1},\epsilon_{n},\tilde{\Theta}_{n-1}) =\sum_{j\in \J}\frac{\widehat{(\tilde{\Theta}\lambda^j)}_{n-1}}{\hat{\lambda}_{n-1}^{j}}(\epsilon_{n}^{j}-\hat{\lambda}_{n-1}^{j})
=\sum_{j\in \J}\frac{\widehat{(\tilde{\Theta}\lambda^{j})}_{n-1}}{\hat{\lambda}_{n-1}^{j}}\mathds{1}_{[\Delta N_{n}=j]}-\E[\Theta_{n}|\cF_{n-1}^N],
\end{aligned}
\end{align*}
we deduce the final form of the filtering formula 
$$
\hat{\Theta}_{n}=\sum_{j\in \J}\frac{\widehat{(\tilde{\Theta}\lambda^{j})}_{n-1}}{\hat{\lambda}_{n-1}^{j}}\mathds{1}_{[\Delta N_{n}=j]}.
$$
Note that, from \eqref{eq:ak}, the previous formula can also be stated as
$$
\hat{\Theta}_{n}=\sum_{j\in \J}\left(\frac{\widehat{(a\lambda^{j})}_{n-1}}{\hat{\lambda}_{n-1}^{j}}+\frac{\widehat{(\Theta\lambda^{j})}_{n-1}}{\hat{\lambda}_{n-1}^{j}}\right)\mathds{1}_{[\Delta N_{n}=j]}.
$$
where $\widehat{(a\lambda^{j})}_{n-1} = \E[a_n \lambda_{n-1}^{j}  |\cF_{n-1}^N]$.
\end{proof}
This formula can be extended to a multivariate setting. 
\begin{remark}
\label{rem:other-proof}
Note that the filtering formula of Proposition \ref{prop:discreet-univariate-rating} derives from this general equation. Indeed, we have
$$
\E(\mathds{1}_{[\Delta N_{n}=j]}|\cF_{n-1})=\E(\mathds{1}_{[\Delta N_{n}=j]}|N_{n-1}, \Theta_{n-1},Y_{n}),
$$
and
$$
\Prob(\Delta N_{n}=j|\Theta_{n-1}=h,Y_{n}=y_{n})=\binom{y_{n}}{j}(L^{h})^{j}(1-L^{h})^{y_{n}-j}.
$$
Then we compute the following expressions 
$$
\hat{\lambda}_{n-1}^{j}=\E\left[\mathds{1}_{[\Delta N_{n}=j]}|\cF_{n-1}^N\right]=\sum_{h=1}^m\binom{Y_{n}}{j}(L^{h})^{j}(1-L^{h})^{Y_{n}-j}\hat{I}_{n-1}^{h},
$$
$$
\widehat{(I^{h}\lambda^j)}_{n-1}=\binom{Y_{n}}{j}(L^{h})^{j}(1-L^{h})^{Y_{n}-j}\hat{I}_{n-1}^{h}
$$
and
\begin{align*}
\begin{aligned}
\widehat{(\tilde{I}^h\lambda^{j})}_{n-1}& =\E\left[\E[I_{n}^h|\cF_{n-1}]\E[\mathds{1}_{[\Delta N_{n}=j]}|\cF_{n-1}]|\cF_{n-1}^N\right]\\
& =\E\left[(\sum_{\mu=1}^m K^{\mu h}I_{n-1}^{\mu})(\sum_{s=1}^m\binom{Y_{n}}{j}(L^{s})^{j}(1-L^{s})^{Y_{n}-j}I_{n-1}^{s})|\cF_{n-1}^N\right]\\
& =\sum_{s=1}^m K^{s h}\binom{Y_{n}}{j}(L^{s})^{j}(1-L^{s})^{Y_{n}-j}\hat{I}_{n-1}^{s}.
\end{aligned}
\end{align*}
Finally using (\ref{eq:discreet-filter}) we obtain the desired equation.
\end{remark}

\subsection{Multivariate form}

Let $\Gamma \in \mathbb{N}$ be the discrete time horizon. We work with the filtered probability space $(\Omega,\bfA,\bfF = (\cF_n)_{n \in \{0,\ldots,\Gamma\}},\Prob)\,,
$.
Let $N = (N^1, \ldots, N^\rho)$ be a discrete-time multivariate counting process where for $i=1, \ldots, \rho$, $N^i=(N_n^i)_{n \in \{0,\ldots,\Gamma\}}$, is a be a 
simple $\bfF$-adapted counting processes starting from $0$.
Let $\bfF^N$ be the natural filtration of $N$, augmented with $\Prob-$null sets. We write
$$
 \forall \ i \in \{1,\ldots,\rho\}, \ \forall \ n \in \{1,\ldots,\Gamma\}, \ \Delta N_{n}^i=N_{n}^i-N_{n-1}^i.
$$
We denote by $\J_i$ the support of the jumps of $N^i$, for $i=1, \ldots, \rho$. The support may vary over time but as this would not impact our results, it is assumed constant for the sake of clearness. To account for the times when $N^i$ does not jump, we assume that $0 \in \J_i$ and we define $\bar{\J}_i=\J_i\backslash\{0\}$ as the support of the true jumps. Let us define the product spaces $\J^{\otimes}=\prod_{i=1}^\rho \J_i$ and $\bar{\J}^{\otimes}=\prod_{i=1}^\rho \bar{\J}_i$.
\\
Let $\Theta$ be a square integrable $\bfF$-adapted process such that conditionally on the $\sigma$-field $\cF_{n-1}$, $\Theta_n$ and $(\Delta N_n^i)_{i=1, \ldots, \rho}$ are independent. Let $\bfF^{\Theta}$ be the natural filtration of $\Theta$,  augmented with $\Prob-$null sets.
We define the filtration $\bfG=(\cG_n)_{n \in \{0,\ldots,\Gamma\}}$  by $\cG_n=\cF^{\theta}_n \vee \cF^N_{n-1}$.
\\
We extend the previous setting to multivariate case,
$$
\forall \delta \in \J^{\otimes}, 
\epsilon_{n}^{\delta}=\mathds{1}_{[\Delta N_{n}=\delta]}=\mathds{1}_{[\bigcap_{i=1}^{\rho} \Delta N^i_{n}=\delta_i]}, \
\lambda_{n-1}^{\delta}=\E[\epsilon_{n}^{\delta}|\cF_{n-1}].
$$
\begin{proposition}
The filtered process $\hat{\Theta}$ satisfies for all $n=1, \dots, \Gamma$,
\begin{equation}
\label{eq:multivariate-discreet-filter}
\hat{\Theta}_{n}=\sum_{\delta \in \J^{\otimes}}\frac{\widehat{(\tilde{\Theta}\lambda^{\delta})}_{n-1}}{\hat{\lambda}_{n-1}^{\delta}}\mathds{1}_{[\Delta N_{n}=\delta]}.
\end{equation}
\end{proposition}
\begin{proof}
We leave the proof to the reader as it goes along the same lines as the proof of Prop. \ref{prop}.
\end{proof}
\begin{remark}
With adequate assumptions, this filtering formula could cover non Markovian case. These considerations are left for future research. 
\end{remark}

%% file: Conclusion.tex
\section{Conclusion}
\label{sec:conclusion}
In this paper, we adapt the filtering framework applied to Markov chains, to the context of rating migrations, in both a discrete-time and a continuous-time setting. For both approaches, we assume that rating transitions in a pool of obligors are driven by the same systematic hidden factor. The two alternatives are studied and compared. We discussed calibration issues and compared the predicted future rating transition probabilities on fictive and real data. Then, under a point process filtering framework, we extend the discrete-time Markov chain filtering framework to a more general filtering one, which may cover non Markovian behaviours.  
\\
As illustrated in Sections \ref{sec:fictiveData} and \ref{sec:realData}, our methodology provides predictors adapted to the evolution of the economical cycle. We believe that our approaches can be used for PIT-estimations of transitions and detection of regimes. During crisis periods, our models are able to predict higher downgrade probabilities. Compared to other PIT-estimation models, our approaches base their predictions on the business cycle without concern of macro economic factors.
From a practical point of view, our approaches also have the advantage to be interpretable. Observing the risk profile of each state and their filtered trajectories allows us to better understand the dynamics of the economic cycle as well as its systematic effect on rating migrations.

However, both approaches cannot capture idiosyncratic information: indeed \cite{schwaab2017global} found that only 18\% to 26\% of global default risk variation is systematic while the reminder is idiosyncratic. The share of systematic default risk is higher (39\% to 51\%) if industry-specific variation is counted as systematic. 
\\
In addition, applying the continuous framework to discrete time data is tedious and presents a risk of altering information and the quality of predictions. 
Since its complexity is much more important, the execution results of the continuous-time algorithm is very slow. Therefore, it suffers from poorer calibration than the discrete-time version.
Thanks to the adaptations presented in Section \ref{outOfSample}, the continuous-time version is able to provide satisfactory predictions but can not reach the performances of the discrete-time model. For these reasons, the discrete-time approach turns out to be more adapted and efficient for the context of rating migrations.   
\\
\\
Several improvements could be made to our framework. Both models could consider additional idiosyncratic observable factors as in \cite{koopman2008multi}. Moreover, the continuous-time framework could be improved by using an EM algorithm which estimates continuous time parameters directly (as, e.g., in \cite{damian2018algorithm}, \cite{liu2017learning} and \cite{nodelman2012expectation}). This algorithm will still have to deal with simultaneous jumps and be computationally tractable.   
Furthermore, many studies show that rating migrations' dynamics first exhibit a non-Markovian behavior (migration data exhibit correlation among rating change dates, known as “rating drift”, contagion effect, \ldots) that cannot be captured by our models. The integration of these effects may represent a subject of reflection. Finally, the general derived discrete-time filtering formula could be applied to the context of rating migrations, possibly driven by non-Markovian hidden process. These considerations are left for future researches.

%% file: Appendice.tex
\section{Calibration of the discrete Version}
\label{proof:calibration-discreet}
This part describes the derivation for the Baum-Welch algorithm adaptation presented in Section \ref{sec:calibration}.
\\
We compute $\forall s \in \{1,\ldots, m\}, \forall n \in \{1,\ldots,\Gamma\}$, the forward probability
$\alpha_{n}(s)=\Prob(Z_{0|n}=z_{0|n},\Theta_{n-1}=s)$ and $\forall s \in \{1,\ldots,m\}, \forall n \in \{1,\ldots,\Gamma-1\}$, the backward probability $\beta_{n}(s)=\Prob(Z_{n+1|\Gamma}=z_{n+1|\Gamma}|Z_{n}=z_{n},\Theta_{n-1}=s)$.
\\
We derive $\forall n \in \{2,\ldots,\Gamma-1\} \ \mbox{and} \ \forall s \in \{1,\ldots,m\}$,
\begin{align*}
\begin{aligned}
\alpha_{n}(s)& =\Prob(Z_{0|n}=z_{0|n},\Theta_{n-1}=s)=\sum_{l=1}^m\Prob(Z_{0|n}=z_{0|n},\Theta_{n-1}=s,\Theta_{n-2}=l)\\
& =\sum_{l=1}^m\prod_{d=1}^Q\Prob(Z_{n}^{d}=z_{n}^{d}|Z_{0|n-1}=z_{0|n-1},\Theta_{n-1}=s,\Theta_{n-2}=l)\Prob(Z_{0|n-1}=z_{0|n-1},\Theta_{n-1}=s,\Theta_{n-2}=l)\\
& =\sum_{l=1}^m\prod_{d=1}^Q\Prob(Z_{n}^{d}=z_{n}^{d}|Z_{n-1}=z_{n-1},\Theta_{n-1}=s)\Prob(\Theta_{n-1}=s|\Theta_{n-2}=l,Z_{0|n-1}=z_{0|n-1})\alpha_{n-1}(l)\\
& =\sum_{l=1}^m\prod_{d=1}^Q\Prob(Z_{n}^{d}=z_{n}^{d}|Z_{n-1}^{d}=z_{n-1}^{d},\Theta_{n-1}=s)\Prob(\Theta_{n-1}=s|\Theta_{n-2}=l)\alpha_{n-1}(l)\\
& =\sum_{l=1}^m \alpha_{n-1}(l) K^{ls}\prod_{d=1}^Q L^{s,z_{n-1}^{d}z_{n}^{d}}=\sum_{l=1}^{m}\alpha_{n-1}(l)K^{ls}\prod_{i,r \in {\Upsilon}}(L^{s,ir})^{\Delta N_n^{ir}}.
\end{aligned}
\end{align*}
For n=1, we do not know the state of the individuals before the simulation. 
To tackle this issue, we use the initial proportion of the ratings. We have
\\
\begin{align*}
\begin{aligned}
\alpha_{1}(s)& =\Prob(Z_{1}=z_{1},\Theta_{0}=s)\\
& =\prod_{d=1}^Q\Prob(Z_{1}^{d}=z_{1}^{d}|\Theta_{0}=s)\Pi(s)\\
& =\prod_{d=1}^Q\sum_{j \in\Upsilon}\Prob(Z_{1}^{d}=z_{1}^{d}|Z_{0}^{d}=j,\Theta_{0}=s)\Prob(Z_{0}^{d}=j|\Theta_{0}=s)\Pi(s)\\
& =\prod_{d=1}^Q\sum_{j\Upsilon}\Prob(Z_{1}^{d}=z_{1}^{d}|Z_{0}^{d}=j,\Theta_{0}=s)\Prob(Z_{0}^{d}=j)\Pi(s)\\
& =\prod_{d=1}^Q\sum_{j\in \Upsilon}L^{s,jz_{1}^{d}}\Prob(Z_{0}^{d}=j)\Pi(s).\\
\end{aligned}
\end{align*}
\\
\\
Similarly, we recursively derive the backward probability for all $n \in \{1,\ldots,\Gamma-2\}$ and $s \in \{1,\ldots,m\}$,
\begin{align*}
\begin{aligned}
\beta_{n}(s)& =\Prob(Z_{n+1|\Gamma}=z_{n+1|\Gamma}|\Theta_{n-1}=s,Z_{n}=z_{n})\\\\
& =\sum_{l=1}^m\Prob(Z_{n+1|\Gamma}=z_{n+1|\Gamma},\Theta_{n}=l|\Theta_{n-1}=s,Z_{n}=z_{n})\\
& =\sum_{l=1}^m\Prob(Z_{n+2|\Gamma}=z_{n+2|\Gamma}|\Theta_{n}=l,Z_{n+1}=z_{n+1},Z_{n}=z_{n},\Theta_{n-1}=s)\Prob(Z_{n+1}=z_{n+1},\Theta_{n}=l|\Theta_{n-1}=s,Z_{n}=z_{n})\\
& =\sum_{l=1}^m\Prob(Z_{n+2|\Gamma}=z_{n+2|\Gamma}|\Theta_{n}=l,Z_{n+1}=z_{n+1})\prod_{d=1}^Q(\Prob(Z_{n+1}^{d}=z_{n+1}^{d}|\Theta_{n}=l,Z_{n}^{d}=z_{n}^{d}))\Prob(\Theta_{n}=l|\Theta_{n-1}=s)\\
& =\sum_{l=1}^m\beta_{n+1}(l)K^{sl}\prod_{d=1}^Q L^{l,z_{n}^{d}z_{n+1}^{d}}=\sum_{l=1}^{m}\beta_{n+1}(l)K^{sl}\prod_{i,r \in {\Upsilon}}(L^{l,ir})^{\Delta N_{n+1}^{ir}}.
\end{aligned}
\end{align*}
For $n=\Gamma-1$, we take $\forall s \in \{1,\ldots,m\}, \ \beta_{\Gamma-1}(s)=1$.
\\
\\
Both estimators will be used to replace the missing data during the maximization phase. The missing data describing the hidden factor are defined $\forall h\in \{1,\ldots,m\}, \forall n \in \{1,\ldots,\Gamma\}$, $u_{n}(h)=\mathds{1}_{[\Theta_{n}=j]}$, and
$v_{n}(s,h)=\mathds{1}_{[\Theta_{n}=h,\Theta_{n-1}=s]}$. We define the associated Bayesian estimators $\check{u}_{n}(h)=\Prob(\Theta_{n}=h|Z_{0|\Gamma}=z_{0|\Gamma}),
\label{eq:u}$ and $\check{v}_{n}(s,h)=\Prob(\Theta_{n}=h,\Theta_{n-1}=s|Z_{0|\Gamma}=z_{0|\Gamma}).
\label{eq:v}$
\\
We derive expression of these Bayesian estimators with the forward and the backward probabilities For all $h \in \{1,\ldots,m\}$ and all $n \in \{1,\ldots,\Gamma-2\}$,
\begin{align*}
\begin{aligned}
\check{u}_{n}(h)& =\Prob(\Theta_{n}=h|Z_{0|\Gamma}=z_{0|\Gamma})\\
& =\frac{\Prob(Z_{n+2|\Gamma}=z_{n+2|\Gamma}|\Theta_{n}=h,Z_{0|n+1}=z_{0|n+1})\alpha_{n+1}(h)}{\Prob(Z_{0|\Gamma}=z_{0|\Gamma})}\\
& =\frac{\Prob(Z_{n+2|\Gamma}=z_{n+2|\Gamma}|\Theta_{n}=h,Z_{n+1}=z_{n+1})\alpha_{n+1}(h)}{L_{\Gamma}}\\
& =\frac{\beta_{n+1}(h)\alpha_{n+1}(h)}{L_{\Gamma}}.
\end{aligned}
\end{align*}
With $L_{\Gamma}$ being the likelihood on the whole sample,
$
 L_{\Gamma}=\Prob(Z_{0|\Gamma}=z_{0|\Gamma})=\sum_{j}\alpha_{\Gamma}(j).\\
$
\begin{align*}
\begin{aligned}
&\check{v}_{n}(s,h) = \Prob(\Theta_{n}=h,\Theta_{n-1}=s|Z_{0|\Gamma}=z_{0|\Gamma})\\
& =\frac{\Prob(\Theta_{n}=h,Z_{n+1|\Gamma}=z_{n+1|\Gamma}|Z_{0|n}=z_{0|n},\Theta_{n-1}=s)\alpha_{n}(s)}{L_{\Gamma}}\\
& =\frac{\Prob(Z_{n+2|\Gamma}=z_{n+2|\Gamma}|Z_{0|n+1}=z_{0|n+1},\Theta_{n-1}=s,\Theta_{n}=h)\Prob(\Theta_{n}=h,Z_{n+1}=z_{n+1}| Z_{0|n}=z_{0|n}, \Theta_{n-1}=s)\alpha_{n}(s)}{L_{\Gamma}}\\
& =\frac{\Prob(Z_{n+2|\Gamma}=z_{n+2|\Gamma}|Z_{n+1}=z_{n+1},\Theta_{n}=h)\prod_{d=1}^Q(\Prob(Z_{n+1}^{d}=z_{n+1}^{d}|\Theta_{n}=h,Z_{n}^{d}=z_{n}^{d}))\Prob(\Theta_{n}=h|\Theta_{n-1}=s)\alpha_{n}(s)}{L_{\Gamma}}\\
& =\frac{\beta_{n+1}(h)K^{sh}\alpha_{n}(s)\prod_{d=1}^Q L^{h,z_{n}^{d}z_{n+1}^{d}}}{L_{\Gamma}}=\frac{\beta_{n+1}(h)K^{sh}\alpha_{n}(s)\prod_{i,r \in {\Upsilon}}(L^{h,ir})^{\Delta N_n^{ir}}}{L_{\Gamma}}.
\end{aligned}
\end{align*}
Using the concavity of the $\log$ function, we establish a useful inequality for the next step of the derivations.
For two strictly positive sequences $w$ and $w'$,
\begin{align*}
\begin{aligned}
\log\left(\frac{\sum_{i}w_{i}'}{\sum_{k}w_{k}}\right)& =\log\left(\sum_{i}\frac{w_{i}w_{i}'}{\sum_{k}w_{k}w_{i}}\right)\\
& \geq \sum_{i}\frac{w_{i}}{\sum_{k}w_{k}}\log(w_{i}')-\frac{w_{i}}{\sum_{k}w_{k}}\log(w_{i})\\
& =\frac{1}{\sum_{k}w_{k}}\left(\sum_{i}(w_{i}\log(w_{i}')-w_{i}\log(w_{i}))\right).\\
\end{aligned}
\end{align*}
The maximization step consists in finding better parameters than those of the previous iteration.
We call 
$M^{(\gamma)}=(\Pi^{(\gamma)},L^{(\gamma)},K^{(\gamma)})$, the parameters of the current iteration $(\gamma)$.
\\
We are seeking new parameters
$M^{(\gamma+1)}=(\Pi^{(\gamma+1)},L^{(\gamma+1)},K^{(\gamma+1)})$.
\\
Let consider the finite spaces $\Xi_k=\{1,\ldots,m\}^k, k \in \{1,\ldots,\Gamma\}$. We call a possible trajectory of $\Theta$, $\theta$, belonging to the finite space $\Xi_\Gamma$.
Using $
w_{\theta}=\Prob(Z=z,\Theta=\theta|M^{(\gamma)})$ and $w_{\theta}'=\Prob(Z=z,\Theta=\theta|M^{(\gamma+1)})$ in the previous inequality and defining
$Q(M^{(\gamma)},M^{(\gamma+1)})=\sum_{\theta \in \Xi_\Gamma}w_{\theta}\log(w_{\theta}')$
and 
$Q(M^{(\gamma)},M^{(\gamma)})=\sum_{\theta \in \Xi_\Gamma}w_{\theta}\log(w_{\theta})$, we obtain
\begin{align*}
\begin{aligned}
\log\left(\frac{\sum_{\theta \in \Xi_\Gamma}w_{\theta}'}{\sum_{\theta \in \Xi_\Gamma}w_{\theta}}\right)& =\log\left(\frac{\Prob(Z=z|M^{(s+1)})}{\Prob(Z=z|M^{(\gamma)})}\right)\\
& \geq\frac{1}{\Prob(Z=z|M^{(\gamma)})}(Q(M^{(\gamma)},M^{(\gamma+1)})-Q(M^{(\gamma)},M^{(\gamma)})).\\
\end{aligned}
\end{align*}
This last inequality shows that we obtain
$\Prob(Z=z|M^{(\gamma+1)})\geq \Prob(Z=z|M^{(\gamma)})$
by maximizing 
\[Q(M^{(\gamma)},M^{(\gamma+1)}) =\sum_{\theta \in \Xi_\Gamma} \Prob(\Theta=\theta, Z=z|M^{(\gamma)})\log\Prob(\Theta=\theta,Z=z|M^{(\gamma+1)}).\]
We cut $\log(\Prob(Z=z,\Theta=\theta|M^{(\gamma+1)})) =\log(\Prob(Z=z|\Theta=\theta,M^{(\gamma+1)}))+\log(\Prob(\Theta=\theta|M^{(\gamma+1)}))$. Since the processes $(Z^{d})_{d}$ are independent knowing the unobserved factor, we have
\begin{align*}
\begin{aligned}
\log(\Prob(Z=z,\Theta=\theta|M^{(\gamma+1)}))& =\log(\Prob(\theta_{0})) + \sum_{n=1}^\Gamma\log(\Prob(\theta_{n}|\theta_{n-1}))+\sum_{d=1}^Q\log(\Prob(z_{0|\Gamma}^{d}|\theta))\\
& =\log(\Prob(\theta_0)) + \sum_{n=1}^\Gamma\log(\Prob(\theta_{n}|\theta_{n-1}))+\sum_{d=1}^Q\sum_{n=1}^\Gamma\log\Prob(z_{n}^{d}|z_{n-1}^{d},\theta_{n-1}).\\
\end{aligned}
\end{align*}
So,
\begin{align*}
\begin{aligned}
Q(M^{(\gamma)},M^{(\gamma+1)})& =\sum_{\theta \in \Xi_\Gamma}\log(\Prob(\theta_{0}))\Prob(\theta,z|M^{(\gamma)}) +\sum_{\theta \in \Xi_\Gamma}\sum_{n=1}^\Gamma\log(\Prob(\theta_{n}|\theta_{n-1}))\Prob(\theta,z|M^{(\gamma)})\\
& +\sum_{\theta \in \Xi_\Gamma}\sum_{d=1}^Q\sum_{n=1}^\Gamma\log\Prob(z_{n}^{d}|z_{n-1}^{d},\theta_{n-1})\Prob(\theta,z|M^{(\gamma)})\\\\
& =\sum_{\substack{\theta \setminus \theta_{0}\\ \in \Xi_{\Gamma-1}}}\sum_{h=1}^m\log(\Prob(\Theta_{0}=h))\Prob(\Theta_{0}=h,\theta \setminus \theta_{0},z|M^{(\gamma)})\\
& + \sum_{n=1}^\Gamma\sum_{\substack{\theta \setminus (\theta_{n},\theta_{n-1}) \\ \in \Xi_{\Gamma-2}}}\sum_{h,s=1}^m\log(\Prob(\Theta_{n}=h|\Theta_{n-1}=s))\Prob(\theta \setminus (\theta_{n},\theta_{n-1}),\Theta_{n}=h,\Theta_{n-1}=s,z|M^{(\gamma)})\\
& +\sum_{n=1}^\Gamma\sum_{\substack{\theta \setminus \theta_{n-1}\\\in\Xi_{\Gamma-1}}}\sum_{s=1}^m\sum_{d=1}^Q\log(\Prob(z_{n}^{d}|z_{n-1}^{d},\Theta_{n-1}=s))\Prob(\theta \setminus \theta_{n-1},\Theta_{n-1}=s,z|M^{(\gamma)})\\\\
& =\sum_{h=1}^m\log(\Prob(\Theta_{0}=h,z|M^{(\gamma)}))\Pi^{h}+\sum_{n=1}^\Gamma\sum_{h,s=1}^m\log(K^{sh})\Prob(\Theta_{n}=h,\Theta_{n-1}=s,z|M^{(\gamma)})\\
& +\sum_{d=1}^Q\sum_{n=1}^\Gamma\sum_{s=1}^m\sum_{i,r\in \Upsilon}\log(L^{h,ir})\Prob(\Theta_{n-1}=s,z|M^{(\gamma)})\mathds{1}_{[Z_{n}^{d}=r,Z_{n-1}^{d}=i]}.\\
\end{aligned}
\end{align*}
\\
\\
Then, we can maximize by considering the three terms independently. 
We obtain
\begin{align*}
\begin{aligned}
\Pi^{h}=\frac{\Prob(\Theta_{0}=h,Z|M^{(\gamma)})}{\sum_{j=1}^m\Prob(\Theta_{0}=j,Z|M^{(\gamma)})}=\Prob(\Theta_{1}=h|Z,M^{(\gamma)})=\check{u}_{0}(h),
\end{aligned}
\end{align*}
\begin{align*}
\begin{aligned}
L^{s,ir}& =\frac{\sum_{d=1}^Q\sum_{n=1}^\Gamma\Prob(\Theta_{n-1}=s,Z|M^{(\gamma)})\mathds{1}_{[Z_{n}^{d}=r,Z_{n-1}^{d}=i]}}{\sum_{d=1}^Q\sum_{n=1}^\Gamma\sum_{j\in \Upsilon}\Prob(\Theta_{n-1}=s,Z|M^{(\gamma)})\mathds{1}_{[Z_{n}^{d}=j,Z_{n-1}^{d}=i]}}
& =\frac{\sum_{n=1}^{\Gamma}\check{u}_{n-1}(s)\Delta N_n^{ir}}{\sum_{n=1}^{\Gamma}\check{u}_{n-1} (s)Y_n^i},
\end{aligned}
\end{align*}
\begin{align*}
\begin{aligned}
K^{sh}& = \frac{\sum_{n=1}^\Gamma\Prob(\Theta_{n-1}=s,\Theta_{n}=h,Z|M^{(\gamma)})}{\sum_{n=1}^\Gamma\sum_{j=1}^Q\Prob(\Theta_{n-1}=s,\Theta_{n}=j,Z|M^{(\gamma)})}
& =\frac{\sum_{n=1}^\Gamma\check{v}_{n}(s,h)}{\sum_{n=1}^\Gamma\check{u}_{n-1}(s)}.
\end{aligned}
\end{align*}

\section{Continuous-time version of the filter}
\label{ap:continuous-filter}
\subsection{Framework and statements}
\label{setting:continuous-filter}
Let $(\Omega,\bfA,\bfF = (\cF_t)_{t \in [0,T]},\Prob)$,
 be a filtered probability space satisfying the ``usual conditions'' of right-continuity and completeness needed to justify all operations to be made. All stochastic processes encountered are assumed to be adapted to the filtration $\bfF$ and integrable on $[0,T]$. In particular, we have $\bfA=\cF_{T}$.
The time horizon $T$ is supposed to be finite.
\\
\\
Let $N = (N^1, \ldots, N^\rho)$ be a multivariate counting process where $N^i=(N_t^i)_{t \in [0,T]}$, $i=1, \ldots, \rho, $ is a set of
simple counting processes, such that,
$N_t^i =\sum_{0<s \leq t} \Delta N_s^i< \infty$ and $\Delta N_s^i \in \{0,1 \}$, for any $i=1, \ldots, \rho$. It is assumed that these processes admit a predictable $\bfF$ - intensity
$\nu^i =(\nu_t^i)_{t\in [0,T]}$, and that they do not have any common jumps, i.e.,  $\Delta N_t^i \Delta N_t^j =\delta_{ij}\Delta N_t^i$ ie $[N^i,N^j]_{t}=0$ (the continuous martingale part of a counting process being null).  We introduce $\bfF^N = (\cF_t^N)_{t \in [0,T]}$ the natural filtration of the multivariate counting process $N=(N^1,\dots,N^\rho)$.
\\
Let $\Theta$ be a square integrable process of the form
\begin{equation}
\Theta_t = \int_0^t a_s \, \diff s + M_t \,,
\label{eq:Theta_t}
\end{equation}
where $a$ is a $\bfF$-adapted process and $M$ is a square integrable $\bfF$-martingale. We assume that $\Theta$ and $\Delta N^i$ have no common jumps.  
Let $\bfF^{\Theta}$ be the natural filtration of $\Theta$,  augmented with $\Prob-$null sets.. 
The problem is to estimate the states of the unobserved process $\Theta$ using only the information $ \bfF^N$, 
resulting from the observation of the multivariate counting process $ N $. By definition of the conditional expectation,
\[\hat{\Theta}_t = \E [\Theta_t | \cF_t^N]\,.
\]
is the $L^2$ approximation of $\Theta$ knowing $N$.
With the same notation, all the processes $O$ filtered  by $ \cF_t^N $ is written 
\[\hat{O}_t = \E [O_t | \cF_t^N].\]
The main result on univariate point process filtering can be stated in the following way (see \cite{Bremaud1981}, \cite{Karr1991}, \cite{Leijdekker_Spreij2008}, \cite{vanSchuppen1997}). 
The following proposition can be found in \cite{Bremaud1981} but the different are expressed in terms of change measure and are not explicit. Although his result is valid with simultaneous jumps between a counting process $N^i$ and the hidden factor $\Theta$, this leads to an extra term which cannot be computed in practice. We aimed at obtaining an implementable formula so we had to forbid simultaneous jumps. For the sake of completeness, we provide a self contained proof yielding the explicit filtering formula with no simultaneous jumps.
\subsection{General filtering equation}
\begin{proposition}
\label{Prop:filtering_multPP}
The process $\hat{\Theta}$ is solution of the SDE
\begin{equation}
\diff \hat{\Theta}_t = \hat{a}_t \, \diff t + \sum_{j=1}^\rho \eta_t^j \, (\diff N_t^j - \hat{\nu}_{t-}^j \, \diff t)\,,
\label{eq:dhatTheta}
\end{equation}
with
\begin{equation}
\eta_t^j = \frac{ \widehat{ (\Theta\,\nu^j )}_{t-} } { \hat{\nu}_{t-}^j} - \hat{\Theta}_{t-} \,,
\label{eq:eta^j}
\end{equation}
and initial condition
\begin{equation}
\hat{\Theta}_0 = \E [\Theta_0].
\label{eq:hatTheta_0}
\end{equation}
\end{proposition}
\begin{proof}[of Prop. \ref{Prop:filtering_multPP}]
 Let $g$ and $h$ be two $\bfF$ predictable processes such that $\E \left[ \int_0^T (g_s^2+ h_s^2)  \, \nu_s^j \, \diff s \right] \leq \infty$. We introduce the processes $X$ and $Y$ defined by
$X_t=\int_0^t g_{r} (\diff N_r^j - \nu_r^j \, \diff r)$ and $Y_{t}=\int_0^t h_s (\diff N_s^k - \nu_s^k \, \diff s)$ for all $t \le T$. X and Y are two $\bfF$--martingales.
The Itô formula applied to $XY$ yields 
\begin{align*}
\label{int-part}
\diff(X_tY_t) = & X_{t-}\diff Y_{t} + Y_{t-}\diff X_{t} + \Delta X_t \Delta Y_{t}.
\end{align*}
Since $N^j$ and $N^k$ have no common jumps,
$
\Delta X_t \Delta Y_{t}=g_{t}h_{t}\Delta N_{t}^j \delta_{jk}.
$
Then,  we obtain
\begin{equation}
\label{eq1}
\E[X_{T}Y_{T}- X_{t}Y_{t}|\cF_{t}]=\delta_{jk}\E\left[\int_{t}^T g_{s}h_s
\diff N^j_{s} |F_{t}\right].
\end{equation}
Note that $X_T Y_T-X_t Y_t=(X_{T}-X_{t})(Y_{T}-Y_{t})-2X_{t}Y_{t}+X_{t} Y_{T} + X_{T} Y_{t}$. Then, 
\begin{equation}
\label{eq2}
\E[X_{t}Y_{T}-X_{t}Y_{t}|\cF_{t}]=\E[(X_{T}-X_{t})(Y_{T}-Y_{t})|\cF_{t}].
\end{equation}
Note that the process $Z=\left(\int_{0}^t g_{s}h_{s}(\diff N^j_{s}-\nu^j_{s}\diff{s})\right)_{0 \le t \le T}$ is a $\bfF$--martingale.
\\ 
So, $\E\left[Z_{T}-Z_{t}|\cF_{t} \right]=0$. Combining this remark with Equations~\eqref{eq1} and \eqref{eq2}, we finally obtain \begin{equation}
\E \left[ \left. \int_t^T g_r (\diff N_r^j - \nu_r^j \, \diff r) \int_t^T h_s (\diff N_s^k - \nu_s^k \, \diff s) \right| \cF_t \right] = \delta_{jk} \, \E \left[ \left. \int_t^T g_s h_s \, \nu_s^j \, \diff s \right| \cF_t \right] .
\label{eq:EintNjintNk}
\end{equation}
The innovation theorem says that the $\bfF^N$-intensities of the counting processes $N^j$ exist and are
$$
\hat{\nu}_{s-}^j =\E \left[ \nu_{s-}^j \,|\, \cF_{s-}^N \right] =\E \left[ \nu_{s}^j \,|\, \cF_{s-}^N \right] \,.
\label{eq:innovation-rule}
$$
For any $\bfF^N$ predictable process $h$ satisfying $\E\left[\int_0^T |h_s \nu^j_s| ds \right]< \infty$, we have
\begin{align*}
\E \left[\int_{0}^{\infty} h_{s}\diff N^j_{s}\right]
=& \E \left[\int_{0}^{\infty} h_{s}\nu^j_{s} \diff s \right]
= \E \left[\int_{0}^{\infty} h_{s}\E \left[\nu^j_{s}|\cF^N_{s-}\right] \diff s\right]
= \E\left[\int_{0}^{\infty} h_{s}\hat{\nu}^j_{s-} \diff s\right].
\end{align*}
Now, rewrite (\ref{eq:Theta_t}) as
\begin{equation}
\Theta_t = \int_0^t \hat{a}_s\, \diff s + L_t + M_t \,,
\label{eq:Theta-rewrite}
\end{equation}
with
$$
L_t = \int_0^t (a_s - \hat{a}_s)\,ds  \,.
\label{eq:M^*}
$$
Taking conditional expectation, w.r.t. $\cF_t^N$, in (\ref{eq:Theta-rewrite}) yields
\begin{equation}
\hat{\Theta}_t = \int_0^t \hat{a}_s\,ds + \hat{L}_t + \hat{M}_t \,.
\label{eq:hat-Theta-repr}
\end{equation}
While $L$ need not be a $\bfF$--martingale, it is clear that $\hat{L}$ is an $\bfF^N$--martingale. For $t_1 < t_2$,
\begin{align*}
\E[\hat{L}_{t_2} - \hat{L}_{t_1}  \,|\, \cF_{t_1}^N ]= \E \left[ \left. \int_{t_1}^{t_2} \left( a_s - \E [ a_s | \cF_s^N ] \right) \,\diff s \right| \cF_{t_1}^N \right]=0. \\
\end{align*}
From the tower property, we deduce that $\hat{M}$ is also a $\bfF^N$--martingale.
Introduce
$$
K_t = L_t + M_t = \Theta_t - \int_0^t \hat{a}_s\,ds \,.
\label{eq:K}
$$
Since $\Theta$ and $N$ have no common jumps, we can deduce that $K$ and $N$ have any either. 
It has been shown that $\hat{K}= \hat{L} + \hat{M}$ is a $\bfF^N$--martingale. Therefore it has a predictable representation,
\begin{equation}
\hat{K}_t = \ga + \sum_j \int_0^t \eta_s^j \, (\diff N_s^j - \hat{\nu}_{s-}^j \, \diff s) ,
\label{eq:pred-repr}
\end{equation}
where $\ga = \hat{K}_0$ is $\cF_0^N$-measurable and the $\eta^j$ are $\bfF^N$-predictable processes (see \cite{Bremaud1981}). Note that $\hat K_0 = \E[\Theta_0]$.
Now, any integrable $\bfF^N$-measurable random variable has a representation $g + \sum_j \int_0^t h_s^j \, (\diff N_s^j - \hat{\nu}_s^j \, \diff s)$, with $g$ constant and the $h^j$ are $\bfF^N$-predictable. Therefore, since $\hat{K}_t$ is the $L^2$ projection of $K_t$ onto the space of square integrable $\cF_t^N$-measurable random variables, the coefficients in the representation (\ref{eq:pred-repr}) are uniquely determined by the normal equations
$$
\E \bigg[ \bigg( K_t - \ga - \sum_j \int_0^t \eta_s^j (\diff N_s^j - \hat{\nu}_{s-}^j \, \diff s) \bigg) \bigg( g + \sum_j \int_0^t h_s^j (\diff N_s^j - \hat{\nu}_{s-}^j \, \diff s)  \bigg) \bigg] = 0,
$$
for all constants $g$ and all $\bfF^N$-predictable processes $h^j$.  Setting $g = 0$ and using (\ref{eq:EintNjintNk}) give
$$
\label{eq:inserting}
\E \left[ K_t \, \sum_j \int_0^t h_s^j \, (\diff N_s^j - \hat{\nu}_{s-}^j \, \diff s) -
\sum_j h_s^j \, \eta_s^j \, \hat{\nu}_{s-}^j \, \diff s   \right] = 0 \,.
$$
For j $\in \{1, \dots,\rho\}$, we compute $\E \left[K_{t}\int_0^t h_s^j \diff N_s^j \right]$. Using that $\hat K$  is a $\bfF^N$--martingale and that $K$ and $N^j$ have no common jumps, we have
\begin{align*}
\begin{aligned}
\E \left[K_{t}\int_0^t h_s^j \diff N_s^j \right]
= & \E \left[K_{t}\sum_{s \leq t} h_s^j \Delta N_s^j \right]
= \sum_{s \leq t}\E \left[\E\left[K_{t}|\cF_{s}^N \right] h_s^j \Delta N_s^j \right]\\
= & \sum_{s \leq t}\E \left[\hat{K}_{s} h_s^j \Delta N_s^j \right] = \E \left[\int_0^t  K_{s-}h_{s}^j \diff N_s^j \right]\\
= & \E \left[\int_0^t K_{s-}h_{s}^j \nu_s^j \diff{s} \right]
=  \int_0^t \E \left[h_{s}^j \E\left[ K_{s-} \nu_s^j| \cF_{s-}^N \right] \right] \diff{s}\\
= & \E \left[\int_{0}^t h_{s}^j \widehat{\Theta \nu^j}_{s-} \diff s \right] - \E \left[\int_{0}^t h^j_{s} \hat{\nu}_{s-}^j\int_{0}^s \hat{a}_{u}\diff u  \ \diff s \right].
\end{aligned}
\end{align*}
Using similar arguments, we compute the second term 
\begin{align*}
\begin{aligned}
\E \left[K_{t}\int_{0}^t h^j_s \hat{\nu}^j_{s-} \diff s\right]
= &  \int_{0}^t \E \left[ h^j_s \E \left[{K}_{t}|\cF_{s-}^N\right]  \hat{\nu}^j_{s-} \right]\diff s 
=  \int_{0}^t \E \left[ h^j_s \hat{K}_{s-}  \hat{\nu}^j_{s-} \right]\diff s \\
= & \E\left[\int_{0}^t h^j_s \hat{\Theta}_{s-}\hat{\nu}^j_{s-} \diff s \right] - \E\left[\int_{0}^t h^j_s \hat{\nu}^j_{s-} \int_{0}^s\hat{a}_{u} \diff u \ \diff s \right]. 
\end{aligned}
\end{align*}
Inserting these expressions into (\ref{eq:inserting}), gives
\begin{align*}
\sum_j \E \left[ \int_0^t h_s^j \left( \widehat{{\Theta}\, \nu^j}_{s-} - \hat{\Theta}_{s-} \, \hat{\nu}_{s-}^j - \eta_s^j \,\hat{\nu}_{s-}^j \right) \, \diff s  \right] = 0.
\end{align*}
Choosing $h_s^j$ equal to the expression in the parentheses, gives $\sum_j \E \left[ \int_0^t (h_s^j)^2 \, \diff s  \right] = 0$ hence all $h^j$ vanish and $\forall \ j = 1,\ldots,\rho$:
\begin{equation}
\eta_s^j = \frac{ \widehat{ {\Theta}_{s-}\,\nu_{s-}^j } } { \hat{\nu}_{s-}^j} - \hat{\Theta}_{s-}.
\label{eq:eta}
\end{equation}
From (\ref{eq:pred-repr}), (\ref{eq:hat-Theta-repr}), and the equality $\hat{K} = \hat{L} + \hat{M}$, it follows that
$$
\hat{\Theta}_t = \E[\Theta_0] + \int_0^t \hat{a}_s\, \diff s + \sum_j \int_0^t \eta_s^j \, (\diff N_s^j - \hat{\nu}_{s-}^j \, \diff s)
$$
with the $\eta^j$ are given by (\ref{eq:eta}). 
This finishes the proof of the proposition. \hfill
\end{proof}

\subsection{Finite latent factor model and a credit risk application}
\label{sec:rating-continuous-demo}
\begin{proof}[of Prop. \ref{prop:continuous-rating}]
In order to apply Prop. \ref{Prop:filtering_multPP}, one needs to find the representation (\ref{eq:Theta_t}) for $I_t^h$. 
Let $\Psi^{rh}$, $r \neq h$, $r,h \in \T$, be the counting processes defined by
\begin{align*}
\Psi_t^{rh} = \sharp \{ s \in (0,t];\, \Theta_{s-} = r, \, \Theta_s = h \}\,.
\end{align*}
The starting point is the expression
\begin{align*}
I_t^h = I_0^h + \sum_{r; r \neq h} (\Psi_t^{rh} -  \Psi_t^{hr}) \,,
\label{ThetaIndicator}
\end{align*}
which comes from the obvious dynamics
\begin{align*}
dI_t^h = \sum_{r; r \neq h} (\diff \Psi_t^{rh} - \diff \Psi_t^{hr})
\end{align*}
The counting processes $\Psi^{rh}$ have intensities of the form $I_{t-}^r \, \ka^{rh}$. Reshaping the last expression as
\begin{align*}
I_t^h &= I_0^h + \int_0^t \sum_{r; r \neq h} ( I_{s-}^r \, \ka^{rh} - I_{s-}^h \, \ka^{hr} ) \, \diff s \\
&+ \int_0^t \sum_{r; r \neq h} \left[ (\diff \Psi_s^{rh} - I_{s-}^r \, \ka^{rh}\, \diff s)  - (\diff \Psi_s^{hr} - I_{s-}^h \, \ka^{hr} \, \diff s ) \right],
\end{align*}
shows that $I^h$ is of the form
\begin{align*}
I_t^h = \int_0^t a_s^h \, \diff s + M_t^h \,,
\end{align*}
with
\begin{equation}
a_t^h = \sum_{r; r \neq h} ( I_{t-}^r \, \ka^{rh} - I_{t-}^h \, \ka^{hr} ) = \sum_r \ka^{rh} \, I_{t-}^r
\label{eq:a}
\end{equation}
and $M^h$ is a martingale commencing at $M_0^h = I_0^h$.\\
Then the role of $a_t$ is taken by $a_t^h$ in (\ref{eq:a}), the role of $(\Theta \, \nu^{j})_{t-}$ is taken by 
\begin{align*}
(I^h\nu^{ij})_{t-}=I_{t-}^h \, Y_t^i \, \sum_r \ell^{r,ij} \, I_{t-}^r = Y_t^i \, \ell^{r,ij} \, I_{t-}^h \,,
\end{align*}
and the $\bfF^N$-intensities of $N$ are given in (\ref{sigmaFil}).
Inserting these expressions
into (\ref{eq:dhatTheta}), gives
$$
d\hat{I}_{t}^{h}=\sum_{r=1}^{m}k^{rh}\hat{I}_{t-}^{r}dt+\sum_{i\neq j}\left(\frac{l^{h,ij}\hat{I}_{t-}^{h}}{\sum_{r}l^{r,ij}\hat{I}_{t-}^{r}}-\hat{I}_{t-}^{h}\right)\left(dN_{t}^{ij}-Y_{t}^{i}\sum_{r=1}^{m}l^{r,ij}\hat{I}_{t-}^{r}dt\right)
$$
\end{proof}
This result may look similar to R4 of \cite[Sec.IV.1]{Bremaud1981} but we actually consider a more general framework. On the hand, we deal with an aggregated version over the entire portfolio of the multivariate process and on the other hand we take censorship into account though the processes of risk exposure $Y^i$.

\subsection{Calibration of the continuous version}
\label{proof:calibration-continuousBaum}

Here, we present the detailed computations of the adaptations of the calibration for the continuous filtering framework, presented in Section \ref{seq: continuous-adaptation}.
In practice the number of entities monitored over time may vary: either because some names appear or disappear or simply because of missing data. This happens when the data is missing, censored or when it is not appeared yet. We attribute the rating 0 to an entity in this case. Then it is clear that a transition involving the rating of censure 0, is assumed to be independent with the states of the hidden factor. Let consider the list of ratings $\bar{\Upsilon}= \{1, \ldots, p \}$ and ${\Upsilon}= \{0, \ldots, p \}$, the completed list of ratings. Note that the total number of entities observed on ${\Upsilon}$ is constant equal to Q. Let $Q_t$, be the number of entities which have their rating in $\bar{\Upsilon}$ (have a real rating) at time t.
\\
We propose a calibration algorithm which assumes that no more than one entity may jump at a given time step. In order to make the model identifiable while considering the impact of the size of the sample (which may evolve), we define an independent process $I$, with values in $\{0, \dots, Q\}$, which uniformly picks the entity that may jump. 
If $I$ picks an entity which is rated 0, (because not already rated or censored), we do not observe jumps. Otherwise the entity jumps according to the transition matrices $(L^{h})_{h}$. 
\\
\\
We have $\forall t \in \{0,\ldots,\Gamma\},\ (i,j) \in \bar{\Upsilon}^2,\ h \in \T,\ q \in \{0,\ldots,Q\}$
\[\Prob(Z_{t}^q=j|Z_{t-1}^q=i,I_{t-1}=q,\Theta_{t-1}=h)=L^{h,ij}.\]
For $z_t, z_{t-1} \in \Upsilon^Q$, we define $W_{t-1}^{h}=\Prob(Z_{t}=z_{t}|Z_{t-1}=z_{t-1},\Theta_{t-1}=h)$, where $Z_t=(Z_t^q)_{q\leq Q_t}$.
\\
We compute 
\begin{align*}
\begin{aligned}
W_{t-1}^{h}& =\sum_{d=1}^Q\Prob(Z_{t}=z_{t}|Z_{t-1}=z_{t-1},I_{t-1}=d,\Theta_{t-1}=h)\Prob(I_{t-1}=d)\\
& =\sum_{d=1}^{Q_{t-1}}\Prob(Z_{t}=z_{t}|Z_{t-1}=z_{t-1},I_{t-1}=d,\Theta_{t-1}=h)\Prob(I_{t-1}=d)\\
& +\sum_{d=Q_{t-1}+1}^{Q}\Prob(Z_{t}=z_{t}|Z_{t-1}=z_{t-1},I_{t-1}=d,\Theta_{t-1}=h)\Prob(I_{t-1}=d)
\end{aligned}
\end{align*}
Let focus on the first sum, describing the situation when the chosen entity has a rating at current time.
\begin{align*}
\begin{aligned}
\Prob(Z_{t}=z_{t}|Z_{t-1}=z_{t-1},I_{t-1}=d,\Theta_{t-1}=h)
& =\Prob(Z_{t}=z_{t}|Z_{t-1}=z_{t-1},I_{t-1}=d,\Theta_{t-1}=h)\mathds{1}_{[z_{t}^d=z_{t-1}^d, \forall l \neq d: z_{t}^l=z_{t-1}^l]}\\
& +\Prob(Z_{t}=z_{t}|Z_{t-1}=z_{t-1},I_{t-1}=d,\Theta_{t-1}=h)\mathds{1}_{[z_{t}^d\neq z_{t-1}^d, \forall l \neq d: z_{t}^l=z_{t-1}^l]}\\
& =L^{h,z_{t-1}^d z_{t}^d}\mathds{1}_{[z_{t}=z_{t-1}]}+L^{h,z_{t-1}^d z_{t}^d}\mathds{1}_{[z_{t}^d\neq z_{t-1}^d, \forall l \neq d: z_{t}^l=z_{t-1}^l]}\\
\end{aligned}
\end{align*}
For the second sum, we have: 
$
\Prob(Z_{t}=z_{t}|Z_{t-1}=z_{t-1},I_{t-1}=d,\Theta_{t-1}=h)=\mathds{1}_{[z_{t}=z_{t-1}]}.
$
\\
\\
So finally,
\begin{align*}
\label{adaptation}
W_{t-1}^{h}=(1-\frac{Q_{t-1}}{Q})\mathds{1}_{[z_{t}=z_{t-1}]}+\sum_{d=1}^{Q_{t-1}}\frac{1}{Q}\mathds{1}_{[|z_{t}-z_{t-1}|_{0}\leq 1]}\mathds{1}_{[\forall l \neq d \ z_{t}^{l}=z_{t-1}^{l}]}L^{h,z_{t-1}^d,z_{t}^d},
\end{align*}
where $|x|_{0}= \# \{x_{i}\neq 0\}$.
\\
Then, it is easy to check that the previous algorithm can be adapted to the new framework
\begin{align*}
\alpha_{t}(h)=\sum_{s=1}^{m}\alpha_{t-1}(s)K^{sh}W_{t-1}^h,
\end{align*}
\begin{align*}
\beta_{t}(h)=\sum_{l=1}^{m}\beta_{t+1}(l)K^{hl}W_{t}^h,
\end{align*}
\begin{align*}
\check{u}_{t}(h)=\Prob(\Theta_{t}=h|Z_{0|\Gamma}=z_{0|\Gamma})=\frac{\beta_{t+1}(h)\alpha_{t+1}(h)}{L_{\Gamma}},
\end{align*}
\begin{align*}
\check{v}_{t}(s,h)=\Prob(\Theta_{t}=h,\Theta_{t-1}=s|Z_{0|\Gamma}=z_{0|\Gamma})=\frac{\beta_{t+1}(h)K^{sh}\alpha_{t}(s)W_{t}^h}{L_{\Gamma}}.
\end{align*}
The forms of the transitions matrices $(L^{h})_{h}$, are a lot impacted by this adaptation. The maximisation does not run as simply as it does for the discrete setting. Explicit forms are heavy to derive. Then, these parameters are directly estimated with optimization algorithms.  
\section{Parameter estimations}
\label{ap:parameters}
In this appendix, we present the parameters chosen for the simulation in the testing framework described in Section \ref{sec:fictiveData}, with the estimated parameters issued from the EM algorithm. 

\begin{table}[H]
   \centering
    \footnotesize
    \caption{Initial and estimated rating transition matrix for $\Theta=0$}
    \begin{tabular}{|*{5}{c|}}
        \hhline{~*{3}{-}}
        \multicolumn{1}{c|}{} & $A$ & $B$ & $C$ \\ \hline
        $A$ & {0.98}&	{0.01}&	{0.01}
 \\ \hline
        $B$ & {0.29}&	{0.7}&	{0.01}
   \\ \hline
        $C$ & {0.1}&	{0.3}&	{0.6}
   \\ \hline
    \end{tabular}
    \hspace*{1cm}
    \begin{tabular}{|*{5}{c|}}
        \hhline{~*{3}{-}}
        \multicolumn{1}{c|}{} & $A$ & $B$ & $C$ \\ \hline
        $A$ & {0.9799}&	{0.0099}&	{0.0102}
\\ \hline
        $B$ & {0.2923}&	{0.6977}&	{0.0100}
   \\ \hline
        $C$ & {0.1023}&	{0.2962}&   {0.6016}
    \\ \hline
    \end{tabular}
     \label{tab:L0}
\end{table}

\begin{table}[H]
   \centering
    \footnotesize
    \caption{Initial and estimated rating transition matrix for $\Theta=1$}
    \begin{tabular}{|*{5}{c|}}
        \hhline{~*{3}{-}}
        \multicolumn{1}{c|}{} & $A$ & $B$ & $C$ \\ \hline
        $A$ & {0.98}&	{0.01}&	{0.01}
 \\ \hline
        $B$ & {0.39}&	{0.6}&	{0.01}
   \\ \hline
        $C$ & {0.2}&	{0.3}&	{0.5}
    \\ \hline
    \end{tabular}
    \hspace*{1cm}
    \begin{tabular}{|*{5}{c|}}
        \hhline{~*{3}{-}}
        \multicolumn{1}{c|}{} & $A$ & $B$ & $C$ \\ \hline
        $A$ & {0.9803}&	{0.0100}&	{0.0097}
 \\ \hline
        $B$ & {0.3887}&	{0.6018}&	{0.0095}
   \\ \hline
        $C$ & {0.2002}&	{0.3003}&	{0.4995}
    \\ \hline
    \end{tabular}
     \label{tab:L1}
\end{table}

\begin{table}[H]
   \centering
    \footnotesize
    \caption{Initial and estimated rating transition matrix for $\Theta=2$}
    \begin{tabular}{|*{5}{c|}}
        \hhline{~*{3}{-}}
        \multicolumn{1}{c|}{} & $A$ & $B$ & $C$ \\ \hline
        $A$ & {0.5}&	{0.3}&	{0.2}
 \\ \hline
        $B$ & {0.01}&	{0.6}&	{0.39}
   \\ \hline
        $C$ & {0.01}&	{0.01}&	{0.98}
    \\ \hline
    \end{tabular}
    \hspace*{1cm}
    \begin{tabular}{|*{5}{c|}}
        \hhline{~*{3}{-}}
        \multicolumn{1}{c|}{} & $A$ & $B$ & $C$ \\ \hline
        $A$ & {0.5072}&	{0.3002}&	{0.1926}
 \\ \hline
        $B$ & {0.0095}&	{0.6004}&	{0.3901}
   \\ \hline
        $C$ & {0.0099}&	{0.0103}&	{0.9798}
    \\ \hline
    \end{tabular}
     \label{tab:L2}
\end{table}

\begin{table}[H]
   \centering
    \footnotesize
    \caption{Initial and estimated rating transition matrix for $\Theta=3$}
    \begin{tabular}{|*{5}{c|}}
        \hhline{~*{3}{-}}
        \multicolumn{1}{c|}{} & $A$ & $B$ & $C$ \\ \hline
        $A$ & {0.98}&	{0.01}&	{0.01}
 \\ \hline
        $B$ & {0.39}&	{0.6}&	{0.01}
   \\ \hline
        $C$ & {0.2}&	{0.3}&	{0.5}
    \\ \hline
    \end{tabular}
    \hspace*{1cm}
    \begin{tabular}{|*{5}{c|}}
        \hhline{~*{3}{-}}
        \multicolumn{1}{c|}{} & $A$ & $B$ & $C$ \\ \hline
        $A$ & {0.9803}&	{0.0100}&	{0.0097}
 \\ \hline
        $B$ & {0.3887}&	{0.6018}&	{0.0095}
   \\ \hline
        $C$ & {0.2002}&	{0.3003}&	{0.4995}
   \\ \hline
    \end{tabular}
     \label{tab:L3}
\end{table}

\begin{table}[H]
   \centering
    \footnotesize
    \caption{Initial and estimated rating transition matrix for $\Theta=4$}
    \begin{tabular}{|*{5}{c|}}
        \hhline{~*{3}{-}}
        \multicolumn{1}{c|}{} & $A$ & $B$ & $C$ \\ \hline
        $A$ & {0.6}&	{0.3}&	{0.1}
 \\ \hline
        $B$ & {0.01}&	{0.7}&	{0.29}
   \\ \hline
        $C$ & {0.01}&	{0.01}&	{0.98}
   \\ \hline
    \end{tabular}
    \hspace*{1cm}
    \begin{tabular}{|*{5}{c|}}
        \hhline{~*{3}{-}}
        \multicolumn{1}{c|}{} & $A$ & $B$ & $C$ \\ \hline
        $A$ & {0.5993} &	{0.2992}&	{0.1015}
 \\ \hline
        $B$ & {0.0104}&	{0.6983}&	{0.2913}
  \\ \hline
        $C$ & {0.0099}&	{0.0095}&	{0.9806}
    \\ \hline
    \end{tabular}
     \label{tab:L4}
\end{table}

\begin{table}[H]
   \centering
    \footnotesize
    \caption{Initial and estimated rating transition matrix for $\Theta=5$}
    \begin{tabular}{|*{5}{c|}}
        \hhline{~*{3}{-}}
        \multicolumn{1}{c|}{} & $A$ & $B$ & $C$ \\ \hline
        $A$ & {0.8}&	{0.15}&	{0.05}
 \\ \hline
        $B$ & {0.01}&	{0.9}&	{0.09}
   \\ \hline
        $C$ & {0.01}&	{0.01}&	{0.98}
    \\ \hline
    \end{tabular}
    \hspace*{1cm}
    \begin{tabular}{|*{5}{c|}}
        \hhline{~*{3}{-}}
        \multicolumn{1}{c|}{} & $A$ & $B$ & $C$ \\ \hline
        $A$ & {0.8001}&	{0.1493}&	{0.0506}
\\ \hline
        $B$ & {0.0099}&	{0.8996}&	{0.0904}
   \\ \hline
        $C$ & {0.0101}&	{0.0098}&	{0.9801}
    \\ \hline
    \end{tabular}
     \label{tab:L5}
\end{table}

\begin{table}[H]
   \centering
    \footnotesize
    \caption{Initial and estimated rating transition matrix for $\Theta=6$}
    \begin{tabular}{|*{5}{c|}}
        \hhline{~*{3}{-}}
        \multicolumn{1}{c|}{} & $A$ & $B$ & $C$ \\ \hline
        $A$ & {0.98}&	{0.01}&	{0.01}
 \\ \hline
        $B$ & {0.09}&	{0.9}&	{0.01}
   \\ \hline
        $C$ & {0.05}&	{0.15}&	{0.8}
    \\ \hline
    \end{tabular}
    \hspace*{1cm}
    \begin{tabular}{|*{5}{c|}}
        \hhline{~*{3}{-}}
        \multicolumn{1}{c|}{} & $A$ & $B$ & $C$ \\ \hline
        $A$ & {0.9799}&	{0.0102}&	{0.0099}
 \\ \hline
        $B$ & {0.0908}&	{0.8991}&	{0.0101}
  \\ \hline
        $C$ & {0.0500}&	{0.1517}&	{0.7984}
   \\ \hline
    \end{tabular}
     \label{tab:L6}
\end{table}

\begin{table}[H]
    \centering
    \footnotesize
    \caption{$\Theta$'s transition matrix}
    \begin{tabular}{|*{9}{c|}}
        \hhline{~*{7}{-}}
        \multicolumn{1}{c|}{} & $\Theta=0$ & $\Theta=1$ & $\Theta=2$ & $\Theta=3$ & $\Theta=4$  & $\Theta=5$ & $\Theta=6$ \\ \hline
        $\Theta=0$ & {0.6}&	{0.3}&	{0.1}&	{0}&	{0}&	{0}&	{0}

   \\ \hline
        $\Theta=1$ & {0.25}&	{0.4}&	{0.25}&	{0.1}&	{0}	&{0}&	{0}
   \\ \hline
        $\Theta=2$ & {0.05}&	{0.15}&	{0.6}&	{0.15}	&{0.05}&	{0}&	{0}
   \\ \hline
        $\Theta=3$ & {0}&	{0.03}&	{0.12}&	{0.7}&	{0.12}&	{0.03}&	{0}

   \\ \hline
        $\Theta=4$ & {0}&	{0}&	{0.05}&	{0.15}&	{0.6}&	{0.15}&	{0.05}

 \\ \hline
        $\Theta=5$ & {0}&	{0}&	{0}&	{0.1}&	{0.25}&	{0.4}&	{0.25}

    \\ \hline
        $\Theta=6$ & {0}&	{0}&	{0}&	{0}&	{0.1}&	{0.3}&	{0.6}

 \\ \hline
    \end{tabular}
     \label{tab:K}
\end{table}

\begin{table}[H]
    \centering
    \footnotesize
    \caption{Estimated $\Theta$'s transition matrix}
    \begin{tabular}{|*{9}{c|}}
        \hhline{~*{7}{-}}
        \multicolumn{1}{c|}{} & $\Theta=0$ & $\Theta=1$ & $\Theta=2$ & $\Theta=3$ & $\Theta=4$  & $\Theta=5$ & $\Theta=6$ \\ \hline
        $\Theta=0$ & {0.6029}&	{0.3426}&	{0.0544}&	{0}&	{0}&	{0}&	{0}
   \\ \hline
        $\Theta=1$ & {0.2827}&	{0.4015}&	{0.2526}&	{0.0632} & {0}  & {0} & {0}   \\ \hline
        $\Theta=2$ & {0.0825}&	{0.1546}&	{0.5773}&	{0.1443}&	{0.0412}&	{0}&	{0}
   \\ \hline
        $\Theta=3$ & {0}&	{0.0583}&	{0.1083}&	{0.7000}&	{0.0750}&	{0.0583}&	{0}

   \\ \hline
        $\Theta=4$ & {0}&	{0}&	{0.0716}&	{0.1592}&	{0.5188}&	{0.2124}&	{0.0381}
 \\ \hline
        $\Theta=5$ & {0}&	{0}&	{0}&	{0.1765}&	{0.2475}&	{0.3053}&	{0.2708}
    \\ \hline
        $\Theta=6$ & {0}&	{0}&	{0}&	{0}&	{0.1237}&	{0.2851}&	{0.5912}
        \\\hline
    \end{tabular}
     \label{tab:estimated-K}
\end{table}

%% file: main.bbl
\begin{thebibliography}{10}

\bibitem{altman1992implications}
Edward~I Altman and Duen~Li Kao.
\newblock The implications of corporate bond ratings drift.
\newblock {\em Financial Analysts Journal}, 48(3):64--75, 1992.

\bibitem{PDLGD2017}
European~Banking Authority.
\newblock Guidelines on pd estimation, lgd estimation and the treatment of
  defaulted exposures.
\newblock Technical report, 2017.

\bibitem{bangia2002ratings}
Anil Bangia, Francis~X Diebold, Andr{\'e} Kronimus, Christian Schagen, and Til
  Schuermann.
\newblock Ratings migration and the business cycle, with application to credit
  portfolio stress testing.
\newblock {\em Journal of banking \& finance}, 26(2-3):445--474, 2002.

\bibitem{baum1970maximization}
Leonard~E Baum, Ted Petrie, George Soules, and Norman Weiss.
\newblock A maximization technique occurring in the statistical analysis of
  probabilistic functions of markov chains.
\newblock {\em The annals of mathematical statistics}, 41(1):164--171, 1970.

\bibitem{bishop2006pattern}
Christopher~M Bishop.
\newblock {\em Pattern recognition and machine learning}.
\newblock Springer New York, 2006.

\bibitem{Bremaud1981}
Pierre Br{\'e}maud.
\newblock {\em Point processes and queues: martingale dynamics}, volume~50.
\newblock Springer, 1981.

\bibitem{caja2015influence}
Anisa Caja, Quentin Guibert, and Fr{\'e}d{\'e}ric Planchet.
\newblock Influence of economic factors on the credit rating transitions and
  defaults of credit insurance business.
\newblock Technical report, 2015.

\bibitem{Carty1997}
Lea Carty.
\newblock Moody's rating migration and credit quality correlation.
\newblock {\em Moody’s Sepcial Report July}, 1997.

\bibitem{ching2009modeling}
Wai-Ki Ching, Tak~Kuen Siu, Li-min Li, Tang Li, and Wai-Keung Li.
\newblock Modeling default data via an interactive hidden markov model.
\newblock {\em Computational Economics}, 34(1):1--19, 2009.

\bibitem{Reda2015}
A.~Cousin and M.~R. Kheliouen.
\newblock A comparative study on the estimation of factor migration models.
\newblock {\em Bulletin français d'actuariat}, 2015.

\bibitem{damian2018algorithm}
Camilla Damian, Zehra Eksi, and R{\"u}diger Frey.
\newblock Em algorithm for markov chains observed via gaussian noise and point
  process information: Theory and case studies.
\newblock {\em Statistics \& Risk Modeling}, 35(1-2):51--72, 2018.

\bibitem{dempster1977maximum}
Arthur~P Dempster, Nan~M Laird, and Donald~B Rubin.
\newblock Maximum likelihood from incomplete data via the em algorithm.
\newblock {\em Journal of the Royal Statistical Society: Series B
  (Methodological)}, 39(1):1--22, 1977.

\bibitem{deroose2008reviewing}
Servaas Deroose, Werner Roeger, and Sven Langedijk.
\newblock Reviewing adjustment dynamics in emu: from overheating to
  overcooling.
\newblock {\em European Economy Economic Paper}, 2008.

\bibitem{duffie2007multi}
Darrell Duffie, Leandro Saita, and Ke~Wang.
\newblock Multi-period corporate default prediction with stochastic covariates.
\newblock {\em Journal of Financial Economics}, 83(3):635--665, 2007.

\bibitem{Stress2018}
European Banking~Authorities EBA.
\newblock Guidelines on institutions’ stress testing.
\newblock Technical report, 2018.

\bibitem{elliott2008hidden}
Robert~J Elliott, Lakhdar Aggoun, and John~B Moore.
\newblock {\em Hidden Markov models: estimation and control}, volume~29.
\newblock Springer Science \& Business Media, 2008.

\bibitem{elliott2014double}
Robert~J Elliott, Tak~Kuen Siu, and Eric~S Fung.
\newblock A double hmm approach to altman z-scores and credit ratings.
\newblock {\em Expert Systems with Applications}, 41(4):1553--1560, 2014.

\bibitem{feng2008ordered}
Dingan Feng, Christian Gouri{\'e}roux, and Joann Jasiak.
\newblock The ordered qualitative model for credit rating transitions.
\newblock {\em Journal of Empirical Finance}, 15(1):111--130, 2008.

\bibitem{figlewski2012modeling}
Stephen Figlewski, Halina Frydman, and Weijian Liang.
\newblock Modeling the effect of macroeconomic factors on corporate default and
  credit rating transitions.
\newblock {\em International Review of Economics \& Finance}, 21(1):87--105,
  2012.

\bibitem{fledelius2004non}
Peter Fledelius, David Lando, and Jens~Perch Nielsen.
\newblock Non-parametric analysis of rating transition and default data.
\newblock {\em Journal of Investment Management}, 2(2), 2004.

\bibitem{fontana2010credit}
Claudio Fontana and Wolfgang~J Runggaldier.
\newblock Credit risk and incomplete information: filtering and em parameter
  estimation.
\newblock {\em International Journal of Theoretical and Applied Finance},
  13(05):683--715, 2010.

\bibitem{Bale2017}
Bank for International~Settlements.
\newblock Basel committee on banking supervision basel iii: Finalising
  post-crisis reforms.
\newblock {\em Official Journal of the European Union}, 2017.

\bibitem{frey2012pricing}
R{\"u}diger Frey and Thorsten Schmidt.
\newblock Pricing and hedging of credit derivatives via the innovations
  approach to nonlinear filtering.
\newblock {\em Finance and Stochastics}, 16(1):105--133, 2012.

\bibitem{Gagliardini2005}
Patrick Gagliardini and Christian Gouri{\'e}roux.
\newblock Stochastic migration models with application to corporate risk.
\newblock {\em Journal of Financial Econometrics}, 3(2):188--226, 2005.

\bibitem{giampieri2005analysis}
Giacomo Giampieri, Mark Davis, and Martin Crowder.
\newblock Analysis of default data using hidden markov models.
\newblock {\em Quantitative Finance}, 5(1):27--34, 2005.

\bibitem{hamilton2004rating}
David~T Hamilton.
\newblock Rating transitions and defaults conditional on watchlist, outlook and
  rating history.
\newblock {\em Outlook and Rating History (February 2004)}, 2004.

\bibitem{jarrow1997markov}
Robert~A Jarrow, David Lando, and Stuart~M Turnbull.
\newblock A markov model for the term structure of credit risk spreads.
\newblock {\em The review of financial studies}, 10(2):481--523, 1997.

\bibitem{Karr1991}
Alan Karr.
\newblock {\em Point processes and their statistical inference}.
\newblock Routledge, 2017.

\bibitem{kavvathas2001estimating}
Dimitrios Kavvathas.
\newblock Estimating credit rating transition probabilities for corporate
  bonds.
\newblock In {\em AFA 2001 New Orleans Meetings}, 2001.

\bibitem{koopman2008multi}
Siem~Jan Koopman, Andr{\'e} Lucas, and Andr{\'e} Monteiro.
\newblock The multi-state latent factor intensity model for credit rating
  transitions.
\newblock {\em Journal of Econometrics}, 142(1):399--424, 2008.

\bibitem{korolkiewicz2008hidden}
Ma{\l}gorzata~W Korolkiewicz and Robert~J Elliott.
\newblock A hidden markov model of credit quality.
\newblock {\em Journal of Economic Dynamics and Control}, 32(12):3807--3819,
  2008.

\bibitem{lando2002analyzing}
David Lando and Torben~M Sk{\o}deberg.
\newblock Analyzing rating transitions and rating drift with continuous
  observations.
\newblock {\em Journal of banking \& finance}, 26(2-3):423--444, 2002.

\bibitem{Leijdekker_Spreij2008}
Vincent Leijdekker and Peter Spreij.
\newblock Explicit computations for a filtering problem with point process
  observations with applications to credit risk.
\newblock {\em Probability in the Engineering and Informational Sciences},
  25(3):393--418, 2011.

\bibitem{liu2014proper}
Tingting Liu, Jan Lemeire, and Lixin Yang.
\newblock Proper initialization of hidden markov models for industrial
  applications.
\newblock In {\em 2014 IEEE China summit \& international conference on signal
  and information processing (ChinaSIP)}, pages 490--494. IEEE, 2014.

\bibitem{liu2017learning}
Yu-Ying Liu, Alexander Moreno, Shuang Li, Fuxin Li, Le~Song, and James~M Rehg.
\newblock Learning continuous-time hidden markov models for event data.
\newblock In {\em Mobile Health}, pages 361--387. Springer, 2017.

\bibitem{merton1974pricing}
Robert~C Merton.
\newblock On the pricing of corporate debt: The risk structure of interest
  rates.
\newblock {\em The Journal of finance}, 29(2):449--470, 1974.

\bibitem{nickell2000stability}
Pamela Nickell, William Perraudin, and Simone Varotto.
\newblock Stability of rating transitions.
\newblock {\em Journal of Banking \& Finance}, 24(1-2):203--227, 2000.

\bibitem{nodelman2012expectation}
Uri Nodelman, Christian~R Shelton, and Daphne Koller.
\newblock Expectation maximization and complex duration distributions for
  continuous time bayesian networks.
\newblock {\em arXiv preprint arXiv:1207.1402}, 2012.

\bibitem{oh2019estimation}
Sung~Youl Oh, Jae~Wook Song, Woojin Chang, and Minhyuk Lee.
\newblock Estimation and forecasting of sovereign credit rating migration based
  on regime switching markov chain.
\newblock {\em IEEE Access}, 7:115317--115330, 2019.

\bibitem{Ozcan2013}
Emre {\"O}zkan, Fredrik Lindsten, Carsten Fritsche, and Fredrik Gustafsson.
\newblock Recursive maximum likelihood identification of jump markov nonlinear
  systems.
\newblock {\em IEEE Transactions on Signal Processing}, 63(3):754--765, 2014.

\bibitem{Rabinet1989}
Lawrence~R Rabiner.
\newblock A tutorial on hidden markov models and selected applications in
  speech recognition.
\newblock {\em Proceedings of the IEEE}, 77(2):257--286, 1989.

\bibitem{IFRS92016}
European~Commission Regulation.
\newblock Amending regulation (ec) no 1126/2008 adopting certain international
  accounting standards in accordance with regulation (ec) no 1606/2002 of the
  european parliament and of the council as regards international financial
  reporting standard 9.
\newblock {\em Official Journal of the European Union}, 2016.

\bibitem{schwaab2017global}
Bernd Schwaab, Siem~Jan Koopman, and Andr{\'e} Lucas.
\newblock Global credit risk: World, country and industry factors.
\newblock {\em Journal of Applied Econometrics}, 32(2):296--317, 2017.

\bibitem{Tenyakov2014}
Anton Tenyakov.
\newblock Estimation of hidden markov models and their applications in finance.
\newblock {\em Electronic Thesis and Dissertation Repository}, 2014.

\bibitem{thomas2002hidden}
Lyn~C Thomas, David~E Allen, and Nigel Morkel-Kingsbury.
\newblock A hidden markov chain model for the term structure of bond credit
  risk spreads.
\newblock {\em International Review of Financial Analysis}, 11(3):311--329,
  2002.

\bibitem{tobin1958estimation}
James Tobin.
\newblock Estimation of relationships for limited dependent variables.
\newblock {\em Econometrica: journal of the Econometric Society}, pages 24--36,
  1958.

\bibitem{vanSchuppen1997}
JH~Van~Schuppen.
\newblock Filtering, prediction and smoothing for counting process
  observations, a martingale approach.
\newblock {\em SIAM Journal on Applied Mathematics}, 32(3):552--570, 1977.

\end{thebibliography}
